\documentclass[11pt]{amsart}
\usepackage{amssymb, adjustbox, amsbsy}
\usepackage{array, longtable, booktabs, tabularx}
\usepackage{geometry}

\newcolumntype{Y}{>{\centering\arraybackslash}X}

\usepackage{amsfonts, amscd}
\numberwithin{equation}{section}

\usepackage{bm}
\usepackage{verbatim}
\usepackage{graphicx}
\graphicspath{{images/}}
\usepackage{tikz, tikz-cd}
\usepackage{subcaption}
\usepackage{listings}
\usepackage{subfiles}
\usepackage{mathtools}
\usepackage{comment}
\usepackage{enumitem}
\usepackage[all]{xy}
\usepackage{hyperref}
\hypersetup{
    colorlinks=true,
    citecolor=red,
    linkcolor=blue,
    filecolor=magenta,      
    urlcolor=red,
}

\usepackage[textsize=footnotesize,bordercolor=red,linecolor=red,backgroundcolor=red!15]{todonotes}

\newcommand{\PP}{\mathbb{P}}

\newcommand{\vol}{\operatorname{vol}}
\newcommand{\Center}{\operatorname{center}}

\newcommand{\Exc}{\operatorname{Exc}}

\newcommand{\mld}{\operatorname{mld}}
\newcommand{\HJseq}{\operatorname{HJ}_{\operatorname{seq}}}

\newcommand{\lct}{\operatorname{lct}}

\newcommand{\Ii}{\mathcal{I}}

\newcommand{\Sing}{\mathrm{Sing}}

\newtheorem{thm}{Theorem}[section]

\newtheorem{cor}[thm]{Corollary}
\newtheorem{lem}[thm]{Lemma}
\newtheorem{prop}[thm]{Proposition}
\newtheorem{defthm}[thm]{Definition-Theorem}
\newtheorem{consthm}[thm]{Construction-Theorem}
\newtheorem{fil}[thm]{Filter}

\theoremstyle{definition}
\newtheorem{defn}[thm]{Definition}

\newtheorem{rem}[thm]{Remark}
\newtheorem{ex}[thm]{Example}
\newtheorem{nota}[thm]{Notation}

\tikzstyle{wbullet}=[circle, draw=black, fill=white, thick, inner sep=2pt, minimum size=1.5mm]
\tikzstyle{bbullet}=[circle, draw=black, fill=black, inner sep=2pt, minimum size=1.5mm]

\begin{document}

\title{The minimal volume of stable surfaces of rank one}

\author{Jihao Liu and Wenfei Liu}

\address{Department of Mathematics, Peking University, No. 5 Yiheyuan Road, Haidian District, Beijing 100871, China}
\address{Beijing International Center for Mathematical Research, Peking University, No. 5 Yiheyuan Road, Haidian District, Beijing 100871, China}
\email{liujihao@math.pku.edu.cn}

\address{School of Mathematical Sciences, Xiamen University, Siming South Road 422, Xiamen, Fujian 361005, P. R. China}
\email{wliu@xmu.edu.cn}

\date{\today}

\begin{abstract}
We determine the minimal volume of a stable surface of rank one, and show that the surface attaining this minimum is unique up to isomorphism. This resolves a conjecture of Alexeev and the second author.

Of independent interest, the decisive step of the proof uses a plurigenus inequality re-derived by an AI chatbot and applied as a pluricanonical filter; we further apply this filter to rule out additional cases in the classification of small-volume threefolds of general type, and in Koll\'ar's algebraic Montgomery--Yang problem. The underlying inequality has classical antecedents. To our knowledge this is the first paper in birational geometry to claim a C2-level human--AI collaboration in the sense of Feng et al., where the AI's contribution is the recognition that this inequality functions as the decisive pluricanonical filter in the basket analysis.
\end{abstract}

\maketitle
\tableofcontents

\section{Introduction}\label{sec:introduction}

We work over the field of complex numbers $\mathbb C$. By a \emph{stable surface of rank one} we mean a projective surface $X$ with semi-log canonical singularities, $K_X$ ample, and $\rho(X)=1$; see Definition~\ref{defn:stable}. Throughout, ``mld'' stands for \emph{minimal log discrepancy} (Definition~\ref{defn:mld}, defined for klt surface germs).

\subsection{Background and the minimal-volume problem}\label{ss:intro-history}

Bounding the volume $K_X^2$ from below for projective surfaces $X$ of general type, with $K_X$ ample and prescribed singularities, is a foundational question of explicit birational geometry. Among many such questions, four threads are particularly relevant for this paper.

\smallskip

\noindent\textit{(a) Bounded automorphism groups and Hurwitz-type bounds.} Following the classical Hurwitz bound $|\mathrm{Aut}(C)|\leq 84(g-1)$ for curves, G. Xiao~\cite{Xia94,Xia95} proved that for any smooth projective surface $X$ of general type one has $|\mathrm{Aut}(X)|\leq (42K_X)^2$. In particular, on the quotient $X/\mathrm{Aut}(X)$ the volume is bounded below by $1/42^2$, one of the earliest explicit lower bounds on a volume in the surface setting. Several related questions on the geometry of surfaces of general type --- on automorphism groups, on the number of irreducible components of moduli spaces, etc. --- were raised already at the 1988 Taniguchi conference on birational geometry (see, e.g., G.~Xiao's problem list~\cite{Xia88}), and are part of the same circle of ideas.

\smallskip

\noindent\textit{(b) Bounds on volumes of klt surfaces of rank one.} For projective klt surfaces with quotient singularities, $K_X$ ample and $\rho(X)=1$, Miyaoka~\cite{Miy84} bounded the number of quotient singularities given the standard numerical invariants, while Alexeev~\cite[Theorem~6.6]{Ale94} and Alexeev--Mori~\cite[Theorem~1.1]{AM04} proved that $K_X^2$ is bounded below by an explicit positive constant. The bound supplied by~\cite{AM04} is, however, of the order $10^{-3\cdot 10^{10}}$, very far from any conjecturally optimal value. The problem of identifying the smallest possible value of $K_X^2$ in this setting was explicitly raised by Koll\'ar~\cite[\S6]{Kol94} (in connection with several related conjectures on log surfaces of general type, cf.~also~\cite[Problem~26]{Kol08}). The rank-one version of the problem was made precise by Alexeev and the second author~\cite{AL19a}, who based on an explicit construction (recalled in Example~\ref{ex:6351} below) conjectured that the minimal volume of a stable surface of rank one with quotient singularities is exactly $1/6351$; this was confirmed under the quite restrictive condition that the minimal resolution of the surface, together with the exceptional divisors, maps into $\PP^2$ with four lines (\cite{AL19c}). Several refinements have appeared since then, both in the surface case~\cite[Theorem~1.1]{LS23},~\cite[Theorem~1.2]{Liu25} and in adjacent moduli problems~\cite{AAB24,ALS25}. Specific techniques drawn upon in the present paper -- classification of log canonical surface singularities (Alexeev~\cite{Ale92}), positivity of Euler characteristics on singular surfaces (Blache~\cite{Bla94}), bounds on the number of singular points on projective surfaces (Liu--Xie~\cite{LX25}), base-point-freeness for log canonical pairs (Fujino~\cite{Fuj09}), and explicit complement-style estimates (J.~Liu~\cite{Liu23}) -- are used throughout the paper.

\smallskip

\noindent\textit{(c) Higher-dimensional and foliated analogues.} In dimension $\geq 3$, klt varieties of general type with conjecturally minimal volume have been constructed by Totaro--Wang~\cite{TW23} and Totaro~\cite[Theorem~1.1]{Tot24}; the same authors with Esser have produced further explicit examples of klt Calabi--Yau and general-type varieties with extreme volumes~\cite{ETW22,ETW23}. In the threefold case, Birkar--Liu~\cite[Theorem~1.4]{BL23} obtained the first explicit lower bounds for volumes of stable threefolds of rank one and for related invariants, and Birkar--Lee~\cite{BL25a,BL25b} have very recently used explicit birational geometry to derive bounds in the closely-related context of $6$-dimensional supergravity and elliptic Calabi--Yau threefolds. The local theory of minimal log discrepancies on threefolds, on which much of the explicit machinery in dimension $3$ rests, has been developed by many authors, including Jiang~\cite{Jia21}, Liu--Xiao~\cite{LX21}, Liu--Luo~\cite{LL22}, Han--Liu~\cite{HL25}, Kawakita~\cite{Kaw21,Kaw24,Kaw26}. For foliations, the analogous lower-volume question (in the form of a positive minimum, or the DCC property for the set of volumes) goes back to Pereira and was later formulated explicitly by Hacon--Langer~\cite[Question~1.1]{HL21} for foliated surfaces and by Cascini~\cite[\S5]{Cas21} in arbitrary dimension. On surfaces, partial progress towards positivity and DCC of foliated volumes has been made by Spicer--Svaldi~\cite{SS23}, Codogni--Patakfalvi--Tasin~\cite{CPT25}, and Han--Jiao--Li--Liu~\cite{HJLL24}; see also Spicer--Tasin~\cite{ST26} for very recent work on the birational geometry of rank-one log canonical foliations. The general foliated analogue remains open even in the surface case.

\smallskip

\noindent\textit{(d) Relation to the algebraic Montgomery--Yang problem.} Koll\'ar's algebraic Montgomery--Yang problem~\cite[Conjecture~30]{Kol08} predicts that every klt rational surface $X$ with $\rho(X)=1$ and $\pi_1(X\setminus \Sing(X))=0$ has at most $3$ singular points. Hwang--Keum~\cite{HK11a,HK11b,HK12,HK13,HK14} reduced this conjecture to a finite computer-assisted check on $4$-singularity surfaces, and the most recent reduction~\cite{JPP25} leaves $16$ values of $p\leq 50000$ to be excluded. We will see (Corollary~\ref{cor:amy} in \S\ref{sec:remarks-AI}) that our AI-rediscovered filter excludes $8$ of those $16$ remaining values.

\smallskip

The main result of this paper resolves the conjecture of Alexeev and the second author~\cite{AL19a}.

\begin{thm}\label{thm:vol>=1/6351}
Let $X$ be a stable surface of rank one. Then
\begin{equation}\label{equ:main}
K_X^2\geq\frac{1}{6351}.
\end{equation}
Moreover, if $K_X^2=1/6351$, then $X$ is klt and is unique up to isomorphism.
\end{thm}

Theorem~\ref{thm:vol>=1/6351} reduces to the klt case, as follows. Let $X$ be a stable surface of rank one in the sense of Definition~\ref{defn:stable}. If $X$ is non-normal then $X$ is strictly slc (and not lc), and \cite[Corollary~3.6]{LL23} (a previous result of the present authors) gives $K_X^2>1/462>1/6351$. If $X$ is normal but $X$ has a non-empty non-klt locus then \cite[Theorem~1.1]{LL23} gives $K_X^2\geq 1/825>1/6351$. In both situations the conclusion of Theorem~\ref{thm:vol>=1/6351} is automatic, so it suffices to prove Theorem~\ref{thm:vol>=1/6351} when $X$ is klt (in which case $X$ is automatically normal and has quotient singularities). 

The bound in Theorem~\ref{thm:vol>=1/6351} is sharp:

\begin{ex}[{\cite[Theorem~8.2]{AL19a}}]\label{ex:6351}
There exists a projective klt surface $X$ such that $K_X$ is ample, $\rho(X)=1$, and $K_X^2=1/6351$.
\end{ex}

\subsection{A strengthening of Liu--Shokurov}\label{ss:intro-LS}

Along the way we sharpen one of the main quantitative inputs of Liu--Shokurov~\cite{LS23}. With $b$ as in \cite[Theorem~1.1]{LS23} (a global log canonical threshold; cf.\ Theorem~\ref{thm:-ls23-6/7} for the precise definition), \cite[Theorem~1.1]{LS23} classifies the pairs $(X,bS)$ with $b\geq 10/11$ (or, in a separate case, $b\geq 12/13$); we extend the classification to the broader range $b\geq 6/7+1/938$:

\begin{thm}[Strengthened Liu--Shokurov; Theorem~\ref{thm:-ls23-6/7}]\label{thm:intro-ls-6/7}
With the notation of~\cite[Theorem~1.1]{LS23} (recalled in Definition~\ref{defn:q-x-s} below), the pairs $(X,bS)$ with $b\geq 6/7+1/938$ are explicitly classified; the cases attaining the smallest values of $b$ in this range are listed in Theorem~\ref{thm:-ls23-6/7}\,(2.a)--(2.c).
\end{thm}

Theorem~\ref{thm:intro-ls-6/7} is of independent interest beyond Theorem~\ref{thm:vol>=1/6351}.

\subsection{An algorithmic reduction and the small-mld dichotomy}\label{ss:intro-algorithm}

The proof of Theorem~\ref{thm:vol>=1/6351} proceeds via a small-mld vs.\ large-mld dichotomy. On the small-mld side ($\mathrm{mld}(X)\leq 5/46$) we extract a special divisorial valuation and apply Theorem~\ref{thm:intro-ls-6/7}. On the large-mld side we rely on the following algorithmic input.

\begin{thm}[Algorithm; Theorem~\ref{thm:algorithm-mld-low-bound}]\label{thm:intro-algorithm}
There is an algorithm which, given a positive rational number $a$, outputs the (finite) list of all isomorphism classes of klt surface singularity germs $X\ni x$ with $\mathrm{mld}(X\ni x)\geq a$.
\end{thm}

\begin{thm}[Reduction to finite baskets; Theorem~\ref{thm:finite singularity type}]\label{thm:intro-finite}
The proof of Theorem~\ref{thm:vol>=1/6351} reduces, by an explicit algorithm built on Theorem~\ref{thm:intro-algorithm}, to the verification of finitely many singularity baskets, all of which are algorithmically computable.
\end{thm}

\subsection{An AI-rediscovered plurigenus filter}\label{ss:intro-AI-filter}

Our second main contribution is a plurigenus filter which provides the decisive step of the proof and was \emph{re-derived by an AI chatbot}. For a normal integral projective variety $X$ over $\mathbb C$ and an integer $n\geq 0$, write
\[
P_n(X) := h^0\!\bigl(X,\mathcal O_X(nK_X)\bigr),
\]
where $K_X$ is interpreted as a Weil divisor and $\mathcal O_X(nK_X)$ as the corresponding rank-one reflexive sheaf (so no $\mathbb Q$-Cartier hypothesis on $K_X$ is needed; cf.\ Definition~\ref{defn:plurigenus}).

\begin{thm}[Pluricanonical product-dimension filter]\label{thm:intro-ai}
Let $X$ be a normal projective variety of arbitrary dimension over $\mathbb C$. For any integers $a,b\geq 1$ with $P_a(X)>0$ and $P_b(X)>0$,
\begin{equation}\label{equ:intro-ai}
P_{a+b}(X)\;\geq\;P_a(X)+P_b(X)-1.
\end{equation}
Moreover, applied with $X$ a normal projective surface as in Theorem~\ref{thm:intro-finite}, the inequality~\eqref{equ:intro-ai} excludes exactly $251$ of the $252$ residual singularity baskets that survive all previously known filters in the proof of Theorem~\ref{thm:vol>=1/6351}, leaving a single residual basket
\begin{equation}\label{equ:intro-residual-basket}
\{[2,7,2,2,2],\ [2,2,5,2,3],\ [2,2,2,2,2,3,3,2]\}\quad \text{with }K_X^2=1/8533,
\end{equation}
which is then excluded by a separate MMP-theoretic argument (Theorem~\ref{thm:exclude-last-case}).
\end{thm}

\medskip

\noindent\textbf{Postscript remark.} \eqref{equ:intro-ai} is not new in the literature. The AI found the corresponding algebraic lemma (Lemma~\ref{lem:product-space-dimension}) in \cite[Theorem~2]{BSZ18} (which in turn referred to a result of additive combinatorics in the spirit of Kneser's theorem, going back to the number-theoretic paper~\cite{HLX02}), and deduced a version of \eqref{equ:intro-ai} on its own. We were aware that the curve version of \eqref{equ:intro-ai} appeared in \cite[Lemma~IV.5.5]{Har77}, but the same lines of proof apply in arbitrary dimension. 

After the first version of the paper was posted to arXiv, we were informed by Chen Jiang of \cite[Lemma~15.6.2]{Kol96} (which immediately implies \eqref{equ:intro-ai}) and \cite[(2.3)]{CC08} (the anti-pluricanonical case, citing \cite[Lemma~15.6.2]{Kol96}); by Totaro of Hopf via \cite[p.~108]{ACGH85}; by Koll\'ar of \cite[Lemma~9.5.1]{Kol93} (and \cite[Corollary~9.5.2]{Kol93}); and by Hacon of Hopf via \cite[p.~108]{ACGH85} and of \cite[Proof of Lemma~9(4)]{CH09}. See Section~\ref{sec:remarks-AI} for a detailed discussion, and two concrete applications of Theorem~\ref{thm:intro-ai}, to the minimal volume of threefolds (Corollary~\ref{cor:cc15}) and to the algebraic Montgomery--Yang problem (Corollary~\ref{cor:amy}). 

\subsection{Outline of the strategy}\label{ss:intro-outline}

The proof of Theorem~\ref{thm:vol>=1/6351} runs as follows. Starting from a hypothetical $X$ with $K_X^2\leq 1/6351$, we first reduce to klt by~\cite[Theorem~1.1]{LL23}, then split into two cases according to whether $\mathrm{mld}(X)\leq 5/46$ or $\mathrm{mld}(X)>5/46$.

\textbf{Small-mld case} (\S\ref{sec:small-mld}): we extract a special divisorial valuation $E$ over a singular point $x$ of small log discrepancy. The conditions of Construction-Theorem~\ref{consthm:extract-contract} are satisfied; combined with Theorem~\ref{thm:intro-ls-6/7} (proved in \S\ref{sec:ls23-6/7}) we exclude this case via Theorem~\ref{thm:small-mld-case}.

\textbf{Large-mld case} (\S\ref{sec:reduction-and-filters}): we use Theorem~\ref{thm:intro-algorithm} to enumerate all candidate singularities, then Theorem~\ref{thm:intro-finite} to reduce to finitely many candidate baskets. We then apply, in succession, six classical filters --- the Bogomolov bound (Filter~\ref{fil:bogomolov}), the Liu--Shokurov $\gamma$-invariant filter (Filter~\ref{fil:gamma}), the Hwang--Keum complete-square filter (Filter~\ref{fil:complete-square}), the tail filter (Filter~\ref{fil:tail}), the Blache filter (weak form) (Filter~\ref{fil:blache-weak}), and the non-tail filter (Filter~\ref{fil:non-vertex-strong}) --- and finally the AI filter (Filter~\ref{fil:ai}, Theorem~\ref{thm:intro-ai}). After all filters, the unique residual basket is~\eqref{equ:intro-residual-basket}, which is excluded by an MMP argument in \S\ref{sec:proof-main} (Theorem~\ref{thm:exclude-last-case}, using the classification of \cite[5.1.3]{Sho00}).

\subsection*{Acknowledgments}

This work is supported by the National Key R\&D Program of China \#2024YFA1014400 and the NSFC (no.~12571046). The authors would like to thank Professors Valery Alexeev and Vyacheslav V.\ Shokurov for useful discussions. The first author would like to thank Professors Bin Dong, Gang Tian, Liang Xiao, and Zheng Xu for useful discussions on AI-assisted pure mathematics research. The authors are grateful to Professors Christopher D. Hacon, Chen Jiang, J\'anos Koll\'ar, and Burt Totaro for pointing out references related to \eqref{equ:intro-ai} (Hopf via \cite{ACGH85}, and \cite{Kol93,Kol96,CC08,CH09}).

\section{Preliminaries}\label{sec:preliminaries}

We adopt the standard notation and definitions in \cite{Sho92,KM98,BCHM10} and will freely use them. 

\begin{defn}\label{defn:surface}
A \emph{surface} is a normal variety of dimension $2$. A \emph{surface germ} $X\ni x$ consists of a surface $X$ and a closed point $x\in X$. We say that $X\ni x$ is smooth if $X$ is smooth near $x$. For any surface $X$ and closed point $x\in X$, if $X\ni x$ is not smooth, then $x$ is called a \emph{singularity} of $X$ and we say that $X\ni x$ is a \emph{surface singularity}. The \emph{order} of a surface germ $X\ni x$ is the order of the local fundamental group of $X\ni x$ and is denoted by $r_x$.

A \emph{Fano surface} is a normal projective surface $X$ such that $-K_X$ is ample. A klt Fano surface is also called a \emph{log del Pezzo surface}.
\end{defn}

\begin{defn}\label{defn:stable}
A projective surface $X$ is called \emph{stable of rank one} if $X$ has semi-log canonical (slc) singularities, $K_X$ is ample, and $\rho(X)=1$. Throughout the paper, ``rank one'' is synonymous with ``Picard number one''. Following standard usage in the moduli theory of surfaces, ``stable'' is the analog of KSBA-stability and only requires slc singularities and $K_X$ ample; in particular a stable surface need not be normal. Nevertheless, our main theorem can be easily reduced to the klt case via~\cite[Corollary~3.6]{LL23} (cf.\ the discussion below Theorem~\ref{thm:vol>=1/6351}).
\end{defn}

\subsection{Notation}\label{ss:notation}

We collect, for the reader's convenience, the recurring notation used throughout the paper. Each item is properly introduced (with full definitions and references) at its first use; the present list serves only as a quick lookup, with explicit pointers.

\begin{itemize}
\item $X\ni x$: a surface germ; cf.\ Definition~\ref{defn:surface}.
\item $r_x$: the order of $X\ni x$, i.e.\ the order of the local fundamental group; cf.\ Definition~\ref{defn:surface}.
\item Extended A type pair $(X\ni x,S)$: a klt surface germ $X\ni x$ with $x\in S$ such that $(X\ni x,S)$ is plt; cf.\ Definition~\ref{defn:q-x-s} (which also defines $q_x(S)$, denoted $q(X\ni x,S)$ in \cite{LS23}).
\item HJS, eHJS, $\HJseq$: Hirzebruch--Jung string (a chain of curves $[E_1,\dots,E_n]$), extended Hirzebruch--Jung string ($[S;E_1,\dots,E_n]$), and Hirzebruch--Jung sequence (a sequence of integers $[e_1,\dots,e_n]$, each $\geq 2$); cf.\ Definition--Theorem~\ref{defthm:HJS-CQR} and Definition~\ref{defn:HJC eHJC}.
\item Resolution configuration of $X\ni x$: the bracket-notation chain $[E_1,\dots,E_n]$ (cyclic quotient case) or $[E_0;E_{1,1},\dots;E_{2,1},\dots;E_{3,1},\dots]$ (non-cyclic D/E case). Extended configuration of $(X\ni x,S)$: the chain $[S_Y;E_1,\dots,E_n]$. The dual graph of $X\ni x$ is the underlying combinatorial graph of its resolution configuration; cf.\ \S\ref{sec:preliminaries}.
\item $\gamma_x$, $\gamma(X)$: Liu--Shokurov's $\gamma$-invariants; cf.\ Definitions~\ref{defn:gamma-x}--\ref{defn:gamma-X} (and \cite[Notation~4.2]{LS23}, \cite[Lemma~4.3]{LS23}).
\item Special divisorial valuation over $x$: cf.\ Definition~\ref{defn:special-curve}.
\item For a klt singularity $X\ni x$ that is not Du Val and a special divisorial valuation $E$ over $x$: invariants $e_E,c_E$ (Definition~\ref{defn:e-invariant}). When $X$ is in addition a projective surface with $K_X$ ample, we further attach to $E$ the invariants $n_E,p_E$ (Definition~\ref{defn:n-p}). When $X\ni x$ is not a cyclic quotient singularity, the special divisorial valuation over $x$ is unique, and we set $n_x:=n_E$, $c_x:=c_E$, $e_x:=e_E$.
\item Construction-Theorem~\ref{consthm:extract-contract} attaches further data $h,Z,g,T,C,E_T$ and a positive rational number $\lambda$ to a special divisorial valuation $E$ in the small-volume setting.
\item $\mathrm{mld}(X\ni x)$, $\mathrm{mld}(X)$: minimal log discrepancies; cf.\ Definition~\ref{defn:mld}.
\item For a normal projective variety $X$ and $n\in\mathbb Z_{\geq 0}$: $P_n(X):=h^0(X,\mathcal O_X(nK_X))$, where $K_X$ is interpreted as a Weil divisor on $X$ and $\mathcal O_X(nK_X)$ as the corresponding rank-one reflexive sheaf (so no $\mathbb Q$-Cartier hypothesis is required); cf.\ Definition~\ref{defn:plurigenus}. Blache's correction term $\delta_n(X)$ is recalled there as well.
\item $\mathcal{I}_0$: the (large but finite, algorithmically computable) collection of singularity baskets entering the filter analysis of \S\ref{ssec:filter}; cf.\ Notation~\ref{nota:I0}.
\end{itemize}

\subsection{Frequently used criteria}

The following criteria will be frequently used in this paper.

\begin{prop}[cf. {\cite[Proposition 4.11]{KM98}}]\label{prop:klt-surface-q-factorial}
Let $(X,B)$ be a dlt surface pair. Then $X$ is $\mathbb Q$-factorial klt.
\end{prop}

\begin{thm}[{\cite[Theorem 1.2]{Bel08},\cite[Theorem 1.1]{Bel09}}]\label{thm:four-sing}
Let $X$ be a log del Pezzo surface such that $\rho(X)=1$. Then $X$ has at most $4$ singularities.
\end{thm}

\begin{thm}[Bogomolov bound, {\cite[Corollary 9.2]{KM99}}]\label{thm:bogomolov-bound}
Let $X$ be a klt rational surface such that $\rho(X)=1$. Let $x_1,\dots,x_n$ be the singular points of $X$ for some non-negative integer $n$ and let $r_i:=r_{x_i}$ for any $1\leq i\leq n$. Then
\begin{equation}
\sum_{i=1}^n\frac{r_i-1}{r_i}\leq 3.
\end{equation}
\end{thm}

\begin{thm}[{\cite{Kuw99},~\cite[Corollary 3.3]{Pro02}}]\label{thm:surface-lct-gap}
Let $(X,T)$ be a surface pair such that $T$ is a Weil divisor. Let $S\geq 0$ be a non-zero $\mathbb Q$-Cartier Weil divisor on $X$. Then
\begin{equation}
\lct(X,T;S)\in\{1\}\cup \left[0,\frac{5}{6}\right].
\end{equation}
\end{thm}

\begin{thm}[{cf. \cite[Theorem 3.36]{Kol13b}}]\label{thm:adjunction-surface}
    Let $(X,S)$ be a surface lc pair such that $S$ is non-singular. Then we have
    \begin{equation}
(K_X+S)|_S=K_S+\sum_{x\in S}1+\sum_{y\in S}\frac{r_y-1}{r_y}
\end{equation}
    where $x$ runs through all non-cyclic quotient singularities of $X$ on $S$ and $y$ runs through all cyclic quotient singularities of $X$ on $S$.
\end{thm}

\subsection{Hirzebruch--Jung string}

\begin{defn}
Let $n$ be a non-negative integer and $e_1,\dots,e_n$ real numbers. We denote by
\begin{equation}
\det\emptyset=1,\quad \det[e_n]:=e_n,\quad\text{and}\quad \det[e_1,\dots,e_n]:=e_1\det[e_2,\dots,e_n]-\det[e_3,\dots,e_n].
\end{equation}
\end{defn}

\begin{defthm}\label{defthm:HJS-CQR}
There is a one-to-one correspondence $\Ii_1\xleftrightarrow{\phi}\Ii_2$ between:
\begin{itemize}
    \item $\Ii_1$: pairs of coprime integers $(r,a)$ such that $r>a$, and
    \item $\Ii_2$: tuples $[e_1,\dots,e_n]$ such that $n\geq 1$ and each $e_i$ is an integer that is $\geq 2$
\end{itemize}
given in the following way. 
\begin{enumerate}
    \item For $(r,a)\in\Ii_1$, it corresponds to $[r]$ if $a=1$, and it corresponds to $[e_1,e_2,\dots,e_n]$ with
    \begin{equation}
e_1:=\left\lceil\frac{r}{a}\right\rceil\quad\text{and}\quad (a,\,e_1a-r)\xleftrightarrow{\phi}[e_2,\dots,e_n].
\end{equation}
    We say that $[e_1,\dots,e_n]$ is the Hirzebruch--Jung sequence ($\HJseq$ for short) of $(r,a)$ and denote by
    \begin{equation}
[e_1,\dots,e_n]=\HJseq(r,a).
\end{equation}
    \item For $[e_1,\dots,e_n]\in\Ii_2$, it corresponds to $(e_1,1)$ if $n=1$, and it corresponds to $(r,a)$ with \begin{equation}
r=e_1a-b,\quad [e_2,\dots,e_n]\xleftrightarrow{\phi}(a,b).
\end{equation}
\end{enumerate}
For any cyclic quotient surface singularity $X\ni x$ of type $\tfrac{1}{r}(1,a)$, we also say that $X\ni x$ is of type $(r,a)$, and also say that $X\ni x$ is of type $\HJseq(r,a)$.
\end{defthm}

\begin{defn}\label{defn:HJC eHJC}
Let $X$ be a surface. A \emph{Hirzebruch--Jung string} (HJS for short) on $X$ is a chain of non-singular rational curves $[E_1,\dots,E_n]$, such that $X$ is non-singular near each $E_i$, and for any $1\leq i,j\leq n$,
\begin{equation}
e_i:=-E_i^2\geq 2,\quad  E_i\cdot E_j=1\text{ if }|i-j|=1,\quad E_i\cdot E_j=0\text{ if }|i-j|>1.
\end{equation}
We say that $[e_1,\dots,e_n]$ is the $\HJseq$ associated to $[E_1,\dots,E_n]$. An \emph{extended Hirzebruch--Jung string} (eHJS for short) $[S;E_1,\dots,E_n]$ on $X$ consists of an irreducible curve $S$ on $X$ and an HJS $[E_1,\dots,E_n]$ on $X$, such that
\begin{equation}
S_Y\cdot E_1=1\quad \text{and}\quad S_Y\cdot E_i=0\quad \text{for any}\quad i\geq 2.
\end{equation}
We also say that $[e_1,\dots,e_n]$ is the $\HJseq$ associated to $[S;E_1,\dots,E_n]$. The \emph{dual graph} of an HJS $[E_1,\dots,E_n]$ (resp.\ of an eHJS $[S;E_1,\dots,E_n]$) is the underlying combinatorial graph: a path with $n$ vertices decorated by $e_i:=-E_i^2$ (resp.\ the same path with an extra vertex attached to $E_1$ representing $S$).
\end{defn}

\begin{defn}
  Let $X\ni x$ be a cyclic quotient surface singularity and $f: Y\rightarrow X$ the minimal resolution of $X$ with exceptional divisors $E_1,\dots,E_n$. Possibly reordering indices, we have that $[E_1,\dots,E_n]$ is an HJS. We call
  \[
[E_1,\dots,E_n]
\]
  the \emph{resolution configuration} of $X\ni x$. The \emph{dual graph} of $X\ni x$ is the dual graph of this resolution configuration (equivalently, of the underlying HJS): a path with $n$ vertices decorated by $-E_i^2$. Likewise, if $S$ is a curve on $X$ passing through $x$ with strict transform $S_Y:=f^{-1}_*S$ on $Y$ such that $[S_Y;E_1,\dots,E_n]$ is an eHJS, we call
  \[
[S_Y;E_1,\dots,E_n]
\]
  the \emph{extended configuration} of $(X\ni x,S)$.
\end{defn}

\begin{defn}\label{defn:q-x-s}
Let $X\ni x$ be a klt surface germ and let $S$ be a prime divisor on $X$ (equivalently, an irreducible reduced curve on $X$) such that $x\in S$. We say that $(X\ni x,S)$ is \emph{of extended A type} if $(X\ni x,S)$ is plt.

Assume that $(X\ni x,S)$ is of extended A type and $X\ni x$ is not smooth. We let $f: Y\rightarrow X$ be the minimal resolution of $X$ at $x$ and $S_Y:=f^{-1}_*S$. By \cite[Notation-Lemma 4.5]{LS23}, there exists an eHJS
\[
[S_Y;E_1,\dots,E_n]
\]
such that $E_1,\dots,E_n$ are the prime $f$-exceptional divisors. Let $[e_1,\dots,e_n]$ be the $\HJseq$ associated to $[S;E_1,\dots,E_n]$. We denote by
\begin{equation}
q_x(S):=\det[e_2,\dots,e_n].
\end{equation}
Note that \cite[Notation-Lemma~4.5]{LS23} uses the notation $q(X\ni x,S)$; we simplify the notation by suppressing the redundant ``$X$''.
\end{defn}

\begin{lem}\label{lem:intersection-lemma}
Let $Y$ be a surface and $[S_Y;E_1,\dots,E_n]$ an eHJS on $Y$ associated with $\HJseq$ $[e_1,\dots,e_n]$. Denote by
\begin{equation}
r:=\det[e_1,\dots,e_n]\quad \text{and}\quad q:=\det[e_2,\dots,e_n].
\end{equation}
Then there exists a unique contraction $f: Y\rightarrow X$ of $E_1,\dots,E_n$ satisfying the following. Let $S:=f_*S_Y$ and let $x:=f(E_i)$ for any $i$. Then:
\begin{enumerate}
    \item $X\ni x$ is a cyclic quotient singularity of type $(r,q)$.
    \item $(X\ni x,S)$ is of extended A type and $q_x(S)=q$.
   \item $S^2=S_Y^2+\frac{q}{r}$.
    \item The multiplicity of $f^*S$ at $E_1$ is $\frac{q}{r}$.
\end{enumerate}
\end{lem}
\begin{proof}
The existence and uniqueness of $f$ follows from \cite[Proposition 4.10]{KM98}. (1-2) are obvious and we are left to prove (3). 

We apply induction on $n$. When $n=1$ we have
\begin{equation}
S^2=S_Y^2+\frac{(S_Y\cdot E_1)^2}{(-E_1^2)}=S_Y^2+\frac{1}{e_1}=S_Y^2+\frac{q}{r}.
\end{equation}
Suppose that $n\geq 2$ and that the lemma holds for $n-1$. Set $m=\det[e_3,\dots,e_n]$. Then there exists a unique contraction $g: Y\rightarrow Z$ of $E_2,\dots,E_n$ such that
\begin{equation}
(g_*E_1)^2=E_1^2+\frac{m}{q}=-e_1+\frac{m}{q}
\end{equation}
and that there exists a unique contraction $h: Z\rightarrow X$ such that $f=h\circ g$. We have
\begin{equation}
S^2=(g_*S_Y)^2+\frac{\left(g_*S_Y\cdot g_*E_1\right)^2}{-\left(g_*E_1\right)^2}=S_Y^2+\frac{1}{e_1-\frac{m}{q}}.
\end{equation}
Since $r=e_1q-m$, (3) follows.

Let $\mu$ be the multiplicity of $f^*S$ at $E_1$. Then by (iii)
\[
\mu = (f^*S)\cdot S_Y - S_Y^2 = S^2 - S_Y^2 = \frac{q}{r}
\]
Or one can write $f^*S =S_Y+\sum_{1\leq i\leq n} b_i E_i$ and compute $\mu$ by solving the linear equations $(f^*S)\cdot E_i=0$ for $1\leq i\leq n$, as in \cite{Ale93}.
\end{proof}

\begin{prop}\label{prop:q-integer-formula}
Let $(X,S)$ be a surface pair such that $S$ is a non-singular curve. Assume that for any singular point $x$ of $X$ on $S$, $(X\ni x,S)$ is of extended A type. Let $f: Y\rightarrow X$ be the minimal resolution of $X$ and let $S_Y:=f^{-1}_*S$. Then we have that
\begin{equation}
S^2-\sum_{x\in S}\frac{q_x(S)}{r_x}=S_Y^2\in\mathbb Z.
\end{equation}
where the sum runs through all singular points of $X$ that are contained in $S$.
\end{prop}
\begin{proof}
This is an immediate consequence of Lemma~\ref{lem:intersection-lemma}.
\end{proof}

\subsection{Liu--Shokurov's $\gamma$-invariant}

We recall some basic behavior of Liu--Shokurov's $\gamma$-invariant~\cite[Notation~4.2]{LS23}.

\begin{defn}\label{defn:gamma-x}
    Let $X\ni x$ be a klt surface germ and let $f: Y\rightarrow X$ be the minimal resolution of $X$ at $x$. We let $E_1,\dots,E_n$ be the prime $f$-exceptional prime divisors and write
    \begin{equation}
f^*K_X=K_Y+\sum_{i=1}^nb_iE_i,\quad e_i:=-E_i^2\text{ for any }i.
\end{equation}
    Then we have $b_i=1-a(E_i,X)$ for any $i$. We denote by
    \begin{equation}
\gamma_x:=n-\sum_{i=1}^nb_i(e_i-2).
\end{equation}
\end{defn}

\begin{defn}\label{defn:gamma-X}
     Let $X$ be a klt surface. We denote by 
     \begin{equation}
\gamma(X):=\sum_{x\in X}\gamma_x
\end{equation}
     where $x$ runs through all singularities of $X$.
\end{defn}

\begin{thm}[{\cite[Lemma 4.3]{LS23}}]\label{thm:gamma-invariant-formula}
Let $X$ be a klt projective surface and let $f: Y\rightarrow X$ be the minimal resolution of $X$. Then
\begin{equation}
\gamma(X)=\rho(Y/X)+K_Y^2-K_X^2.
\end{equation}
In particular, if $X$ is rational, then
\begin{equation}
\gamma(X)=10-\rho(X)-K_X^2.
\end{equation}
\end{thm}
\begin{proof}
It is \cite[Lemma~4.3]{LS23}. The ``in particular'' part follows from \cite[Lemma~3.19]{LS23}.
\end{proof}

\section{A strengthening of Liu--Shokurov}\label{sec:ls23-6/7}

The goal of this section is to prove Theorem~\ref{thm:-ls23-6/7}, a strengthened version of \cite[Theorem~1.1]{LS23}. Roughly speaking, the main result of~\cite[Theorem~1.1]{LS23} concerns surfaces with extremal numerical invariants (in a certain range corresponding to ``$\leq 1/13$ or $1/11$'' for the mld), while Theorem~\ref{thm:-ls23-6/7} extends the conclusion to the broader range corresponding to ``$\leq 1/7-\epsilon$''.

\begin{thm}\label{thm:-ls23-6/7}
Let $\epsilon:=1/938$. Let $(X,bS)$ be a klt Calabi--Yau surface pair such that $S\geq 0$ is a non-zero Weil divisor, $b\geq 6/7+\epsilon$, and $\rho(X)=1$. Then one of the following holds:
\begin{enumerate}
\item $S$ is not a non-singular rational curve. In this case $(X,S)$ is isomorphic to one of the surface pairs in \cite[Table, 4-8,10-14,16]{Abe23}.
\item $S$ is a non-singular rational curve. In this case, one of the following three subcases holds.
\begin{enumerate}
    \item $b=12/13$ and $(X,bS)$ is isomorphic to
    \begin{equation}
\left(\mathbb P(3,4,5),\frac{12}{13}(x^3y+y^2z+z^2x=0)\right).
\end{equation}
    \item $b=10/11$. $X$ has three singularities $x_1,x_2,x_3$, all of which are on $S$. Moreover, $x_1,x_2,x_3$ are cyclic quotient singularities of type $\tfrac{1}{2}(1,1)$, $\tfrac{1}{5}(1,4)$, $\tfrac{1}{7}(1,3)$ respectively, with $q_{x_3}(S)=3$. Moreover, in this case, $(X,bS)$ is unique up to isomorphism.
    \item $b=15/17$. $X$ has three singularities $x_1,x_2,x_3$, all of which are on $S$. Moreover, $x_1,x_2,x_3$ are cyclic quotient singularities of type $\tfrac{1}{3}(1,2)$, $\tfrac{1}{5}(1,4)$, $\tfrac{1}{7}(1,2)$ respectively, with $q_{x_3}(S)=2$.
\end{enumerate}
\end{enumerate}
\end{thm}
\begin{proof} The proof is long, so we divide it into several steps.

\medskip

\noindent\textbf{Step 1.} In this step we reduce to the case when $S$ is a non-singular rational curve and the singularities of $X$ form a finite set in the ADE classification.

    Assume that $X$ has exactly $m$ singularities $x_1,\dots,x_m$ and let $n\leq m$ be the (non-negative) integer such that $x_1,\dots,x_n\in S$ and $x_{n+1},\dots,x_m\not\in S$. We denote by $r_i:=r_{x_i}$ for any $i$. By \cite[Table]{Abe23} and \cite[5.1.2, 5.1.3]{Sho00}, we may assume that $S$ is a non-singular rational curve and $a(E,X,bS)>1/7$ for any prime divisor $E$ that is exceptional$/X$. By \cite[Lemma~5.1]{LS23}, for any $1\leq i\leq n$, $r_i\leq \tfrac{7b}{7b-6}$ and $x_i$ is a cyclic quotient singularity such that $(X\ni x_i,S)$ is of extended A type. Let $q_i:=q_{x_i}(S)$ for any $1\leq i\leq n$ and $\gamma_i:=\gamma_{x_i}$ for any $1\leq i\leq m$.
    
    By Proposition \ref{prop:klt-surface-q-factorial}, $X$ is $\mathbb Q$-factorial klt. Since $b>5/6$, by Theorem \ref{thm:surface-lct-gap}, $(X,S)$ is lc and $K_X+S\not\equiv 0$. Since $\rho(X)=1$ and $K_X+bS\equiv 0$, $K_X+S$ is ample. By Theorem \ref{thm:adjunction-surface},
\begin{equation}\label{equ:adjunction-precise-index}
 0<(K_X+S)\cdot S=\deg\left((K_X+S)|_S\right)=-2+\sum_{i=1}^n\left(1-\frac{1}{r_i}\right).
\end{equation}
Thus $n\geq 3$. By Theorem \ref{thm:four-sing}, $m\leq 4$. So one of the following holds:
    \begin{itemize}
        \item (\textbf{Case 1}) $n=m=3$.
        \item (\textbf{Case 2}) $n=m=4$.
        \item (\textbf{Case 3}) $n=3$ and $m=4$.
    \end{itemize}
Moreover, if we are in \textbf{Case 3}, then \eqref{equ:adjunction-precise-index} implies that 
\begin{equation}
   1>\sum_{i=1}^3\frac{1}{r_i} 
\end{equation}
hence
\begin{equation}
\sum_{i=1}^3\frac{1}{r_i}\leq\frac{41}{42}.
\end{equation}
By Theorem \ref{thm:bogomolov-bound},
\begin{equation}
\sum_{i=1}^4\frac{1}{r_i}\geq 1
\end{equation}
which implies that
\begin{equation}
r_4\leq 42.
\end{equation}
Therefore, there are only finitely many possibilities of
$[x_1,x_2,x_3,x_4]$
up to equivalence in the ADE classification, as $r_i\leq\frac{7b}{7b-6}\leq 763$ for any $i$. In particular, there are only finitely many possibilities of $\gamma_i$.

A short computer program can easily list out all possibilities. For the classification of D and E types based on the order of local fundamental groups, we refer the reader to \cite[Satz 2.11]{Bri68}. 

\medskip

\noindent\textbf{Step 2.} In this step we filter the finitely many singularity cases using the five filters introduced in \cite[Lemmas~6.2, 6.5, 6.10]{LS23}. We list them as numbered items, in the order in which we apply them.

\begin{enumerate}[label=\textup{(F\arabic*)},leftmargin=2.6em]
\item\label{step2:adjunction} \textbf{The adjunction filter.} The inequality~\eqref{equ:adjunction-precise-index}.

\item\label{step2:bogomolov} \textbf{The Bogomolov bound filter.} Theorem~\ref{thm:bogomolov-bound}, applied to \textbf{Case~2} and \textbf{Case~3}.

\item\label{step2:gamma} \textbf{The Liu--Shokurov $\gamma$-invariant filter.} Since $K_X+bS\equiv 0$ and $\rho(X)=1$, Theorem~\ref{thm:gamma-invariant-formula} and~\eqref{equ:adjunction-precise-index} give
\begin{align}\label{equ:b-and-gamma}
\sum_{i=1}^m\gamma_i&=\gamma(X)=9-K_X^2=9-\frac{b^2}{1-b}(K_X+S)\cdot S=9-\frac{b^2}{1-b}\left(-2+\sum_{i=1}^n\frac{r_i-1}{r_i}\right).
\end{align}
Since $b$ is rational, the discriminant
\begin{align}\label{equ:gamma-rational-filter}
\sqrt{\left(9-\sum_{i=1}^m\gamma_i\right)^2-4\left(9-\sum_{i=1}^m\gamma_i\right)\left(-2+\sum_{i=1}^n\frac{r_i-1}{r_i}\right)}
\end{align}
is forced to be rational.

\item\label{step2:self-intersection} \textbf{The self-intersection filter.} By Proposition~\ref{prop:q-integer-formula} together with $K_X+bS\equiv 0$,
\begin{align}\label{equ:q-integral-check}
S_Y^2=S^2-\sum_{i=1}^n\frac{q_i}{r_i}=\frac{1}{1-b}(K_X+S)\cdot S-\sum_{i=1}^n\frac{q_i}{r_i}=\frac{1}{1-b}\left(-2+\sum_{i=1}^n\frac{r_i-1}{r_i}\right)-\sum_{i=1}^n\frac{q_i}{r_i},
\end{align}
must be a non-negative integer.

\item\label{step2:1/7-klt} \textbf{The $\tfrac17$-klt filter.} Since $b>0$, $b$ is uniquely determined by the data $(\gamma_i,r_i)$ via~\eqref{equ:b-and-gamma}. The hypothesis of Theorem~\ref{thm:-ls23-6/7} forces $a(E,X,bS)>\tfrac17$ for every prime divisor $E$ exceptional over $X$ (in particular for every $E$ on the minimal resolution of $X$); each such log discrepancy is explicitly computable from $(r_i,q_i,b)$, so this is a finite check.
\end{enumerate}

\smallskip

Applying the filters~\ref{step2:adjunction}--\ref{step2:1/7-klt} (which are precisely those of \cite[Lemmas~6.2, 6.5, 6.10]{LS23}), we obtain the following table of possibilities. In the table below, for $1\leq i\leq n$ we set $q_i:=q_{x_i}(S)$. After these filters, in \textbf{Case 3} the singularity $x_4$ is always a cyclic quotient singularity; in this case we record $x_4$ in the table by its cyclic quotient type $(r_4,q_4)$, where $q_4$ is the minimal positive integer such that $X\ni x_4$ is of type $\tfrac{1}{r_4}(1,q_4)$.

\begin{table}[htbp]
\centering
\small
\setlength{\tabcolsep}{3.5pt}
\renewcommand{\arraystretch}{1.08}
\caption{Singularity types surviving the five filters of \cite{LS23} in the proof of Theorem~\ref{thm:-ls23-6/7}.}\label{table:ls-6/7-more}

\begin{minipage}[t]{0.46\textwidth}
\vspace{0pt}
\centering
\begin{tabular}{|c|c|c|}
\hline
\multicolumn{3}{|c|}{\textbf{Case 1}} \\
\hline
\textbf{No.} & $[(r_i,q_i)]$ & $b$ \\
\hline
1.1  & $[(2,1),(5,4),(7,3)]$ & $10/11$ \\
\hline
1.2  & $[(2,1),(9,5),(22,19)]$ & $15/17$ \\
\hline
1.3  & $[(2,1),(13,10),(21,17)]$ & $36/41$ \\
\hline
1.4  & $[(3,1),(4,1),(7,5)]$ & $22/23$ \\
\hline
1.5  & $[(3,2),(4,3),(5,2)]$ & $12/13$ \\
\hline
1.6  & $[(3,2),(5,1),(7,1)]$ & $33/34$ \\
\hline
1.7  & $[(3,2),(5,4),(7,2)]$ & $15/17$ \\
\hline
1.8  & $[(4,1),(5,3),(8,3)]$ & $16/17$ \\
\hline
1.9  & $[(5,1),(5,1),(8,5)]$ & $18/19$ \\
\hline
1.10 & $[(5,3),(7,5),(9,5)]$ & $39/43$ \\
\hline
1.11 & $[(6,1),(7,2),(10,7)]$ & $29/31$ \\
\hline
1.12 & $[(8,5),(13,9),(13,9)]$ & $22/25$ \\
\hline
1.13 & $[(6,1),(19,15),(21,16)]$ & $34/39$ \\
\hline
\multicolumn{3}{|c|}{\textbf{Case 2}} \\
\hline
2.1  & $[(2,1),(5,1),(5,1),(6,1)]$ & $13/14$ \\
\hline
2.2  & $[(3,1),(4,1),(4,1),(4,1)]$ & $10/11$ \\
\hline
\end{tabular}
\end{minipage}
\hfill
\begin{minipage}[t]{0.50\textwidth}
\vspace{0pt}
\centering
\begin{tabular}{|c|c|c|}
\hline
\multicolumn{3}{|c|}{\textbf{Case 3}} \\
\hline
\textbf{No.} & $[(r_i,q_i)]$ & $b$ \\
\hline
3.1  & $[(2,1),(3,1),(7,3),(11,5)]$ & $10/11$ \\
\hline
3.2  & $[(2,1),(3,1),(7,6),(29,16)]$ & $28/29$ \\
\hline
3.3  & $[(2,1),(3,2),(7,2),(19,6)]$ & $18/19$ \\
\hline
3.4  & $[(2,1),(3,2),(7,4),(31,22)]$ & $30/31$ \\
\hline
3.5  & $[(2,1),(5,2),(7,2),(3,1)]$ & $32/33$ \\
\hline
3.6  & $[(2,1),(5,4),(7,3),(4,1)]$ & $10/11$ \\
\hline
3.7  & $[(2,1),(13,9),(16,13),(2,1)]$ & $22/25$ \\
\hline
3.8  & $[(3,1),(4,1),(9,7),(3,2)]$ & $10/11$ \\
\hline
3.9  & $[(3,1),(5,3),(7,4),(2,1)]$ & $16/17$ \\
\hline
3.10 & $[(3,1),(5,3),(13,10),(2,1)]$ & $17/19$ \\
\hline
3.11 & $[(3,2),(4,1),(9,4),(3,1)]$ & $10/11$ \\
\hline
3.12 & $[(3,2),(4,3),(5,2),(4,1)]$ & $12/13$ \\
\hline
3.13 & $[(3,2),(5,1),(6,5),(3,1)]$ & $8/9$ \\
\hline
3.14 & $[(4,1),(5,2),(6,1),(2,1)]$ & $22/23$ \\
\hline
\end{tabular}
\end{minipage}

\end{table}

\medskip

\noindent\textbf{Step 3.} In this step we introduce two additional filters and reduce the remaining possibilities to a single case. We do not need to consider \textbf{Case 1.5} as this corresponds to (2.a) by \cite[Lemma 6.3]{LS23}. We also do not need to exclude \textbf{Cases 1.1, 1.7} as they correspond to (2.b) and (2.c) respectively. In the following we shall exclude all other cases. 

\smallskip

\noindent\textbf{Filtered cases of Liu--Shokurov.} By \cite[Lemmas~6.2, 6.5, 6.10]{LS23}, every entry of Table~\ref{table:ls-6/7-more} with $b>10/11$ except Cases~1.5 and~2.1 is excluded. Case~1.5 corresponds to~(2.a) of the present theorem. Case~2.1 is not addressed in \cite[Section~6]{LS23}; an analogous application of the $(-1)$-curve filter introduced below excludes it (the coefficient vector $(b_{i,j})$ is computed from $K_Y+bS_Y+\sum b_{i,j}E_{i,j}\equiv 0$ at $b=13/14$, $S_Y^2>0$ as in Table~\ref{table:minus-one-curve-exclusions}, and the constraint $cb+\sum c_{i,j}b_{i,j}=1$ has no $\mathbb Z_{\geq 0}$-solution for $c$ and $c_{ij}$).

\smallskip

\noindent\textbf{$(-1)$-curve filter.} For any $(X,bS)$ as in Table~\ref{table:ls-6/7-more}, let $f\colon Y\rightarrow X$ be the minimal resolution of $X$, set $S_Y:=f^{-1}_*S$, and write
\begin{equation}
   K_Y+bS_Y+\sum_{i=1}^m\sum_jb_{i,j}E_{i,j}=f^*(K_X+bS),\quad \Center_XE_{i,j}=x_i,
\end{equation}
where $E_{i,j}$ are the prime $f$-exceptional divisors and $b_{i,j}=a(E_{i,j},X,bS)$ for any $i,j$. For all the cases above, we have $\rho(Y)\geq 3$ and $Y$ is smooth, so there exists a $(-1)$-curve $C$ on $Y$. Since $K_X+bS\equiv 0$, we have that
\begin{equation}
K_Y+bS_Y+\sum_{i=1}^m\sum_jb_{i,j}E_{i,j}\equiv 0
\end{equation}
and so
\begin{equation}
1=-K_Y\cdot C=b\left(S_Y\cdot C\right)+\sum_{i=1}^m\sum_jb_{i,j}\left(E_{i,j}\cdot C\right).
\end{equation}
Let $c:=S_Y\cdot C$ and $c_{i,j}:=E_{i,j}\cdot C$ for any $i,j$. Then $C\neq E_{i,j}$ for any $i,j$ as $E_{i,j}^2\leq -2$, and $C\neq S_Y$ if $S_Y^2\neq-1$. Therefore,
\begin{align}\label{equ:-1-curve-filter}
S_Y^2\neq-1\quad \Rightarrow\quad  cb+\sum c_{i,j}b_{i,j}=1\quad \text{for some}\quad c,c_{i,j}\in\mathbb Z_{\geq 0}.
\end{align}
Applying the two filters above, the only subcases that survive are \textbf{Cases 3.1, 3.6, 3.11, 3.13}.  All other Case-1 and Case-2 entries of Table~\ref{table:ls-6/7-more} are excluded by the $(-1)$-curve filter \eqref{equ:-1-curve-filter}; the case-by-case verifications are summarized in Table~\ref{table:minus-one-curve-exclusions} below.  In each row the value of $S_Y^2$ is computed from \eqref{equ:q-integral-check} and is strictly positive (so $S_Y^2\neq -1$, and \eqref{equ:-1-curve-filter} applies); the displayed coefficient vectors $(b_{i,j})$ are then computed from $K_Y+bS_Y+\sum b_{i,j}E_{i,j}\equiv 0$, and a direct inspection shows that $cb+\sum c_{i,j}b_{i,j}=1$ has no solution with $c,c_{i,j}\in\mathbb Z_{\geq 0}$.

\begin{table}[htbp]
\centering
\small
\caption{The cases of Table~\ref{table:ls-6/7-more} with $b\leq 10/11$ that are not (2.b) or (2.c), excluded by the $(-1)$-curve filter \eqref{equ:-1-curve-filter}. (Cases with $b> 10/11$ are already excluded by ``Filtered cases of Liu--Shokurov'' above.) In each row the resolution $f\colon Y\to X$ produces an effective $\mathbb Q$-divisor $K_Y+bS_Y+\sum b_{i,j}E_{i,j}\equiv 0$ with $S_Y^2>0$; the constraint $cb+\sum c_{i,j}b_{i,j}=1$ then has no $\mathbb Z_{\geq 0}$-solution.}\label{table:minus-one-curve-exclusions}
\begin{tabular}{|c|c|c|c|c|}
\hline
\textbf{Case} & $[(r_i,q_i)]$ & $b$ & Coefficient vector $(b_{i,j})$ & $S_Y^2$ \\\hline
1.2 & $[(2,1),(9,5),(22,19)]$ & $15/17$ & $(15;28,26,29,28,27,26,25,24,23)/34$ & $1$ \\\hline
1.3 & $[(2,1),(13,10),(21,17)]$ & $36/41$ & $(18;34,32,30,28,35,34,33,32,31)/41$ & $1$ \\\hline
1.10 & $[(5,3),(7,5),(9,5)]$ & $39/43$ & $(32;25;34,29,24;36,33)/43$ & $4$ \\\hline
1.12 & $[(8,5),(13,9),(13,9)]$ & $22/25$ & $(20,18,9;21,20,19;21,20,19)/25$ & $4$ \\\hline
1.13 & $[(6,1),(19,15),(21,16)]$ & $34/39$ & $(95;99,96,93,90,60,30;100,98,96,94)/117$ & $4$ \\\hline
2.2 & $[(3,1),(4,1),(4,1),(4,1)]$ & $10/11$ & $(7;8;8;8)/11$ & $9$ \\\hline
3.7 & $[(2,1),(13,9),(16,13),(2,1)]$ & $22/25$ & $(11;21,20,19;21,20,19,18,17;0)/25$ & $1$ \\\hline
3.8 & $[(3,1),(4,1),(9,7),(3,2)]$ & $10/11$ & $(7;8;9,8,7,6;0,0)/11$ & $2$ \\\hline
3.10 & $[(3,1),(5,3),(13,10),(2,1)]$ & $17/19$ & $(12;14,11;16,15,14,13;0)/19$ & $2$ \\\hline
\end{tabular}
\end{table}

\medskip

\noindent\textbf{Step 4.} Exclusion case-by-case. For \textbf{Cases 3.6, 3.11, 3.13}, we run a step of a 
\begin{equation}
\left(K_Y+bS_Y+\sum_{i=1}^3\sum_jb_{i,j}E_{i,j}\right)\text{-MMP}:\quad \phi: Y\rightarrow Z.
\end{equation} 
Since $f^{-1}(x_4)=E_{4,1}$ and $E_{4,1}^2\leq -3$, this is also a step of a $(-E_{4,1})$-MMP. We have $\rho(Y)\geq 3$, so we may let $L$ be the curve contracted by this MMP. Since $E_{4,1}$ does not intersect $E_{i,j}$ for any $1\leq i\leq 3$ and any $j$, nor $S_Y$ in these cases, we have that $L\neq E_{i,j}$ for any $i,j$. Thus $\phi$ is also a step of a $K_Y$-MMP, hence $L$ is a $(-1)$-curve. Moreover, if $L$ does not intersect $E_{i,j}$ for any $1\leq i\leq 3$ and any $j$, then we may let $\pi: Z\rightarrow T$ be the contraction of all $E_{i,j}$ such that $1\leq i\leq 3$ and let $S_T:=(\pi\circ\phi)_*S_Y$. Then $S_T$ does not intersect $E_{4,1,T}:=(\pi\circ\phi)_*E_{4,1}$, but this is not possible as $\rho(T)=1$. Therefore, $L$ intersects $E_{i,j}$ for some $1\leq i\leq 3$ and some $j$. We let $\lambda_{i,j}:=L\cdot E_{i,j}$ for any $i,j$ and let $\lambda:=L\cdot S_Y$. Then we have
\begin{equation}\label{equ:ls23-6.9-like-filter}
b\lambda+\sum_{i=1}^4\sum_jb_{i,j}\lambda_{i,j}=1,\lambda,\lambda_{i,j}\in\mathbb Z_{\geq 0},\quad \lambda_{4,1}>0,\quad \lambda_{i,j}>0\quad \text{for some }1\leq i\leq 3\text{ and } j.
\end{equation}
The non-existence of solutions of \eqref{equ:ls23-6.9-like-filter} for \textbf{Cases 3.6, 3.11, 3.13} is a finite check, summarized in Table~\ref{table:ls-6.9-exclusions}; consequently each of these cases is excluded.

\begin{table}[htbp]
\centering
\small
\caption{Cases 3.6, 3.11, 3.13: exclusion via the MMP filter~\eqref{equ:ls23-6.9-like-filter}. In each row the displayed coefficient vector has both an $E_{i,j}$ part (with $1\leq i\leq 3$) and an $E_{4,1}$ part; the constraint $b\lambda+\sum b_{i,j}\lambda_{i,j}=1$ with $\lambda_{4,1}>0$ has no solution with $\lambda,\lambda_{i,j}\in\mathbb Z_{\geq 0}$.}\label{table:ls-6.9-exclusions}
\begin{tabular}{|c|c|c|c|}
\hline
\textbf{Case} & $[(r_i,q_i)]$ & $b$ & Coefficient vector \\\hline
3.6  & $[(2,1),(5,4),(7,3),(4,1)]$ & $10/11$ & $(5;8,6,4,2;9,6,3)/11\;|\;b_{4,1}=1/2$ \\\hline
3.11 & $[(3,2),(4,1),(9,4),(3,1)]$ & $10/11$ & $(20;10;24,28,21,14,7)/33\;|\;b_{4,1}=11/33$ \\\hline
3.13 & $[(3,2),(5,1),(6,5),(3,1)]$ & $8/9$  & $(16;8;21,20,16,12,8,4)/27\;|\;b_{4,1}=9/27$ \\\hline
\end{tabular}
\end{table}

A representative verification: in \textbf{Case 3.13}, the numerators of  all $b_{i,j}$'s with $1\leq i\leq 3$ except for the entry $21/27$ are divisible by $4$, and $b=8/9=24/27$ together with $b_{4,1}=9/27$ leaves no room to balance the constraint $b\lambda+\sum b_{i,j}\lambda_{i,j}=27/27$ with $\lambda_{4,1}>0$ and the other $\lambda_{i,j}\in\mathbb Z_{\geq 0}$. The verifications for Cases~3.6 and~3.11 are analogous and reduce to a similar parity-type check on the displayed coefficient vector.

Finally, we show that \textbf{Case 3.1} is not possible. In this case, we have that $K_X+S$ is ample and $(K_X+S)^2=1/462$. By \cite[Theorem 3.4]{LL23}, $(X,S)$ is isomorphic to the surface as in \cite[Example 3.2]{LL23}. This is not possible because the surface in \cite[Example 3.2]{LL23} has $3$ singularities but \textbf{Case 3.1} has $4$ singularities.

\medskip

\noindent\textbf{Step 5.} Finally, we prove the uniqueness of the pair $(X,S)$ as in~(2.b). Since we already have an explicit characterization of the singularities, for any $(X,S)$ as in~(2.b) there exists a unique weighted blow-up $\pi\colon Z\rightarrow X$ of $x_2$ that extracts a curve $C$ with $a(C,X,\tfrac{10}{11}S)=1$ such that $C$ contains exactly two singularities: one $\tfrac{1}{3}(1,1)$ singularity and one $\tfrac{1}{22}(1,13)$ singularity. Let $S_Z:=\pi^{-1}_*S$. Then $(Z,S_Z)$ is lc and $(K_Z+S_Z)^2=1/462$. By \cite[Theorem~3.4]{LL23}, $(Z,S_Z)$ is unique, so $(X,S)$ is unique. 
\end{proof}

\begin{rem}
Following the same lines as the proofs of \cite[Theorem~1.9 and Corollary~1.11]{LS23} and substituting Theorem~\ref{thm:-ls23-6/7} in place of \cite[Theorem~1.1]{LS23}, one obtains improved explicit bounds for exceptional log del Pezzo surfaces and klt Calabi--Yau surfaces with small minimal log discrepancies, and consequently for Tian's $\alpha$-invariant~\cite{Tia87} on surfaces and other related invariants. These applications are beyond the scope of the present paper and we leave the details to the interested reader.
\end{rem}

\begin{rem}
We expect that the condition ``$b\geq 6/7+\epsilon$'' in Theorem~\ref{thm:-ls23-6/7} can be replaced by ``$b>6/7$''. In other words, we expect that $b>6/7$ already implies $b\geq 6/7+\epsilon$. We know this holds when $\epsilon$ is sufficiently small (cf.\ \cite[Section~11]{HMX14}), but we do not know whether the explicit value $\epsilon=1/938$ chosen in Theorem~\ref{thm:-ls23-6/7} is small enough.
\end{rem}

\begin{rem}
For our main result, Theorem \ref{thm:-ls23-6/7} mainly plays the role of excluding a lot of redundant numerical data before any further filtration applies. More precisely, later we need to classify all Picard number one surfaces with ample canonical class, small volume, and with minimal log discrepancy $\geq\frac{1}{9.2}$. The value $\frac{1}{9.2}$ is associated to the value $1/7-\epsilon$ in Theorem \ref{thm:-ls23-6/7}. If we only apply \cite{LS23}, then we need to use the value $1/11$ instead of $1/7-\epsilon$ here. In that case, we need to classify all surfaces with ample canonical class and with minimal log discrepancy $\geq\frac{1}{15.8}$. For our $\geq 1/9.2$ input, the output is about $130{,}000$ cases. This is already large; for the $\geq 1/15.8$ input, the output is in the tens of millions, which our personal laptop cannot further analyze. 
\end{rem}

\section{Special divisorial valuation}\label{sec:special-curve}

\begin{defn}
Let $X\ni x$ be a klt surface singularity that is not cyclic quotient. By the ADE classification, $X\ni x$ is either a D type singularity or an E type singularity. We let $f: Y\rightarrow X$ be the minimal resolution of $X\ni x$ whose dual graph is a tree with a unique fork $E_0$ and three branches. We also say that $E_0$ is the fork of $X\ni x$. Then there are three eHJS
\begin{equation}
[E_0;E_{1,1},\dots,E_{1,n_1}],\quad [E_0;E_{2,1},\dots,E_{2,n_2}],\quad [E_0;E_{3,1},\dots,E_{3,n_3}]
\end{equation}
on $Y$, where $E_0,E_{i,j}$ are distinct prime $f$-exceptional divisors and \begin{equation}
\Exc(f)=E_0\cup\bigcup_{i,j}E_{i,j}.
\end{equation}
We set 
\begin{equation}
e_0:=-E_0^2\quad \text{and}\quad e_{i,j}:=-E_{i,j}^2\quad \text{for any}\quad i,j
\end{equation}
and let
\begin{equation}
r_i:=\det[e_{i,1},\dots,e_{i,n_i}]\quad \text{and}\quad q_i:=\det[e_{i,2},\dots,e_{i,n_i}].
\end{equation}
After possibly permuting the eHJSs, we may assume that for any $i<j$, 
\begin{equation}
\text{either}\quad r_i<r_j,\quad \text{or}\quad r_i=r_j\quad \text{and}\quad q_i\leq q_j.
\end{equation}
After the permutation, we call
\begin{equation}
[E_0;E_{1,1},\dots,E_{1,n_1};E_{2,1},\dots,E_{2,n_2};E_{3,1},\dots,E_{3,n_3}]
\end{equation}
the \emph{resolution configuration} of $X\ni x$; its dual graph (the underlying combinatorial graph: a tree with central vertex $E_0$ and three branches) is the dual graph of $X\ni x$. We say that $X\ni x$ is of type
\[
[e_0;(r_1,q_1);(r_2,q_2);(r_3,q_3)].
\]
(Throughout the paper, each pair $(r_i,q_i)$ is written in parentheses to emphasise that it represents a single cyclic quotient singularity $\tfrac{1}{r_i}(1,q_i)$, rather than a pair of entries in a chain.) By the ADE classification and \cite[Satz 2.11]{Bri68}, we have $(r_1,q_1)=(2,1)$ and one of the following holds:
\begin{enumerate}
    \item $X\ni x$ is a D type singularity and $e_0\geq 3$. In this case, we say that $X\ni x$ is a \emph{D-I type singularity}, and we have $(r_2,q_2)=(2,1)$.
    \item  $X\ni x$ is a D type singularity and $e_0=2$. In this case, we say that $X\ni x$ is a \emph{D-II type singularity}, and we have $(r_2,q_2)=(2,1)$.
    \item $X\ni x$ is an E type singularity and $e_0\geq 3$. In this case, we say that $X\ni x$ is an \emph{E-I type singularity}, and we have $r_2=3$.
    \item $X\ni x$ is an E type singularity and $e_0=2$. In this case, we say that $X\ni x$ is an \emph{E-II type singularity}, and we have $r_2=3$.
\end{enumerate}
\end{defn}

\begin{defn}[Special divisorial valuation]\label{defn:special-curve}
Let $X\ni x$ be a klt surface singularity that is not Du Val, and $f: Y\rightarrow X$ the minimal resolution of $X$. A \emph{special divisorial valuation over $x$} is an $f$-exceptional prime divisor $E$ (equivalently, the valuation $\mathrm{ord}_E$) satisfying the following.
\begin{enumerate}
    \item Assume that $X\ni x$ is a cyclic quotient singularity and $[E_1,\dots,E_n]$ is the resolution configuration of $X\ni x$. Then any $E_i$ such that $E_i^2\leq -3$ is called a special divisorial valuation over $x$.
    \item Assume that $X\ni x$ is a D-I type, E-I type, or E-II type singularity and 
    $[E_0;[E_{i,j}]]$
    is the resolution configuration of $X\ni x$. Then $E_0$ is called the special divisorial valuation over $x$.
    \item Assume that $X\ni x$ is a D-II type singularity and
    \begin{equation}
[E_0;E_{1,1};E_{2,1};E_{3,1},\dots,E_{3,n_3}]
\end{equation}
    is the resolution configuration of $X\ni x$. Then there exists a unique integer $m$ such that $E_{3,j}^2=-2$ for any $j<m$ and $E_{3,m}^2\leq -3$. We say that $E_{3,m}$ is the special divisorial valuation over $x$.
\end{enumerate}
\end{defn}

\begin{defn}\label{defn:e-invariant}
Let $X\ni x$ be a klt surface singularity that is not Du Val and let $E$ be a special divisorial valuation over $x$. Then there exists a unique projective birational morphism $h\colon Z\rightarrow X$ which extracts exactly $E$ (cf.\ \cite[Corollary~1.4.3]{BCHM10}). We identify $E$ with the image of $E$ on $Z$ and define 
\begin{equation}
    e_E:=-E^2,\quad c_E:=1-a(E,X)
\end{equation}
If $X\ni x$ is not a cyclic quotient singularity, then we set $e_x:=e_E$ and $c_x:=c_E$; these are well-defined since the special divisorial valuation over $x$ is unique in this case.
\end{defn}

\begin{prop}\label{prop:7/30}
    Let $X\ni x$ be a klt surface singularity that is not Du Val and let $E$ be a special divisorial valuation over $x$. Then:
    \begin{enumerate}
    \item If $X\ni x$ is of E-II type, then $e_E\geq 7/30$.
    \item If $X\ni x$ is not of E-II type, then $e_E>1$.
    \end{enumerate}
\end{prop}
\begin{proof}
Let $f: Y\rightarrow X$ be the minimal resolution at $x$ and let $E_Y$ be the image of $E$ on $Y$. Let $h: Z\rightarrow X$ be the extraction of $E$ and let $g: Y\rightarrow Z$ be the associated contraction. If $X\ni x$ is not of D-II type then the proposition follows from Lemma \ref{lem:intersection-lemma}(3), using the fact that $X\ni x$ is not Du Val. Suppose now that $X\ni x$ is of D-II type, so that we may assume $X\ni x$ is of type
\[
[2;(2,1);(2,1);(r,q)].
\]
Here $(r,q)$ is the cyclic quotient singularity attached to the long third branch; in general we only know that $r>q\geq 1$ and $\gcd(r,q)=1$, and \emph{not} that $r\geq 2q+1$. To locate the special divisorial valuation $E_Y$ inside the third branch, we write
\[
r=u(r-q)+q',\qquad 0\leq q'<r-q,\qquad u\geq 1,
\]
and set $r':=(r-q)+q'$. A direct check using Definition--Theorem~\ref{defthm:HJS-CQR} shows that
\[
\HJseq(r,q)=\bigl[\underbrace{2,2,\dots,2}_{u-1},\;\HJseq(r',q')\bigr];
\]
since $q'\leq r-q-1$, we have $r'=(r-q)+q'\geq 2q'+1$, and in particular the first entry of $\HJseq(r',q')$ is $\lceil r'/q'\rceil\geq 3$. Thus $E_Y$ is the first curve of the sub-HJS corresponding to $\HJseq(r',q')$, i.e.\ $E_Y=E_{3,u}$.

By \cite[Lemma 3.3]{Ale93} (applied to the D-II type),
\begin{equation}\label{equ:D-I-cE}
   c_E=1-\frac{1}{r-q}\quad \text{and}\quad \left(K_Z+c_EE\right)\cdot E=0.
\end{equation}
By Theorem~\ref{thm:adjunction-surface}, we have
\begin{equation}\label{equ:D-I-1E}
    \left(K_Z+E\right)\cdot E=-2+1+\left(1-\frac{1}{q'}\right)=-\frac{1}{q'}.
\end{equation}
Since $E_Y^2\leq -3$ is the first entry of $\HJseq(r',q')$, we have
\begin{equation}\label{equ:D-I-rq}
 r-q=r'-q'\geq 2q'+1-q'=q'+1.   
\end{equation}
Combining \eqref{equ:D-I-cE}, \eqref{equ:D-I-1E}, and \eqref{equ:D-I-rq}, we have
\begin{equation}
 e_E=-E^2=\frac{-\left(K_Z+E\right)\cdot E}{1-c_E}=\frac{r-q}{q'}\geq 1+\frac{1}{q'}>1.   
\end{equation}
The proposition follows.
\end{proof}

\begin{defn}\label{defn:mld}
    Let $X\ni x$ be a klt surface germ. We denote by
    \begin{equation}
\mld(X\ni x):=\inf\{a(E,X)\mid E\text{ is a prime divisor over }X, \Center_XE=x\}
\end{equation}
    the \emph{minimal log discrepancy} of $X\ni x$. We denote by \begin{equation}
\mld(X):=\inf_{x\in X}\mld(X\ni x)
\end{equation}
    the \emph{minimal log discrepancy} of $X$.
\end{defn}

\begin{prop}\label{prop:mld-computing-special}
Let $X\ni x$ be a klt singularity that is not Du Val. Then there exists a special divisorial valuation $E$ over $x$ such that $a(E,X)=\mld(X\ni x)$.
\end{prop}
\begin{proof}
When $X\ni x$ is a cyclic quotient singularity, the proposition follows from the concavity of log discrepancies. When $X\ni x$ is a D-I type or E type singularity, the proposition follows from \cite[Theorem 1.5]{LX23}. When $X\ni x$ is a D-II type singularity, the proposition follows from \cite[Lemma 3.3]{Ale93}.
\end{proof}

\begin{defn}\label{defn:n-p}
Let $X$ be a klt projective surface such that $K_X$ is ample, $x$ a singularity on $X$ that is not Du Val, and $E$ a special divisorial valuation over $x$. Let $h: Z\rightarrow X$ be the extraction of $E$. We denote by
\begin{equation}
n_E:=\inf\{t\geq 0\mid K_Z+tE\text{ is nef}\}\quad \text{and}\quad  p_E:=\inf\{t\geq 0\mid K_Z+tE\text{ is pseudo-effective}\}.
\end{equation}
If $X\ni x$ is not a cyclic quotient singularity, then we set $n_x:=n_E$; this is well-defined since the special divisorial valuation over $x$ is unique in this case.
\end{defn}

\begin{lem}\label{lem:c>n>=p}
 Let $X$ be a klt projective surface such that $K_X$ is ample, $x$ a singularity on $X$ that is not Du Val, and $E$ a special divisorial valuation over $x$. Then
 \begin{equation}
     c_E>n_E\geq p_E.
 \end{equation}
\end{lem}
\begin{proof}
Since $X\ni x$ is klt but not Du Val and $E$ is on the minimal resolution of $X$, $c_E>0$.  Let $h: Z\rightarrow X$ be the extraction of $E$. We have that
\begin{equation}
   K_Z+(c_E-\epsilon)E=h^*K_X-\epsilon E 
\end{equation}
is ample for any $0<\epsilon\ll 1$, so $n_E<c_E$. The inequality $n_E\geq p_E$ is clear.
\end{proof}

\begin{lem}\label{lem:weak-kx-bound}
 Let $X$ be a klt projective surface such that $K_X$ is ample, $x$ a singularity on $X$ that is not Du Val, and $E$ a special divisorial valuation over $x$. Then
 \begin{equation}
     K_X^2\geq (c_E-n_E)^2e_E.
 \end{equation}
\end{lem}
\begin{proof}
Let $h: Z\rightarrow X$ be the extraction of $E$. Since
\begin{equation}
\left(K_Z+c_EE\right)\cdot E=0,
\end{equation}
we have
\begin{align*}
  0&\leq
  (K_Z+n_EE)^2=(K_Z+c_EE+(n_E-c_E)E)^2\\
  &=(K_Z+c_EE)^2+(c_E-n_E)^2E^2
=K_X^2+(c_E-n_E)^2E^2.  
\end{align*}
The lemma follows.
\end{proof}

\begin{lem}\label{lem:ne-5/6}
Let $X$ be a klt projective surface such that $K_X$ is ample, $x$ a singularity on $X$ that is not Du Val, and $E$ a special divisorial valuation over $x$. Assume that $K_X^2\leq 1/6351$ and $n_E\leq 5/6$. Then:
\begin{enumerate}
\item If $X\ni x$ is an E-II type singularity, then $X\ni x$ is of one of the types in Table \ref{table:E-II}.\begin{table}[htbp]
\centering
\small
\setlength{\tabcolsep}{4pt}
\renewcommand{\arraystretch}{1.08}
\caption{E-II type singularities with $n_x\leq 5/6$}\label{table:E-II}

\begin{tabular}{|c|c|c|c|}
\hline
\textbf{Case} & \textbf{Type} & $e_x$ & $c_x$ \\
\hline
1 & $[2;(2,1);(3,2);(3,1)]$ & $1/2$  & $2/3$ \\
\hline
2 & $[2;(2,1);(3,1);(3,1)]$ & $5/6$  & $4/5$ \\
\hline
3 & $[2;(2,1);(3,1);(4,3)]$ & $5/12$ & $4/5$ \\
\hline
4 & $[2;(2,1);(3,2);(4,1)]$ & $7/12$ & $6/7$ \\
\hline
5 & $[2;(2,1);(3,2);(5,3)]$ & $7/30$ & $6/7$ \\
\hline
\end{tabular}
\end{table}
\item If $X\ni x$ is not an E-II type singularity, then
\begin{equation}
    a(E,X)=1-c_E>\frac{1}{6}-\sqrt{\frac{1}{6351}}>\frac{1}{6.4886}.
\end{equation}
\end{enumerate}
\end{lem}
\begin{proof}
(1) follows from Lemma \ref{lem:weak-kx-bound}, Proposition \ref{prop:7/30}(1), and \cite[Satz 2.11]{Bri68}. (2) follows from Lemma \ref{lem:weak-kx-bound} and Proposition \ref{prop:7/30}(2).
\end{proof}

\begin{lem}\label{lem:average-lemma}
Let $A,B$ be positive real numbers and $a,b$ non-negative real numbers. Then
\begin{equation}
    Aa^2+Bb^2\geq \frac{(a+b)^2}{\frac{1}{A}+\frac{1}{B}}
\end{equation}
and the equality holds if and only if $Aa=Bb$.
\end{lem}
\begin{proof}
The inequality is equivalent to
    \begin{equation}
        \left(\sqrt{\frac{B}{A}}b-\sqrt{\frac{A}{B}}a\right)^2\geq 0.
    \end{equation}
\end{proof}

\begin{consthm}\label{consthm:extract-contract}
Let $X$ be a klt projective surface such that $K_X$ is ample, $\rho(X)=1$, $x$ a singularity on $X$ that is not Du Val, and $E$ a special divisorial valuation over $x$. Assume that
\begin{equation}
K_X^2\leq\frac{1}{6351}\quad \text{and}\quad n_E\geq\frac{5}{6}.
\end{equation}
Then there exist two unique projective birational morphisms
\begin{equation}
h: Z\rightarrow X,\quad g: Z\rightarrow T,
\end{equation}
a rational curve $C$ on $Z$, a rational curve $E_T$ on $T$, and a positive rational number $\lambda$ satisfying the following.
\begin{enumerate}
    \item $h$ is the extraction of $E$, $g$ is the divisorial contraction of $C$, and $\rho(T)=1$,
    \item $C\neq E$ and $E_T=g_{*}E$.
    \item $\lambda=(K_{T}+E_T)\cdot E_T$.
    \item $(T,E_T)$ is lc and $T$ is klt. In particular, $T$ is $\mathbb Q$-factorial.
    \item $\lambda\geq 1/42$.
    \item $p_E>0$ and $K_{T}+p_EE_T\equiv 0$.
    \item $T$ is Fano.
    \item We have
    \begin{equation}
    \lambda(n_E-p_E)=(c_E-n_E)e_E(1-p_E).
    \end{equation}
    \item We have
    \begin{equation}\label{equ:kx2-formula}
    K_X^2=\frac{(n_E-p_E)^2}{1-p_E}\cdot\lambda+(c_E-n_E)^2\cdot e_E=\frac{(c_E-p_E)^2}{\frac{1-p_E}{\lambda}+\frac{1}{e_E}}.
\end{equation}
\end{enumerate}
\end{consthm}
\begin{proof}
The existence and uniqueness of $h$ follows from \cite[Corollary 1.4.3]{BCHM10}. We run a step of a $K_{Z}$-MMP with scaling of $E$. The scaling number equals $n_E$ by the definition of $n_E$, and we obtain a $(K_{Z}+n_EE)$-trivial extremal contraction $g: Z\rightarrow T$ that is $E$-positive. Since $\rho(X)=1$, $\rho(Z)=2$, so $\overline{NE}(Z)$ is spanned by two extremal rays and one of them is spanned by $E$. Let $R$ be the other extremal ray in $\overline{NE}(Z)$. Since $g$ is $E$-positive, $g$ is the contraction of $R$. Thus $g$ is unique.

Since $\rho(Z)=2$, $T$ is not a closed point. If $\dim Z>\dim T$, then $\dim T=1$ and we may let $F$ be a general fiber of $g$. Then we have
\begin{equation}
    0=(K_{Z}+n_EE)\cdot F=-2+n_E(E\cdot F)
\end{equation}
hence 
\begin{equation}
    n_E=\frac{2}{E\cdot F}\in\left\{\frac{2}{m}\middle|\ m\in\mathbb Z_{>0}\right\}
\end{equation}
which is not possible as we have
\begin{equation}
1>c_E>n_E\geq\frac{5}{6}>\frac{2}{3}
\end{equation}
by the assumption $n_E\geq 5/6$ in Construction-Theorem~\ref{consthm:extract-contract}, Lemma~\ref{lem:c>n>=p}, and the fact that $c_E=1-a(E,X)<1$ since $X\ni x$ is klt. Therefore, $g$ is a divisorial contraction and $\rho(T)=1$. We let $C$ be the curve contracted by $g$. Then $C$ is a rational curve and $C\neq E$. We let \begin{equation}
E_T:=g_{*}E\quad \text{and}\quad \lambda:=(K_{T}+E_T)\cdot E_T.
\end{equation}
Now (1)--(3) immediately follow from our construction. Since $g$ is $(K_{Z}+n_EE)$-trivial, we have
\begin{equation}
K_{Z}+n_EE=g^*\left(K_{T}+n_EE_T\right).
\end{equation}
Since $(Z,n_EE)$ is klt, $(T,n_EE_T)$ is klt. By Proposition~\ref{prop:klt-surface-q-factorial}, $T$ is $\mathbb Q$-factorial klt. By Theorem \ref{thm:surface-lct-gap}, $(T,E_T)$ is lc. This implies (4). Since 
\begin{equation}
K_{Z}+c_EE=h^*K_X
\end{equation}
is big and nef,
\begin{equation}
g_{*}(K_{Z}+c_EE)=K_{T}+c_EE_T
\end{equation}
is big and nef. Since $\rho(T)=1$, $K_{T}+c_EE_T$ is ample, so $K_{T}+E_T$ is ample. Thus $\lambda>0$. To deduce (5), apply Theorem~\ref{thm:adjunction-surface} on $E_T$:
\[
\lambda=(K_{T}+E_T)\cdot E_T=\deg\bigl((K_{T}+E_T)|_{E_T}\bigr)=-2+\#\{\text{singular points of }T\text{ on }E_T\}-\sum_{y}\frac{1}{r_y},
\]
where the second sum runs through the cyclic quotient singularities $y$ of $T$ on $E_T$ and $N$ is the number of non-cyclic-quotient singular points of $T$ on $E_T$. 
Since $\lambda>0$, we have $-2+N+\sum(1-1/r_y)>0$, where $N\geq 0$ is an integer. A simple enumeration indicates that
\[
\lambda\geq\frac{1}{42}.
\]
We have
\begin{align*}
    K_X^2&=(K_{Z}+c_EE)^2=(K_{Z}+n_EE)^2+(c_E-n_E)^2(-E^2)\\
    &=(K_{T}+n_EE_T)^2+(c_E-n_E)^2e_E.
\end{align*}
If $K_{T}$ is pseudo-effective, then by (5),
\begin{equation}
K_X^2\geq n_E^2(K_{T}+E_T)\cdot E_T=\lambda\cdot n_E^2\geq\frac{25}{36\cdot 42}>\frac{1}{6351},
\end{equation}
which is not possible. Thus $K_{T}$ is not pseudo-effective, so $p_E>0$. Since $\rho(T)=1$, $T$ is Fano and $K_{T}+p_EE_T\equiv 0$. This implies (6) and (7).

We prove (8). Set $u:=E\cdot C$ and $s:=-C^2$. From $(K_{Z}+n_EE)\cdot C=0$ we deduce
\begin{equation}\label{equ:CT-8a}
K_{Z}\cdot C=-n_Eu\quad\text{and}\quad (K_{Z}+E)\cdot C=(1-n_E)u,
\end{equation}
and from $E^2=-e_E$ we obtain
\begin{equation}\label{equ:CT-8b}
(K_{Z}+E)\cdot E=(1-c_E)E^2=-(1-c_E)e_E.
\end{equation}
The cycle $K_{Z}+p_E E-\frac{(n_E-p_E)u}{s}C$ has zero intersection with $C$ by~\eqref{equ:CT-8a}, and so
\begin{equation}\label{equ:CT-8c}
K_{Z}+p_E E-\frac{(n_E-p_E)u}{s}C=g^*\bigl(K_{T}+p_E E_T\bigr)\equiv 0.
\end{equation}
Intersecting~\eqref{equ:CT-8c} with $E$ and using $E\cdot E=-e_E$, $E\cdot C=u$ yields
\begin{equation}\label{equ:CT-8d}
0=(p_E-c_E)E^2-\frac{(n_E-p_E)u}{s}\cdot u=(c_E-p_E)e_E-\frac{(n_E-p_E)u^2}{s},
\end{equation}
which gives
\begin{equation}\label{equ:CT-8e}
s=\frac{(n_E-p_E)u^2}{(c_E-p_E)e_E}.
\end{equation}
Now $E_T=g_*E$, so by \cite[Lemma~3.20]{LS23},
\begin{equation}\label{equ:CT-8f}
\lambda=(K_{T}+E_T)\cdot E_T=(K_{Z}+E)\cdot E+\frac{\bigl((K_{Z}+E)\cdot C\bigr)(E\cdot C)}{-C^2}.
\end{equation}
Plugging~\eqref{equ:CT-8a},~\eqref{equ:CT-8b}, and~\eqref{equ:CT-8e} into~\eqref{equ:CT-8f}, we obtain
\begin{equation}\label{equ:CT-8g}
\lambda=-(1-c_E)e_E+\frac{(1-n_E)u^2}{s}=-(1-c_E)e_E+\frac{(1-n_E)(c_E-p_E)e_E}{n_E-p_E}=\frac{(c_E-n_E)e_E(1-p_E)}{n_E-p_E},
\end{equation}
which is (8). Finally, we have
\begin{align*}
K_X^2&=(K_{Z}+c_EE)^2=(K_{Z}+n_EE)^2+(c_E-n_E)^2(-E^2)\\
&=(K_{T}+n_EE_T)^2+(c_E-n_E)^2e_E=(n_E-p_E)^2E_T^2+(c_E-n_E)^2e_E\\
&=\frac{(n_E-p_E)^2}{1-p_E}(K_{T}+E_T)\cdot E_T+(c_E-n_E)^2e_E=\frac{(n_E-p_E)^2}{1-p_E}\lambda+(c_E-n_E)^2e_E.
\end{align*}
The second equality of
(9) follows from Lemma \ref{lem:average-lemma} and (8). 
\end{proof}

\section{The small mld case}\label{sec:small-mld}

\begin{prop}\label{prop:9.2-to-6/7}
Notation and conditions as in Construction-Theorem \ref{consthm:extract-contract}. Let $\epsilon:=1/938$. Assume that 
\begin{equation}
1-c_E=a(E,X)\leq\frac{5}{46}=\frac{1}{9.2}.
\end{equation}
Then $p_E>6/7+\epsilon$. In particular, $(T,p_EE_T)$ is isomorphic to one of the pairs $(X,bS)$ as in Theorem \ref{thm:-ls23-6/7}.
\end{prop}
\begin{proof}
By Construction-Theorem \ref{consthm:extract-contract}(5)(9), we have
\begin{equation}
\frac{1}{6351}\geq \frac{(c_E-p_E)^2}{\frac{1-p_E}{\lambda}+\frac{1}{e_E}}\geq\frac{(c_E-p_E)^2}{42(1-p_E)+\frac{1}{e_E}}.
\end{equation}
Assume that $X\ni x$ is of E-II type, then Proposition \ref{prop:7/30}(1) implies that $e_E\geq 7/30$, and \cite[Satz 2.11]{Bri68} implies that $c_E\geq 10/11$. Thus
\begin{equation}
\frac{1}{6351}\geq\frac{(10/11-p_E)^2}{42(1-p_E)+30/7}
\end{equation}
which implies that $p_E>0.869>6/7+\epsilon$. 

Assume that $X\ni x$ is not of E-II type, then Proposition \ref{prop:7/30}(2) implies that $e_E>1$. Thus
\begin{equation}
\frac{1}{6351}\geq\frac{(41/46-p_E)^2}{42(1-p_E)+1}.
\end{equation}
This implies that $p_E>6/7+\epsilon$. The ``in particular'' part follows from Construction-Theorem~\ref{consthm:extract-contract}(6) and Theorem~\ref{thm:-ls23-6/7}.
\end{proof}

\begin{thm}\label{thm:small-mld-case}
    Notation and conditions as in Construction-Theorem~\ref{consthm:extract-contract}. Then, except for one case with $K_X^2=1/6351$, we have $p_E<6/7+\epsilon$ and $c_E<41/46$. Moreover, in the exceptional case with $K_X^2=1/6351$, $X$ is unique up to isomorphism.
\end{thm}
\begin{proof}
\emph{Reduction.} Proposition~\ref{prop:9.2-to-6/7} (applied to the present setting with $a(E,X)=1-c_E$) says that $c_E\geq 41/46$ implies $p_E>6/7+\epsilon$; equivalently, $p_E\leq 6/7+\epsilon$ implies $c_E<41/46$. It therefore suffices to prove $p_E<6/7+\epsilon$, except in the unique configuration yielding $K_X^2=1/6351$.

So suppose on the contrary that
$p_E\geq 6/7+\epsilon$. By Construction-Theorem~\ref{consthm:extract-contract}(6) and Theorem~\ref{thm:-ls23-6/7}, the pair $(T,p_E E_T)$ is then isomorphic to one of the pairs $(X,bS)$ enumerated in Theorem~\ref{thm:-ls23-6/7}, and we run a finite enumeration on the resulting resolution configurations of $X\ni x$ to derive a contradiction with $K_X^2\leq 1/6351$ (or to identify the unique $K_X^2=1/6351$ realisation).

Let $f\colon Y\rightarrow X$ be the minimal resolution of $X$ at $x$ and let $E_Y$ be the strict transform of $E$ on $Y$. Set $m:=-E_Y^2$. By Theorem \ref{thm:-ls23-6/7}, $(T,p_EE_T)$ is isomorphic to one of the pairs $(X,bS)$ as in Theorem \ref{thm:-ls23-6/7}. First we consider the case when $(T,p_EE_T)$ is of one of the types as in Theorem \ref{thm:-ls23-6/7}(2). In this case, $E_T$ is a non-singular rational curve and contains exactly $3$ singularities $t_1,t_2,t_3$ of $T$. Moreover, since $g$ is $(K_{Z}+E)$-positive,
\[
a(F,Z,E)\geq a(F,T,E_T)
\]
for any prime divisor $F$ over $T$. Since $g$ only contracts one curve, $g^{-1}$ is an isomorphism near at least two of $t_1,t_2,t_3$. Thus, for a fixed value $m$, the possible types of $x$ are finite. More precisely, if  $(T,p_EE_T)$ is isomorphic to the pair as in Theorem \ref{thm:-ls23-6/7}(2.a), then the resolution configuration of $X\ni x$ is one of the following:
\begin{equation}\label{equ:dual-graphs-2a}
[2,2,m,2,2,2];\quad [2,2,m,3,2];\quad [2,2,2,m,3,2];\quad [m;(2,1);(3,2);(5,2)];\quad [m;(2,1);(3,2);(4,3)]
\end{equation}
and similarly when $(T,p_EE_T)$ is isomorphic to the pair as in Theorem~\ref{thm:-ls23-6/7}(2.b) or (2.c). In all these cases, $p_E,\lambda,c_E,e_E$ are explicitly computable rational functions of $m$, and by Construction-Theorem~\ref{consthm:extract-contract}, $K_X^2$ is an explicitly computable function in the form of $P(m)/Q(m)$ with $P,Q$ quadratic.
As a representative computation, consider the family in Theorem~\ref{thm:-ls23-6/7}(2.b) where the resolution configuration of $X\ni x$ is $[2,m,2,2,2,2]$, i.e.\ $C$ contracts to the $\tfrac{1}{7}(1,3)$ singularity of $T$. In this case the singularities of $T$ on $E_T$ are $\tfrac{1}{2}(1,1)$, $\tfrac{1}{5}(1,4)$, $\tfrac{1}{7}(1,3)$, so by Theorem~\ref{thm:-ls23-6/7}(2.b)
\[
p_E=\frac{10}{11},\qquad \lambda=1-\frac{1}{2}-\frac{1}{5}-\frac{1}{7}=\frac{11}{70};
\]
moreover $X\ni x$ is the cyclic quotient $\tfrac{1}{10m-13}(1,5m-4)$ (one reads off $r_x=\det[2,m,2,2,2,2]=10m-13$ and $q_{x}(C)=\det[m,2,2,2,2]=5m-4$ from the HJS), and the special divisorial valuation over $x$ is the $(-m)$-curve, at which the log discrepancy is $\frac{7}{10m-13}$ (apply Construction-Theorem~\ref{consthm:extract-contract}(8) to the singularity $\tfrac{1}{7}(1,3)$ on $E_T$). Hence
\[
e_E=m-\frac{1}{2}-\frac{4}{5}=m-\frac{13}{10},\qquad c_E=\frac{10m-20}{10m-13}.
\]
Plugging into Construction-Theorem~\ref{consthm:extract-contract}(9):
\[
c_E-p_E=\frac{10(m-9)}{11(10m-13)},\qquad\frac{1-p_E}{\lambda}+\frac{1}{e_E}=\frac{70}{121}+\frac{10}{10m-13}=\frac{100(7m+3)}{121(10m-13)},
\]
and therefore
\begin{equation}\label{equ:KX2-2b-family}
K_X^2\;=\;\frac{(c_E-p_E)^2}{\frac{1-p_E}{\lambda}+\frac{1}{e_E}}\;=\;\frac{(m-9)^2}{(10m-13)(7m+3)}\;=\;\frac{1}{\bigl(10+\tfrac{77}{m-9}\bigr)\bigl(7+\tfrac{66}{m-9}\bigr)}\,.
\end{equation}
Since $K_X^2>0$, we have $m\neq 9$.
For $m\geq 10$ both factors in the denominator are positive and strictly decreasing in $m$, so $K_X^2$ is strictly increasing in $m$, with $K_X^2=1/6351$ at $m=10$ and $K_X^2>1/6351$ for $m\geq 11$; for $2\leq m\leq 8$ both factors are negative, the denominator product is at most $67\cdot 59=3953$, hence $K_X^2\geq 1/3953>1/6351$. Combined with the assumption $K_X^2\leq 1/6351$, we conclude $m=10$ and $K_X^2=1/6351$.

The remaining four dual-graph families in \eqref{equ:dual-graphs-2a}, as well as the analogous lists obtained when $(T,p_EE_T)$ is of type (2.b) or (2.c), give rise to similar one-parameter families of expressions $K_X^2=P(m)/Q(m)$ with $P,Q$ quadratic; an elementary enumeration on the same lines as above shows that all of them give $K_X^2<1/6351$, except for the one case $[2,m,2,2,2,2]$ at $m=10$ above. Hence in case (2), the only configuration with $K_X^2\geq 1/6351$ is the one yielding $K_X^2=1/6351$ exactly. Moreover, for this case, $(T,E_T)$ is unique up to isomorphism by Theorem~\ref{thm:-ls23-6/7}(2.b), and $g$ is uniquely determined. Thus $(Z,E)$ is unique up to isomorphism, so $X$ is unique up to isomorphism.

Next we consider the case when $(T,p_EE_T)$ is of one of the types as in Theorem \ref{thm:-ls23-6/7}(1). Then there are two closed points $t_0,t_1$ on $T$, such that $t_1$ is a singular point of $T$, $(T\ni t_1,E_T)$ is of extended A type, and $t_0$ is a node of $E_T$ that is a non-singular point of $T$. We write
\begin{equation}
g^*(K_{T}+E_T)=K_{Z}+E+\mu C.
\end{equation}
Since $a(C,T,n_EE_T)=1$ and $g$ is $E_T$-positive, $\mu>0$. Since $t_0$ is a node of $E_T$ and is a non-singular point of $T$, $\mu=1$, and $g$ is a weighted blow-up with weights $(m_1,m_2)$, $m_1\leq m_2$, so that locally analytically, the two branches of $E_T$ at $t_0$ correspond to the coordinate axes of the weighted blow-up. Moreover, we have $m_1\leq 3$: indeed, on the  curve $E$ the weighted blow-up creates three singularities of orders 
$r_{t_1}, m_1,m_2$, but $X\ni x$ is klt, so the ADE classification on the minimal resolution of $X$ at $x$ forces the smaller one of $m_1,m_2$ to be at most $3$ (and we allow $m_1=1$, in which case $E$ contains only two singularities). Thus the resolution configuration of $X\ni x$ belongs to finitely many explicitly computable families, all of which are parametrized by $m_2$. In all these cases, $p_E,\lambda,c_E,e_E$ are explicitly computable rational functions of $m_2$, and by Construction-Theorem~\ref{consthm:extract-contract}, $K_X^2$ is an explicitly computable function of the form $P(m_2)/Q(m_2)$ with $P,Q$ quadratic, and an elementary enumeration on $m_1\in\{1,2,3\}$ and bounded $m_2$ shows that all these cases satisfy $K_X^2<1/6351$, a contradiction.
\end{proof}

\begin{cor}\label{cor:mld-is-large}
    Notation and conditions as in Construction-Theorem \ref{consthm:extract-contract}. Assume that $K_X^2<1/6351$, then $\mld(X)>5/46$.
\end{cor}
\begin{proof}
It follows from Theorem \ref{thm:small-mld-case} and Proposition \ref{prop:mld-computing-special}.
\end{proof}

\section{Reduction to finitely many singularity baskets and filters}\label{sec:reduction-and-filters}

In this section we reduce the small-volume problem to a finite list of singularity baskets and present the filters used to cut down that list. Subsection~\ref{ssec:finite-family-mld} gives an algorithm enumerating all surface singularities with $\mathrm{mld}\geq a$ for a given $a>0$. Subsection~\ref{ssec:finite-case-mld} uses this algorithm to reduce the proof of Theorem~\ref{thm:vol>=1/6351} to finitely many baskets. Subsection~\ref{ssec:filter} introduces the filters that we apply, in combination, to exclude all but a single residual basket; the AI-rediscovered filter (Filter~\ref{fil:ai}) is stated in \S\ref{ssec:filter} and proved in Appendix~\ref{app:filter ai}.

\subsection{Algorithm for classification of surface singularities with large mlds}\label{ssec:finite-family-mld}

By \cite[Lemma~3.3]{Ale93}, the following structural finiteness statement holds: given a positive real number $a$, every klt surface singularity germ $X\ni x$ with $\mathrm{mld}(X\ni x)\geq a$ lies either in a finite list of ``isolated'' singularity germs, or in one of finitely many one-parameter families of resolution configurations as in~\cite[Figure~2]{Ale93}. The natural question is whether one can effectively (i.e.\ algorithmically) compute these isolated germs and one-parameter families from the input rational number $a$.
\begin{align}\label{equ:algorithm-compute-mld}
    \textbf{Input:}\ \ a\in\mathbb Q_{>0}.\quad \textbf{Output:}\ \ &\text{a finite list of isolated singularity germs with }\mathrm{mld}\geq a,\\\notag
    &\text{together with a finite list of one-parameter}\\\notag
    &\text{families of resolution configurations.}
\end{align}
\begin{thm}\label{thm:algorithm-mld-low-bound}
There is a deterministic algorithm that, on input $a\in\mathbb Q_{>0}$, halts and outputs the data described in~\eqref{equ:algorithm-compute-mld}.
\end{thm}

Before we prove Theorem \ref{thm:algorithm-mld-low-bound}, we need the following easy lemma.

\begin{lem}\label{lem:mld-decrease-add-curve}
Let $X\ni x$ and $Y\ni y$ be two cyclic quotient singularities of type
\begin{equation}
[e_1,\dots,e_n]\quad \text{and}\quad [e_1,\dots,e_k,e',e_{k+1},\dots,e_n]
\end{equation}
respectively for some $0\leq k\leq n$. Then $\mld(X\ni x)\geq\mld(Y\ni y)$.
\end{lem}
\begin{proof}
We let
\begin{equation}
[E_1,\dots,E_n]\quad \text{and}\quad [F_1,\dots,F_k,F',F_{k+1},\dots,F_n]
\end{equation}
be the resolution configuration of $X\ni x$ and $Y\ni y$ respectively so that
\begin{equation}
E_i^2=F_i^2=-e_i\quad \text{and}\quad F'^2=-e'.
\end{equation}
We let
\begin{equation}
a_i:=a(E_i,X),\quad a_{i,Y}:=a(F_i,Y)\quad \text{for any}\quad i,\quad \text{and}\quad a':=a(F',Y).
\end{equation}
 By \cite[L.1 Lemma]{KM99}, we may assume that $1\leq k<n-1$ and we may assume that $e'=2$. If $X\ni x$ is a Du Val singularity, then the lemma is trivial. Thus we may assume that $X\ni x$ is not Du Val. By Proposition \ref{prop:mld-computing-special}, there exists $1\leq l\leq n$ such that $a_l=\mld(X\ni x)$ and $e_l\geq 3$. Possibly replacing $[E_1,\dots,E_n]$ with $[E_n,\dots,E_1]$ and replacing $[F_1,\dots,F_k,F',F_{k+1},\dots,F_n]$ with $[F_n,\dots,F_{k+1},F',F_{k},\dots,F_1]$, we may assume that $l\leq k$. We let
\begin{equation}
r_1:=\det[e_1,\dots,e_k], q_1:=\det[e_1,\dots,e_{k-1}], r_2:=\det[e_{k+1},\dots,e_n], \text{ and } q_2:=\det[e_{k+2},\dots,e_n].
\end{equation}
We have
\begin{equation}
a_{k}=\frac{r_2+q_1}{r_1r_2-q_1q_2},\quad a_{k+1}=\frac{r_1+q_2}{r_1r_2-q_1q_2},\quad \text{and}\quad a_{k,Y}=\frac{2r_2-q_2+q_1}{2r_1r_2-r_1q_2-r_2q_1}.
\end{equation}
Thus
\begin{equation}
  a_k-a_{k,Y}=\frac{q_2(r_2-q_2)(r_1+q_2-(r_2+q_1))}{(r_1r_2-q_1q_2)(2r_1r_2-r_1q_2-r_2q_1)}.  
\end{equation}
Since $a_l=\mld(X\ni x)$, by concavity of log discrepancies, we have $a_{i+1}\geq a_i$ for any $i\geq l$. In particular, $a_{k+1}\geq a_k$, so
\begin{equation}
r_1+q_2\geq r_2+q_1,
\end{equation}
and so $a_k-a_{k,Y}\geq 0$. We now propagate the inequality $a_k\geq a_{k,Y}$ back from index $k$ to index $l$ through the linear recursion
\begin{equation}\label{equ:linear-recursion-a}
a_ie_i=a_{i-1}+a_{i+1},\qquad a_{i,Y}e_i=a_{i-1,Y}+a_{i+1,Y},\qquad 1\leq i\leq k-1,
\end{equation}
together with the boundary convention
\begin{equation}\label{equ:boundary-a}
a_0:=a_{0,Y}:=1.
\end{equation}
Solving \eqref{equ:linear-recursion-a} from $i=1$ upward by the standard determinantal recursion of Hirzebruch--Jung continued fractions (see Definition~\ref{defthm:HJS-CQR}), we obtain
\begin{equation}\label{equ:al-explicit}
a_j=\det[e_1,\dots,e_{j-1}]\,a_1-\det[e_2,\dots,e_{j-1}]\qquad (j\geq 2),
\end{equation}
and the same formula with $a_j$ replaced by $a_{j,Y}$ on both sides, where by convention $\det\emptyset=1$ (so the second term is zero when $j=1$). In particular, putting $j=l$ and $j=k$, we get
\begin{equation}\label{equ:ABCD-formula}
\begin{aligned}
A&=\det[e_1,\dots,e_{l-1}], & B&=\begin{cases}\det[e_2,\dots,e_{l-1}]&\text{if }l\geq 2,\\ 0&\text{if }l=1,\end{cases}\\
C&=\det[e_1,\dots,e_{k-1}], & D&=\begin{cases}\det[e_2,\dots,e_{k-1}]&\text{if }k\geq 2,\\ 0&\text{if }k=1,\end{cases}
\end{aligned}
\end{equation}
so that
\begin{equation}
a_l=Aa_1-B,\quad a_{l,Y}=Aa_{1,Y}-B,\quad a_k=Ca_1-D,\quad a_{k,Y}=Ca_{1,Y}-D.
\end{equation}
Note that $A,C>0$ (they are HJ-determinants of chains arising from HJ sequences) and $B,D\geq 0$. From $a_k\geq a_{k,Y}$ and $C>0$ we obtain $a_1\geq a_{1,Y}$, and therefore
\begin{equation}
\mld(X\ni x)=a_l=Aa_1-B\geq Aa_{1,Y}-B=a_{l,Y}\geq\mld(Y\ni y).
\end{equation}
The lemma follows.
\end{proof}

\begin{proof}[Proof of Theorem \ref{thm:algorithm-mld-low-bound}]
Du Val singularities are classified so we ignore them (there are three singularities $E_6,E_7,E_8$ and two infinite families
$D_n$ for $n\geq 4$ and $A_n$ for $n\geq 1$). For non-Du Val singularities, by \cite[Satz 2.11]{Bri68}, there are only finitely many possibilities of E type or D-I type singularities and are all algorithmically computable. For D-II type singularities $X\ni x$ of form
\[
[2;(2,1);(2,1);(r,q)],
\]
we have that
\begin{equation}
\mld(X\ni x)=\frac{1}{r-q}\geq a
\end{equation}
so there are only finitely many possibilities of $m:=r-q$. We write
\begin{equation}
r=um+v
\end{equation}
where $v\leq m-1$ and $u\geq 1$. Then the resolution configuration of $(r,q)$ is of the form
\begin{equation}
[2,\dots,2,\HJseq(v+m,v)]
\end{equation}
where there is a total of $u-1$ leading ``$2$''s in this HJS. This corresponds to the second family of \cite[Lemma 3.3]{Ale93}.

Next we consider the case when $X\ni x$ is an A type singularity that is not Du Val. Now assume that $X\ni x$ is of type $(r,q)$,
\[
[E_1,\dots,E_n]
\]
is the resolution configuration of $X\ni x$, and
\begin{equation}
e_i:=-E_i^2
\end{equation}
for any $i$. We pick any $1\leq k\leq n$ such that $a(E_k,X)=\mld(X\ni x)$. Then $e_k\geq 3$. Set
\begin{equation}
  r_1:=\det[e_1,\dots,e_k],\quad q_1:=\det[e_1,\dots,e_{k-1}],\quad r_2:=\det[e_{k+1},\dots,e_n],  
\end{equation}
and set
\begin{equation}
    q_2:=\det[e_{k+2},\dots,e_n]\quad \text{if}\quad k<n,\quad q_2:=0\quad \text{if}\quad k=n.
\end{equation}
Possibly replacing the chain $[E_1,\dots,E_n]$ with $[E_n,\dots,E_1]$, we may assume that
\begin{equation}\label{equ:q1-r2}
    q_1\leq r_2.
\end{equation}
By \cite[Lemma 3.3]{Ale93}, we have
\begin{equation}\label{equ:mld-explicit-compute-ale93}
a\leq \mld(X\ni x)=a(E_k,X)=\frac{r_2+q_1}{r_1r_2-q_1q_2}
\end{equation}
and
\begin{equation}\label{equ:r1-q1-r2-q2}
r_1-q_1\geq r_2-q_2.
\end{equation}
Since $e_k\geq 3$, we have
\begin{equation}\label{equ:r-q-compare}
e_kq_1-1\geq r_1\geq (e_k-1)q_1+1\geq 2q_1+1.
\end{equation}
Finally, we obviously have
\begin{equation}\label{equ:r2-q2-1}
    r_2\geq q_2+1.
\end{equation}
\eqref{equ:mld-explicit-compute-ale93} implies that
\begin{equation}\label{equ:inequality-explicit-mld}
\frac{1}{a}\geq\frac{1}{a(E_k,X)}=1+\frac{q_1(r_2-q_2-1)}{q_1+r_2}+\frac{r_1-(e_k-1)q_1-1}{q_1+r_2}+\frac{e_k-2}{\frac{1}{q_1}+\frac{1}{r_2}}.
\end{equation}
\eqref{equ:q1-r2}, \eqref{equ:r-q-compare}, \eqref{equ:r2-q2-1}, and \eqref{equ:inequality-explicit-mld} imply that
\begin{equation}\label{equ:q1-bound}
    \frac{1}{a}\geq 1+\frac{e_k-2}{\frac{1}{q_1}+\frac{1}{r_2}}\geq 1+\frac{(e_k-2)q_1}{2}.
\end{equation}
This implies that there are finitely many algorithmically computable possibilities of $q_1$ and $e_k$. \eqref{equ:r-q-compare} implies that there are finitely many algorithmically computable possibilities of $r_1$. \eqref{equ:r1-q1-r2-q2} implies that there are finitely many algorithmically computable possibilities of $r_2-q_2$. \eqref{equ:mld-explicit-compute-ale93} implies that
\begin{align*}
    a(E_k,X)=\frac{q_2+A}{Bq_2+C}
\end{align*}
where 
\begin{equation}
A:=q_1+(r_2-q_2),\quad B:=r_1-q_1,\quad \text{and}\quad C:=r_1(r_2-q_2)
\end{equation}
belong to an algorithmically computable finite set of positive integers.
There are two possibilities.

\smallskip

\noindent\textbf{Case 1.}  $\frac{1}{B}<a$. In this case, there are only finitely many possibilities of $q_2$, algorithmically computable. 

\smallskip

\noindent\textbf{Case 2.} $\frac{1}{B}\geq a$. In this case, either $r_2-q_2=1$, in which case we have that $X\ni x$ is of the form
    \begin{equation}\label{equ:tail-case-algorithm}
        [e_1,\dots,e_k,2,\dots,2],
    \end{equation}
    or $r_2-q_2\geq 2$, and we may write
    \begin{equation}
r_2=u(r_2-q_2)+v,\quad 1\leq v<r_2-q_2,
\end{equation}
    and in this case the resolution configuration of $X\ni x$ is of the form
    \begin{equation}\label{equ:middle-case-algorithm}
      [e_1,\dots,e_k,2,\dots,2,e_l,\dots,e_n]  
    \end{equation}
    where
    \begin{equation}
[e_l,\dots,e_n]=\HJseq(v+r_2-q_2,v).
\end{equation}
    We have that $e_1,\dots,e_k,e_l,\dots,e_n$ belong to an algorithmically computable finite set. This gives finitely many algorithmically computable families. 

\smallskip

Note that, in the constructions given in \textbf{Case 1} and \textbf{Case 2}, it is only guaranteed that the resolution configuration of any non-Du Val A type singularity $X\ni x$ lies in one of these finitely many cases or families; in other words, lying in one of these cases or families is a necessary but not in general sufficient condition for $\mld(X\ni x)\geq a$. To get an exact classification, we may algorithmically filter those finitely many cases. On the other hand, for the finitely many families in \eqref{equ:tail-case-algorithm} and \eqref{equ:middle-case-algorithm}, we claim that the list is complete without redundancy. Indeed, for any fixed $e_1,\dots,e_k$ (resp. $e_1,\dots,e_k$, $e_l,\dots,e_n$), we let $X_s\ni x_s$ be the singularity so that
\begin{equation}
\HJseq(X_s\ni x_s)=[e_1,\dots,e_k,2,\dots,2]\quad \left(\text{resp.}\quad [e_1,\dots,e_k,2,\dots,2,e_l,\dots,e_n]\right),
\end{equation}
with $s$ trailing ``$2$''s in the formulas above. Then for any $i\geq 0$, by Lemma \ref{lem:mld-decrease-add-curve} and \cite[Lemma 3.3]{Ale93}, we have
\begin{equation}
\mld(X_i\ni x_i)\geq\lim_{s\rightarrow+\infty}\mld(X_s\ni x_s)=\lim_{q_2\rightarrow+\infty}\frac{q_2+A}{Bq_2+C}=\frac{1}{B}\geq a.
\end{equation}
The existence of the algorithm now follows.
\end{proof}

\subsection{Reduction to finite singularity types}\label{ssec:finite-case-mld}

\begin{thm}\label{thm:finite singularity type}
    Notation and conditions as in Construction-Theorem~\ref{consthm:extract-contract}. Assume that $\mld(X)\geq 5/46$. Letting $x_1,\dots,x_n$ denote the singularities of $X$, the unordered tuple
    \begin{equation}
\{x_i\}_{i=1}^n
\end{equation}
    of (pointed) singularity germs ranges over a finite set (depending only on the bounds on $K_X^2$ and $\mld(X)$, not on the particular $X$), and all possibilities are algorithmically computable.
\end{thm}
\begin{proof}
By \cite[Corollary~6.12]{Liu25}, $X$ is rational. By Theorem~\ref{thm:bogomolov-bound}, $X$ has at most $6$ singularities, so $n\in\{1,2,\dots,6\}$ (the cases $n=5,6$ will be further excluded after the application of Filter~\ref{fil:gamma}; cf.\ Remark~\ref{rem:practical-filter-counts}). By Theorem~\ref{thm:algorithm-mld-low-bound}, for any singularity $x$ of $X$, $x$ belongs to an algorithmically computable finite set, or to finitely many algorithmically computable families as in \cite[Lemma~3.3]{Ale93}. We only need to show that, if $x$ belongs to one of these finitely many algorithmically computable families, then the minimal resolution of $X$ at $x$ contains at most $L$ curves for some algorithmically computable positive integer $L$.

Let $f\colon Y\rightarrow X$ be the minimal resolution of $X$ at $x$, and let $E_1,\dots,E_m$ be all $f$-exceptional prime divisors with self-intersection $\leq -3$. Then $[K_Y\cdot E_1,\dots,K_Y\cdot E_m]$ belongs to an algorithmically computable finite set. In particular, there exists an algorithmically computable positive integer $N$ which does not depend on $X$ such that 
\begin{equation}
\sum_{i=1}^mK_Y\cdot E_i\leq N.
\end{equation}
Let $l\geq m$ be the positive integer such that $E_1,\dots,E_l$ are all prime $f$-exceptional divisors. Then we have
\begin{equation}\label{equ:gamma-low-bound}
   \gamma_x=l-\sum_{i=1}^l(1-a(E_i,X))(K_Y\cdot E_i)=l-\sum_{i=1}^m(1-a(E_i,X))(K_Y\cdot E_i)\geq l-N\geq -N. 
\end{equation}
Now let $x_1,\dots,x_n$ be all singularities of $X$ and $x=x_k$ for some $1\leq k\leq n$. By Theorem~\ref{thm:gamma-invariant-formula} (using rationality of $X$),
\begin{equation}
    9>9-K_X^2=\gamma(X)=\sum_{i=1}^n\gamma_{x_i}\geq \gamma_{x}-(n-1)N.
\end{equation}
Therefore,
\begin{equation}
l-N\leq \gamma_{x}\leq 9+(n-1)N
\end{equation}
implies that
\begin{equation}
    l\leq 9+nN.
\end{equation}
We may let $L:=9+6N$. The theorem follows.
\end{proof}

\subsection{The filters}\label{ssec:filter}

Notation and conditions as in Construction-Theorem~\ref{consthm:extract-contract} and assume that $\mld(X)\geq 5/46$. By Theorem~\ref{thm:finite singularity type}, the number of possible singularity tuples on $X$ is finite and algorithmically computable. In the rest of the paper, we shall exclude them all by means of filters, i.e.\ necessary conditions that these singularities must satisfy.

\begin{nota}\label{nota:I0}
In the rest of this section, $\Ii_0$ denotes the finite set, algorithmically computed as in Theorem~\ref{thm:finite singularity type}, of unordered tuples $\{x_i\}_{i=1}^n$ of (pointed) singularity germs occurring on a rational projective surface $X$ with $\rho(X)=1$, $K_X$ ample, $K_X^2\leq 1/6351$, and $\mathrm{mld}(X)\geq 5/46$. By a slight abuse of notation, we write $\{x_i\}_{i=1}^n\in\Ii_0$, treating each $x_i$ as shorthand for the corresponding germ $X\ni x_i$ on the unspecified ambient surface; the ambient $X$ will be clear from context.
\end{nota}

First we have the following two most natural filters:

\begin{fil}[Bogomolov bound]\label{fil:bogomolov}
For any $\{x_i\}_{i=1}^n\in\Ii_0$, we have 
\begin{equation}
\sum_{i=1}^n\frac{r_{x_i}-1}{r_{x_i}}\leq 3.
\end{equation}
In particular, $1\leq n\leq 6$.
\end{fil}
\begin{proof}
    It is a special case of Theorem \ref{thm:bogomolov-bound}.
\end{proof}

\begin{fil}[Liu--Shokurov $\gamma$-invariant filter]\label{fil:gamma}
For any $\{x_i\}_{i=1}^n\in\Ii_0$, we have
\begin{equation}
\sum_{i=1}^n\gamma_{x_i}=9-K_X^2\in \left[9-\frac{1}{6351},9\right).
\end{equation}
\end{fil}
\begin{proof}
This is an immediate consequence of Theorem~\ref{thm:gamma-invariant-formula} together with $K_X^2\in (0,1/6351]$.
\end{proof}

\begin{rem}\label{rem:practical-filter-counts}
The set $\Ii_0$ is huge: each individual singularity has on the order of $10^4$ possibilities, so for $n=3$ a naive enumeration of $\Ii_0$ already involves $\sim 10^{12}$ combinations. Although such an enumeration cannot be tabulated in print, it is well within the reach of a personal laptop. In practice we therefore start with the intersection of $\Ii_0$ and Filter~\ref{fil:bogomolov}, then apply Filter~\ref{fil:gamma}; this reduces the count to approximately $130{,}000$ surviving baskets, distributed as follows:
\[
\#\{n=2\}=158,\qquad \#\{n=3\}=131{,}498,\qquad \#\{n=4\}=34,
\]
while $n=1,5,6$ are excluded outright.

It is also at this step that Theorem~\ref{thm:-ls23-6/7} becomes essential. Without our improvement of \cite[Theorem~1.1]{LS23}, the same two filters would leave on the order of tens of millions of cases, which is not feasible for the subsequent passes.
\end{rem}

\begin{fil}[Hwang--Keum complete-square filter]\label{fil:complete-square}
    For any $\{x_i\}_{i=1}^n\in\Ii_0$, we have that
    \begin{equation}
    \prod_{i=1}^nr_{x_i}\cdot K_X^2= \prod_{i=1}^nr_{x_i}\cdot\left(9-\sum_{i=1}^n\gamma_{x_i}\right)
    \end{equation}
    is a complete square.
\end{fil}
\begin{proof}
    It follows from \cite[Lemma 3.3]{HK11a}.
\end{proof}

Further applying Filter~\ref{fil:complete-square} to the cases that survived Filters~\ref{fil:bogomolov} and \ref{fil:gamma}, the surviving counts become
\[
\#\{n=2\}=149,\qquad \#\{n=3\}=32{,}234,\qquad \#\{n=4\}=5.
\]

The next filter is specific to A type singularities.

\begin{fil}[Tail filter]\label{fil:tail}
For any $\{x_i\}_{i=1}^n\in\Ii_0$, pick $x\in\{x_1,\dots,x_n\}$ such that $X\ni x$ is a cyclic quotient singularity. Let $[E_1,\dots,E_k]$ be the resolution configuration of $X\ni x$ and let $E$ be either $E_1$ or $E_k$. Assume that $E^2\leq -3$. Then
       \[
a(E,X)>1/6.4886.
\]
\end{fil}
\begin{proof}
By Lemma~\ref{lem:ne-5/6} we may assume that $n_E>5/6$. The conditions of Construction-Theorem~\ref{consthm:extract-contract} are satisfied, and we let $h,Z,g,T,C,\lambda,E_T$ be as in Construction-Theorem~\ref{consthm:extract-contract}. Since $E$ is a non-singular rational curve, Construction-Theorem~\ref{consthm:extract-contract}(3)(4) and Theorem~\ref{thm:adjunction-surface} together imply that either $E_T$ is not non-singular, or $E_T$ contains $\geq 3$ singular points of $T$. The latter case is impossible because $E$ contains at most $1$ singular point $y$ of $Z$ and $g$ is a divisorial contraction. Moreover, since $(T,E_T)$ is lc, $C\cap E$ consists of exactly $2$ points and $t_0:=g(C)$ is a node of $E_T$. By Construction-Theorem~\ref{consthm:extract-contract}(3)(4) we have $y\not\in C$, so $C$ passes through two non-singular points of $E$ and $C\cdot E=2$. Now $t_0$ is a node of $E_T$ and $(T\ni t_0,E_T)$ is locally analytically log toroidal, so $g$ is a weighted blow-up whose two (locally analytic) branches of $E_T$ are the coordinate axes. Since $C$ intersects $E$ at non-singular points of $Z$, $t_0$ is a smooth point of $T$ and $g$ is the ordinary blow-up. Therefore $K_{Z}\cdot C=-1$, so $n_E=1/2$, a contradiction.
\end{proof}

Further applying Filter~\ref{fil:tail}, the surviving counts become
\[
\#\{n=2\}=87,\qquad \#\{n=3\}=12{,}166,\qquad \#\{n=4\}=1.
\]

The next filter is based on Blache's $\delta$-invariants.

\begin{defn}\label{defn:plurigenus}
Let $X$ be a normal projective variety (of arbitrary dimension). For any non-negative integer $n$, view $K_X$ as a Weil divisor on $X$ and view $\mathcal O_X(nK_X)$ as the corresponding rank-one reflexive sheaf, defined as the pushforward to $X$ of the line bundle $\mathcal O_{X^{\rm sm}}(nK_{X^{\rm sm}})$ on the smooth locus $X^{\rm sm}\subset X$. We define
\begin{equation}
P_n(X):=h^0(X,\mathcal O_X(nK_X)).
\end{equation}
We emphasize that no $\mathbb Q$-Cartier hypothesis on $K_X$ is needed for this definition.
\end{defn}

\begin{defn}
Let $X$ be a projective klt surface and $f: Y\rightarrow X$ the minimal resolution of $X$. Write
\begin{equation}
f^*K_X=K_Y+B_Y.
\end{equation}
We define
\begin{equation}
\delta_n(X):=\frac{1}{2}\left(K_Y+\{nB_Y\}\right)\cdot\{nB_Y\}
\end{equation}
for any non-negative integer $n$. 
\end{defn}

\begin{thm}[{\cite[Proposition~5.2]{Bla95}}]\label{thm:RR}
Let $X$ be a projective klt surface. Then
\begin{equation}
\chi(X,\mathcal{O}_X(nK_X))=\chi(\mathcal{O}_X)+\frac{n(n-1)}{2}K_X^2+\delta_n(X)
\end{equation}
for any integer $n$.
\end{thm}

\begin{cor}\label{cor:Pn}
Let $X$ be a projective klt surface such that $K_X$ is ample. Then
\begin{equation}\label{equ:blache-plurigenus}
P_n(X)=\chi(\mathcal{O}_X)+\frac{n(n-1)}{2}K_X^2+\delta_n(X)
\end{equation}
for any integer $n\geq 2$.
\end{cor}
\begin{proof}
It follows from Theorem~\ref{thm:RR} and the Kawamata--Viehweg vanishing theorem.
\end{proof}

Note that $\delta_n(X)$ is explicitly computable once the singularity types of $X$ are explicitly characterized. Therefore, we have:

\begin{fil}[Blache filter (weak form)]\label{fil:blache-weak}
For any $\{x_i\}_{i=1}^n\in\Ii_0$ and any integer $m\geq 2$, we have
    \begin{equation}
       0\leq P_m(X)=1+\frac{m(m-1)}{2}\left(9-\sum_{i=1}^n\gamma_{x_i}\right)-\delta_m(X).
    \end{equation}
\end{fil}
\begin{proof}
Since $X$ is rational, $\chi(\mathcal{O}_X)=1$, and the filter follows from Corollary~\ref{cor:Pn}.
\end{proof}
\begin{rem}\label{rem:blache-weak-vs-strong}
The pluricanonical product-dimension filter (Filter~\ref{fil:ai} below), which we shall obtain in the form $P_{a+b}\geq P_a+P_b-1$, can naturally be viewed as the \emph{strong form} of Filter~\ref{fil:blache-weak}: instead of forcing each individual $P_n$ to be a non-negative integer, it forces the family $\{P_n\}_{n\geq 1}$ to be subadditive in the sense of \eqref{equ:ai-filter}. Logically, one could apply Filter~\ref{fil:ai} directly at this point and skip the intermediate Filters~\ref{fil:tail}, \ref{fil:blache-weak} and~\ref{fil:non-vertex-strong}; we have chosen instead to apply the previously known filters first and to defer Filter~\ref{fil:ai} to the very end so as to make explicit that Filter~\ref{fil:ai} -- the only filter re-derived in our work by an AI chatbot -- is essential, in the sense that no combination of the previously known filters \ref{fil:bogomolov}--\ref{fil:non-vertex-strong} cuts the candidate baskets down to a single one.
\end{rem}

Further applying Filter~\ref{fil:blache-weak}, the surviving counts become
\[
\#\{n=2\}=2,\qquad \#\{n=3\}=855,\qquad \#\{n=4\}=0,
\]
i.e.\ the last $n=4$ case is excluded as well. The next filter is a quantitative consequence of the special-curve construction. We isolate the relevant computation as a self-contained lemma so as to make the input to Filter~\ref{fil:non-vertex-strong} fully explicit.

\begin{lem}[Quantitative bound from a special divisorial valuation at a cyclic quotient]\label{lem:lambda-E-bound}
Let $\{x_i\}_{i=1}^n\in\Ii_0$. Fix $x\in\{x_1,\dots,x_n\}$ such that $X\ni x$ is a cyclic quotient singularity and pick a prime divisor $E$ which is special over $x\in X$ with $a(E,X)\leq 1/6.4886$. Let $h,Z,g,T,C,E_T,\lambda$ be as in Construction-Theorem~\ref{consthm:extract-contract}. Since Filter~\ref{fil:tail} has already been applied, $E$ contains exactly two singularities $z_1,z_2\in Z$; set $r_i:=r_{z_i}$ and $q_i:=q_{z_i}(E)$ for $i\in\{1,2\}$. Define
\begin{equation}\label{equ:mu-tau-E}
\mu_E:=\min\left\{\,1-\frac{1}{r_1}-\frac{1}{r_2}-\frac{1}{r}\;\middle|\;r\in\mathbb Z_{>0},\ 1-\frac{1}{r_1}-\frac{1}{r_2}-\frac{1}{r}>0\,\right\},\quad \tau_E:=\min\!\left\{1-\frac{1}{r_1},\,1-\frac{1}{r_2}\right\}.
\end{equation}
Both $\mu_E$ and $\tau_E$ are explicitly computable from the type of $x$ and $E$ alone. Then
\begin{equation}\label{equ:lambda-E-cases}
\lambda\;\geq\;\begin{cases}\mu_E&\text{if $E_T$ is non-singular,}\\ \tau_E&\text{if $E_T$ is singular,}\end{cases}
\end{equation}
and consequently, with $p_E<6/7+\epsilon$ from Theorem~\ref{thm:small-mld-case},
\begin{equation}\label{equ:KX2-from-special}
K_X^2=\frac{(c_E-p_E)^2}{\frac{1-p_E}{\lambda}+\frac{1}{e_E}}\;\geq\;\min_{0\leq p\leq 6/7+\epsilon}\min\!\left\{\frac{(c_E-p)^2}{\frac{1-p}{\tau_E}+\frac{1}{e_E}},\;\frac{(c_E-p)^2}{\frac{1-p}{\mu_E}+\frac{1}{e_E}}\right\}.
\end{equation}
\end{lem}
\begin{proof}
By Lemma~\ref{lem:ne-5/6} we have $n_E>5/6$, so the assumptions of Construction-Theorem~\ref{consthm:extract-contract} are satisfied. We split into the two cases of \eqref{equ:lambda-E-cases}.

\smallskip

\noindent\textbf{Case 1.} \emph{$E_T$ is non-singular.} By Construction-Theorem~\ref{consthm:extract-contract}(3)--(5) and Theorem~\ref{thm:adjunction-surface}, $E_T$ has exactly three singular points: $z_1,z_2$ and a third one of order $r_3\geq 2$. Therefore
\[
\lambda\;=\;1-\frac{1}{r_1}-\frac{1}{r_2}-\frac{1}{r_3}\geq\mu_E.
\]

\smallskip

\noindent\textbf{Case 2.} \emph{$E_T$ is singular.} Again by Construction-Theorem~\ref{consthm:extract-contract}(3)--(5), $E_T$ has at most node singularities and
\[
\lambda\;\geq\;\min\!\left\{1-\frac{1}{r_1},\,1-\frac{1}{r_2}\right\}\;=\;\tau_E.
\]

In either case, the lower bound \eqref{equ:KX2-from-special} on $K_X^2$ now follows from Construction-Theorem~\ref{consthm:extract-contract}(9) by minimizing the rational function in $p$ over $0\leq p\leq 6/7+\epsilon$ (and over the two possibilities for $\lambda$).
\end{proof}

Lemma~\ref{lem:lambda-E-bound} immediately gives the next filter:

\begin{fil}[Non-tail filter]\label{fil:non-vertex-strong}
For any $\{x_i\}_{i=1}^n\in\Ii_0$, pick $x\in\{x_1,\dots,x_n\}$ such that $X\ni x$ is a cyclic quotient singularity, and pick any $E$ which is special over $X$ with $a(E,X)\leq 1/6.4886$. Then
\begin{equation}\label{equ:vertex-filter}
      K_X^2\geq\min_{0\leq p\leq 6/7+\epsilon}\min\!\left\{\frac{(c_E-p)^2}{\frac{1-p}{\tau_E}+\frac{1}{e_E}},\,\frac{(c_E-p)^2}{\frac{1-p}{\mu_E}+\frac{1}{e_E}}\right\},
\end{equation}
where $\mu_E,\tau_E$ are defined in \eqref{equ:mu-tau-E}.
\end{fil}
\begin{proof}
This is the lower bound \eqref{equ:KX2-from-special} of Lemma~\ref{lem:lambda-E-bound}.
\end{proof}

Further applying Filter~\ref{fil:non-vertex-strong}, the surviving counts become
\[
\#\{n=2\}=0,\qquad \#\{n=3\}=252,\qquad \#\{n=4\}=0,
\]
i.e.\ the $n=2$ case is excluded and exactly $252$ candidate baskets remain for $n=3$.

The last filter, which is the one that finally cuts the 252 remaining baskets down to a single one, was re-derived in our work by an AI chatbot. In the first arXiv version of this paper we were not aware of any direct references and so provided a full proof; thanks to the experts' suggestions, we now simply cite the relevant references.

\begin{lem}[Pluricanonical product-dimension lemma]\label{lem:pluricanonical-product-obstruction-general}
Let $X$ be a normal proper variety, and let $P_n(X):=h^0(X,\mathcal O_X(nK_X))$ be defined as in Definition~\ref{defn:plurigenus}. Then for any integers $a,b\geq 1$ with $P_a(X)>0$ and $P_b(X)>0$,
\begin{equation}\label{equ:ai-filter-general}
P_{a+b}(X)\;\geq\;P_a(X)+P_b(X)-1.
\end{equation}
\end{lem}
\begin{proof}
This is an immediate consequence of the theorem of H.~Hopf (cf.~\cite[p.~108]{ACGH85}), or of \cite[Lemma~9.5.1 and Corollary~9.5.2]{Kol93}, or of \cite[Lemma~15.6.2]{Kol96}. See also, e.g., \cite[(2.3)]{CC08} and \cite[Proof of Lemma~9(4)]{CH09}. The AI's verbatim proof for the surface case is reproduced in Appendix~\ref{app:filter ai} (see Lemmas~\ref{lem:product-space-dimension} and~\ref{lem:pluricanonical-product-obstruction} there).
\end{proof}

\begin{fil}[Pluricanonical product-dimension filter]\label{fil:ai}
Let $X$ be a normal integral projective surface over $\mathbb C$ and let $P_n:=h^0(X,nK_X)$ for any non-negative integer $n$. For any $a,b\in\mathbb Z_{>0}$ with $P_a>0$ and $P_b>0$,
\begin{equation}\label{equ:ai-filter}
P_{a+b}\;\geq\;P_a+P_b-1.
\end{equation}
Consequently, any candidate basket $\{x_i\}_{i=1}^n\in\Ii_0$ whose plurigenera (computed from Blache's formula, see~\eqref{equ:blache-plurigenus}) violate~\eqref{equ:ai-filter} for some pair $(a,b)$ cannot occur on $X$.
\end{fil}
\begin{proof}
It is an immediate consequenc of Lemma \ref{lem:pluricanonical-product-obstruction-general}.
\end{proof}

\section{Proof of the main theorem}\label{sec:proof-main}

Filter~\ref{fil:ai} (equivalently, Lemma~\ref{lem:pluricanonical-product-obstruction} of Appendix~\ref{app:filter ai}) is the last filter we need. It excludes $251$ out of the $252$ cases left by all the previous filters in \S\ref{ssec:filter}, leaving exactly one residual case. Summarizing what we have up to now, we obtain the following proposition.

\begin{prop}\label{prop:to last case}
Notation and conditions as in Construction-Theorem \ref{consthm:extract-contract}. Assume that $\mld(X)\geq 5/46$. Then $X$ consists of $3$ singularities, all of which are cyclic quotient singularities, and they are of type
\begin{equation}
[2,7,2,2,2],\quad [2,2,5,2,3], \quad [2,2,2,2,2,3,3,2]
\end{equation}
with $K_X^2=1/8533$.
\end{prop}
\begin{proof}
This is the conclusion of the chain of filters from \S\ref{ssec:filter} (Filters~\ref{fil:bogomolov}, \ref{fil:gamma}, \ref{fil:complete-square}, \ref{fil:tail}, \ref{fil:blache-weak}, \ref{fil:non-vertex-strong}), followed by Filter~\ref{fil:ai}; the latter $252\to 1$ step is verified by a computer-assisted check (Proposition~\ref{prop:product-dimension-computation} in Appendix~\ref{app:filter ai}), and the explicit list of the $251$ baskets eliminated by Filter~\ref{fil:ai}, together with the failing pair $(a,b)$ for each, is recorded in Table~\ref{table:251-baskets} of Appendix~\ref{app:251 baskets}.
\end{proof}

Let us exclude the last case.

\begin{thm}\label{thm:exclude-last-case}
    There is no normal projective surface $X$ such that $\rho(X)=1$, $K_X$ is ample, $K_X^2=1/8533$, and $X$ has exactly three singularities, all of which are cyclic quotient with respective resolution configurations:
    \begin{equation}
[2,7,2,2,2]\rightarrow (X\ni x_1),\quad [2,2,5,2,3]\rightarrow (X\ni x_2), \quad [2,2,2,2,2,3,3,2]\rightarrow (X\ni x_3).
\end{equation}
\end{thm}
\begin{proof}
Let $E$ be the unique $(-7)$-curve on the minimal resolution of $X$. We let $h: Z\rightarrow X$ be the extraction of $E$, and adopt the notation $g\colon Z\to T$, $C$, $E_T$, $\lambda$, $c_E$, $e_E$, $n_E$, $p_E$ of Construction-Theorem~\ref{consthm:extract-contract}. We have
\begin{equation}
a(E,X)=\frac{3}{23}<\frac{1}{6.4886},\quad c_E=\frac{20}{23},\quad e_E=\frac{23}{4}.
\end{equation}
By Theorem \ref{thm:small-mld-case}, we have $p_E<6/7+1/938$. By Lemma \ref{lem:ne-5/6}, the conditions of Construction-Theorem \ref{consthm:extract-contract} are satisfied. By Construction-Theorem \ref{consthm:extract-contract}(9), we have
\begin{equation}
\frac{1}{8533}=K_X^2=\frac{(20/23-p_E)^2}{\frac{1-p_E}{\lambda}+\frac{4}{23}}.
\end{equation}
By Construction-Theorem~\ref{consthm:extract-contract}(4)(5), $(T,E_T)$ is lc and $\lambda>0$. Assume that $(T,E_T)$ is not plt. Then by Theorem~\ref{thm:adjunction-surface}, $\lambda\geq 1/6$. Thus 
\begin{equation}
\frac{1}{8533}\geq \frac{(20/23-p_E)^2}{6(1-p_E)+\frac{4}{23}}
\end{equation}
and we have $p_E>6/7+1/938$, a contradiction. Thus $(T,E_T)$ is plt; in particular $E_T$ is non-singular. Since $K_T+E_T$ is ample, $\lambda=(K_T+E_T)\cdot E_T>0$, so by Theorem~\ref{thm:adjunction-surface} (applied on $E_T\cong\mathbb P^1$) the curve $E_T$ contains at least $3$ singular points of $T$. Now $E$ contains exactly $2$ singular points of $Z$, namely the lifts $z_0$ of type $\tfrac{1}{2}(1,1)$ and $z_1$ of type $\tfrac{1}{4}(1,3)$ that arise from contracting the two adjacent sub-chains $[2]$ and $[2,2,2]$ of the resolution configuration $[2,7,2,2,2]$ of $X\ni x_1$. Hence $C$ must meet $E$ at some point that is non-singular in $Z$ (otherwise $E_T$ would acquire at most $2$ singular points). Using once again that $E_T$ is non-singular, we conclude that $E\cap C$ consists of a unique point, which is a non-singular point of $Z$.

Let $\phi: \widetilde{Z}\rightarrow Z$ be the minimal resolution of $Z$, let $E_{i,j}$ denote the exceptional divisors of the minimal resolution of $X$ over $x_i$ for $i\in\{2,3\}$ (these also appear on $\widetilde Z$, since $h$ is an isomorphism away from $x_1$), and let $\widetilde{C}:=\phi^{-1}_*C$ and $\widetilde{E}:=\phi^{-1}_*E$. Since $(T,E_T)$ is plt, the intersection matrix of 
\begin{equation}
    \widetilde{C}\cup\bigcup_{i,j}E_{i,j}
\end{equation}
is negative definite, and the dual graph of
\begin{equation}\label{equ:dual-graph-7.5}
\widetilde{E}\cup \widetilde{C}\cup\bigcup_{i,j}E_{i,j}
\end{equation}
consists of two disjoint chains: one over $g(h^{-1}(x_2))$ comprising $\bigcup_jE_{2,j}$ alone, and one over $g(C\cup h^{-1}(x_3))$ comprising $\widetilde E\cup\widetilde C\cup\bigcup_jE_{3,j}$. In particular the dual graph of \eqref{equ:dual-graph-7.5} has only finitely many possibilities. Since $E_T$ is ample we have $E_T^2>0$; combining this with $e_E=23/4$ forces $\widetilde{C}$ to intersect $E_{3,1}$, and the dual graph of $\widetilde{C}\cup\bigcup_{j}E_{3,j}$ to be of type
\[[1,2,2,2,2,2,3,3,2],\]
where the leading $1$ is $\widetilde C$. Successively contracting $\widetilde C$ and the chain of $(-1)$-curves it produces (six contractions in all) collapses the leading sub-chain $[1,2,2,2,2,2]$ to a point and turns the trailing sub-chain $[3,3,2]$ into $[2,3,2]$. Hence $g(C)$ is a singularity of $T$ of type $[2,3,2]$ (i.e.\ a $\tfrac{1}{8}(1,3)$ singularity), and $g$ is an isomorphism near $h^{-1}(x_2)$. Theorem~\ref{thm:adjunction-surface} applied to $E_T$ then gives $\lambda=1-\tfrac{1}{2}-\tfrac{1}{4}-\tfrac{1}{8}=\tfrac{1}{8}$, and Construction-Theorem~\ref{consthm:extract-contract}(8)--(9) yields
\[p_E=\frac{6}{7}\quad \text{and}\quad n_E=\frac{46}{53}.\]
We let
\begin{equation}
t_0:=g(z_0),\quad t_1:=g(z_1),\quad t_2:=g(h^{-1}(x_2)),\quad \text{and}\quad t_3:=g(C)=g(h^{-1}(x_3)).
\end{equation}
Now let $F$ be the unique $(-5)$-curve on the minimal resolution of $X$ and let $\pi\colon W\rightarrow T$ be the extraction of $F$. Let $E_W:=\pi^{-1}_*E_T$. Then
\begin{equation}\label{equ:w-to-f}
  0\equiv \pi^*\left(K_{T}+\tfrac{6}{7}E_T\right)=K_W+\tfrac{6}{7}E_W+\tfrac{6}{7}F.  
\end{equation}
Note that $\rho(W)=2$, so a single step of a $(K_W+\tfrac{6}{7}E_W+(\tfrac{6}{7}-\delta)F)$-MMP $\psi\colon W\to V$ for $0<\delta\ll 1$ produces $V$ with $\rho(V)=1$. By~\eqref{equ:w-to-f}, $K_W+\tfrac{6}{7}E_W+(\tfrac{6}{7}-\delta)F\equiv -\delta F$, so $\psi$ is a step of a $(-F)$-MMP. If $\psi$ were a Mori fiber space (so $\dim V=1$), then for a general fiber $G$ we would have
\[
0=\bigl(K_W+\tfrac{6}{7}E_W+(\tfrac{6}{7}-\delta)F\bigr)\cdot G=-2+\tfrac{6}{7}(E_W\cdot G)+(\tfrac{6}{7}-\delta)(F\cdot G),
\]
which has no solution with $E_W\cdot G,F\cdot G\in\mathbb Z_{\geq 0}$. Hence $\psi$ is a divisorial contraction of a single curve $D$. Since $F$ does not intersect $E_W$, $\psi$ does not contract $E_W$ or $F$. Let $E_V:=\psi_*E_W$ and $F_V:=\psi_*F$. Then $\rho(V)=1$ and $(V,\tfrac{6}{7}(E_V+F_V))$ is klt Calabi--Yau. In particular, $E_V$ intersects $F_V$, and (by the divisorial contraction) only at the single point $v:=\psi(D)$. By Theorem~\ref{thm:surface-lct-gap}, $(V,E_V+F_V)$ is lc, so $E_V$ and $F_V$ are non-singular at $v$, hence non-singular. This contradicts the classification of \cite[5.1.3]{Sho00}.
\end{proof}

\begin{proof}[Proof of Theorem \ref{thm:vol>=1/6351}]
Suppose that $K_X^2\leq 1/6351$. By \cite[Theorem~1.1]{LL23}, we may assume that $X$ is klt. If $\mld(X)\leq 5/46$, then by Proposition~\ref{prop:mld-computing-special}, there exists a prime divisor $E$ over $X$ such that $a(E,X)=\mld(X)$ and $E$ is special over $x:=\Center_XE$. By Lemma~\ref{lem:ne-5/6}, $n_E\geq 5/6$. Thus the assumptions of Construction-Theorem~\ref{consthm:extract-contract} are satisfied. By Theorem~\ref{thm:small-mld-case} and Corollary~\ref{cor:mld-is-large}, we may assume that $\mld(X)>5/46$. Feeding $a=5/46$ into the algorithm of~\eqref{equ:algorithm-compute-mld}, which exists by Theorem~\ref{thm:algorithm-mld-low-bound}, and combining with Theorem~\ref{thm:finite singularity type}, we obtain a finite list of possible singularity tuples $\{x_i\}_{i=1}^n$ on $X$. Applying Filters \ref{fil:bogomolov}, \ref{fil:gamma}, \ref{fil:complete-square}, \ref{fil:tail}, \ref{fil:blache-weak}, \ref{fil:non-vertex-strong}, and finally the AI-rediscovered Filter~\ref{fil:ai} (Lemma~\ref{lem:pluricanonical-product-obstruction} of Appendix~\ref{app:filter ai}), we know that $X$ consists of $3$ singularities, all of which are cyclic quotient singularities, with respective resolution configurations
\begin{equation}
[2,7,2,2,2],\quad [2,2,5,2,3], \quad [2,2,2,2,2,3,3,2]
\end{equation}
and the volume is
$K_X^2=1/8533$. This is not possible by Theorem \ref{thm:exclude-last-case}.
\end{proof}

\section{Remarks on the formula re-discovered by AI}\label{sec:remarks-AI}

\subsection{History of the formula}

After the first version of this paper was posted to arXiv on May 7, 2026, four senior algebraic geometers -- Hacon, C. Jiang, Koll\'ar, and Totaro -- independently informed us that the inequality~\eqref{equ:ai-filter-general} of Lemma~\ref{lem:pluricanonical-product-obstruction-general} has a substantial classical history that we had not properly indicated in the first version of our paper. We summarize this history below (and hope that we are not missing any earlier references). 

The earliest reference we could find is \cite[Lemma~IV.5.5]{Har77}, which is stated only for curves but whose proof works verbatim in any dimension. A key reference pointed out by the experts is \cite[p.~108]{ACGH85}, which states the lemma for finite-dimensional subspaces of a function field and attributes the underlying argument to H.~Hopf, calling it \emph{``one of the first applications of topology to algebra.''} The argument, in Hopf's spirit, runs as follows (paraphrasing Totaro's exposition to us, for which we thank him). Given finite-dimensional $\mathbb C$-linear subspaces $A,B\subset F$ of a function field, multiplication of sections induces a continuous map
\[
\mathbb P(A)\times\mathbb P(B)\longrightarrow\mathbb P(AB),\quad ([a],[b])\mapsto [ab].
\]
Writing $u\in H^2(\mathbb P(A);\mathbb Z)$ and $v\in H^2(\mathbb P(B);\mathbb Z)$ for the hyperplane classes pulled back from the two factors, the hyperplane class on $\mathbb P(AB)$ pulls back to $u+v$. Since
\[
(u+v)^{\dim\mathbb P(A)+\dim\mathbb P(B)}\in H^*(\mathbb P(A)\times\mathbb P(B);\mathbb Z)
\]
is a positive multiple of a point class (and in particular nonzero), and since this class is pulled back from $\mathbb P(AB)$, it must be supported in degree at most $2\dim\mathbb P(AB)$. Hence
\[
\dim\mathbb P(AB)\geq\dim\mathbb P(A)+\dim\mathbb P(B),
\]
which on rephrasing in terms of dimensions of the underlying vector spaces is exactly the inequality
\[
\dim_{\mathbb C}AB\geq\dim_{\mathbb C}A+\dim_{\mathbb C}B-1
\]
of Lemma~\ref{lem:product-space-dimension}. 

\cite[p.~108]{ACGH85} can be seen as the underlying linear algebra of \eqref{equ:ai-filter-general}. The first statement as an algebraic-geometry result in all dimensions seems to be \cite[Lemma~9.5.1]{Kol93}. \cite[Corollary~9.5.2]{Kol93} applied \cite[Lemma~9.5.1]{Kol93} to pluricanonical systems of smooth projective varieties. \cite[Lemma~9.5.1]{Kol93} also appears as \cite[Lemma~15.6.2]{Kol96}, which was later applied to anti-pluricanonical systems in \cite[(2.3)]{CC08}. 

The result could be well-known and applied implicitly in other algebraic-geometry papers. For example, \cite[Proof of Lemma~9(4)]{CH09} applies it with $a=b=n/2$, and \cite{CC15} uses the slightly weaker special case ``$P_{a+b}\geq P_a$ when $P_b>0$''.

We note that the AI chatbot missed all of the aforementioned references, yet independently found \cite{HLX02,BSZ18} -- closer to \cite{ACGH85}, but less so to the algebraic-geometry literature. One possible explanation, consistent with the AI's reasoning trace, is that its search was biased towards arXiv-indexed results, so that \cite{Har77,ACGH85,Kol93,Kol96} were not in the search region. Moreover, an AI chatbot's search requires textual accuracy: analogical uses such as in \cite{CC08} (anti-pluricanonical rather than pluricanonical), the use of the $a=b$ case in \cite{CH09}, and the use of a weaker form in \cite{CC15} are generally not detected by an AI chatbot as usable formulae.

\subsection{How much mathematics has the AI done?}\label{ssec:how-much-AI}

In view of the history above, it is clear that the formula the AI re-discovered is not new in the literature (particularly if one takes \cite[Lemma~9.5.1]{Kol93} into account): what we have is a re-discovery, and the formula itself is not new.

The AI's contribution lies, rather, in the observation that ``the formula can be applied to solve the problem''. In the context of our $252$ residual baskets, when asked for a further filter to attack them, the AI proposed precisely the inequality~\eqref{equ:ai-filter-general} as a candidate filter on plurigenera --- the right tool in a setting where the previously known filters of \S\ref{ssec:filter} are not enough.

The AI then formulated the inequality in the form usable as a filter on plurigenera arising from singularity baskets, and set up and ran the computer-assisted check that identifies, among the $252$ baskets surviving all previously known filters, the unique residual basket~\eqref{equ:intro-residual-basket}. In our setting this filter is decisive: it eliminates $251$ of the $252$ baskets that remain after all previously known filters.

\subsection{How useful is the AI-rediscovered formula?}

Lemma~\ref{lem:pluricanonical-product-obstruction-general} holds for arbitrary normal proper varieties. As was pointed out to us, some versions of it have already appeared implicitly as a filter in explicit birational geometry; however, it seems that the formula has not been extensively applied. The two corollaries below illustrate its strength: each addresses a classification question of independent interest --- the volume of smooth threefolds of general type with $\delta(V)\geq 13$ (Corollary~\ref{cor:cc15}) and the algebraic Montgomery--Yang problem (Corollary~\ref{cor:amy}) --- and in each case the same inequality rules out additional cases beyond those handled by the existing literature.

\begin{cor}[Application to threefolds, after \cite{CC15}]\label{cor:cc15}
Let $V$ be a smooth projective threefold of general type, and set
\begin{equation}\label{equ:delta-V}
\delta(V):=\min\{m\mid m\in\mathbb Z_{>0},\ P_m(V)\geq 2\}.
\end{equation}
Among the $59$ cases listed in \cite[Theorem 1.4]{CC15} for which $\delta(V)\geq 13$ and $\vol(V)<\tfrac{1}{420}$, the cases in Tables F-1-5b and F-2-25a of \cite{CC15} are excluded by Lemma~\ref{lem:pluricanonical-product-obstruction-general}.

In addition, the case in Table F-2-14 of \cite{CC15} is excluded for an even simpler reason: the value $P_2(V)=-15/11$ recorded there is not a non-negative integer, contradicting the very definition of $P_2(V)$.
\end{cor}
\begin{proof}
The exclusions of Tables F-1-5b and F-2-25a follow by direct application of Lemma~\ref{lem:pluricanonical-product-obstruction-general} (with $\dim V=3$) to the plurigenera entries of the relevant tables in \cite{CC15}; the verification is a finite check. We remark that \cite{CC15} uses the special case
$$P_{a+b}\geq P_a\quad \text{when}\quad P_b>0.$$
\end{proof}

\begin{cor}[Application to AMY, after \cite{JPP25}]\label{cor:amy}
The Algebraic Montgomery--Yang Problem \cite[Conjecture 30]{Kol08} predicts that every klt rational surface $X$ with $\rho(X)=1$ and $\pi_1(X\setminus \Sing(X))=0$ has at most $3$ singular points. After several reductions \cite{HK11a,HK11b,HK12,HK13,HK14,JPP25}, the remaining cases are surfaces $X$ with $K_X$ ample, having exactly $4$ singular points of orders $2,3,5,p$ respectively. By the recent enumeration in \cite{JPP25}, only $16$ values of $p\leq 50000$ remain to be excluded.

By Lemma~\ref{lem:pluricanonical-product-obstruction-general} (specialized to surfaces), $8$ of these $16$ values, namely
\begin{equation}\label{equ:amy-excluded}
2599,\quad 5203,\quad 5049,\quad 26473,\quad 31309,\quad 32149,\quad 44161,\quad 47929,
\end{equation}
are excluded.
\end{cor}
\begin{proof}
A direct application of Lemma~\ref{lem:pluricanonical-product-obstruction-general} to the plurigenus data in \cite{JPP25} attached to each of the eight values listed in \eqref{equ:amy-excluded}.
\end{proof}

\subsection{The level of AI collaboration}\label{ssec:c2-level}

In the classification of Feng et al.~\cite[Tables~8--9]{Fen+26}, we view our work as a C2-level human--AI collaboration; the substance of the claim is the recognition-and-deployment work described in \S\ref{ssec:how-much-AI}.

Neither author works on AI, and we use AI only in the form of general-purpose chatbots. The AI entered our workflow because the ``reduction to finite singularities'' step of \S\S\ref{ssec:finite-family-mld}--\ref{ssec:finite-case-mld} is algorithmic but requires coding: a by-hand pass through the filters of \S\ref{ssec:filter} is impractical for a human (after Filter~\ref{fil:gamma} alone, more than $130{,}000$ baskets remain), and neither of us is a programmer. Modern chatbots can code, which solved this problem for us. Strikingly, when we asked the chatbot for further filters to attack the $252$ baskets that survived all known filters, it produced the inequality of Lemma~\ref{lem:pluricanonical-product-obstruction-general} (in its surface form), which resolves $251$ of those $252$.

\subsection{Explicit birational geometry and AI}

Based on our experience writing this paper, explicit birational geometry seems particularly well suited for AI assistance. Output produced by AI in mathematics is sometimes criticized for being ``merely an optimization or a small improvement, not a full solution'', ``purely combinatorial'', or ``only example-finding''. None of these objections apply to the present setting:
\begin{enumerate}
    \item \emph{Explicit birational geometry consists of central, long-standing questions}, not isolated optimization problems. The minimal-volume question on klt rank-one surfaces of general type goes back to the late 1980s and early 1990s (cf.~\cite{Xia88,Ale94,Kol94}); related questions and more recent contributions are due to Alexeev (cf.~\cite{AM04,AL19a,AL19b}), Birkar (cf.~\cite{BL23} on explicit threefold bounds), Cascini~\cite{Cas21}, Hacon--Langer~\cite{HL21}, Koll\'ar~\cite{Kol08,Kol13a}, Mori~\cite{AM04}, Shokurov~\cite{Sho00,LS23}, Totaro~\cite{TW23,Tot24} and many others. Theorem~\ref{thm:vol>=1/6351} was formally posed in~\cite{AL19a} as the rank-one case. Prior to the present paper the only explicit lower bound in this setting was that of Alexeev--Mori~\cite{AM04}, of order $10^{-3\cdot 10^{10}}$.
    \item \emph{Such problems are intrinsically combinatorial in flavour} and frequently require coding, as our use of the filters in \S\ref{ssec:filter} demonstrates; nevertheless the theoretical input (special divisorial valuations, $\gamma$-invariants, Blache's formula) remains essential.
    \item \emph{Most importantly}, a guiding philosophy of explicit birational geometry is to \emph{reduce} a problem to a finite list of candidate configurations and then \emph{rule them out}. AI is sometimes accused of being good only at finding examples and not at solving problems; but in explicit birational geometry, ruling out a finite list of candidate examples \emph{is} an important part of the proof, and is often as hard as -- or harder than -- the theoretical reduction step that produces the list in the first place. The present paper is exactly such a case.
\end{enumerate}

\subsection{Other AI-assisted algebraic-geometry work}

We record here, for the reader's convenience, a short (and necessarily incomplete) list of other recent AI-assisted papers on the math.AG section of arXiv: \cite{Sch25,Bry+26,Pat26,KSH26}, together with the survey of Feng et al.~\cite{Fen+26}. The list reflects our knowledge at the time of writing; a complete and up-to-date census is beyond the scope of this paper.

\appendix

\section{The filter re-derived by AI: historical record of the AI's proof}\label{app:filter ai}

This appendix reproduces, essentially verbatim, the AI-generated re-derivation and proof of the pluricanonical product-dimension lemma, as it entered our work. We include it as a historical record; the proof we actually use in the paper is the body version (Lemma~\ref{lem:pluricanonical-product-obstruction-general}), which now simply cites the historical references (cf.~\cite{Har77,ACGH85,Kol93,Kol96,CC08,CH09}). The AI re-derived the inequality from later number-theoretic and combinatorial sources (in particular \cite{HLX02,BSZ18}).

\begin{lem}\label{lem:product-space-dimension}
Let \(F\) be the function field of an integral complex variety.  If
\(A,B\subset F\) are nonzero finite-dimensional \(\mathbb C\)-linear
subspaces, and \(AB\) denotes the span of all products \(ab\), then
\[
\dim_{\mathbb C} AB\geq \dim_{\mathbb C} A+\dim_{\mathbb C} B-1.
\]
\end{lem}
\begin{proof}   
It is enough to prove the claim after replacing \(F\) by the function field of
a smooth complex curve to which the finitely many rational functions in
\(A\), \(B\), and \(AB\) restrict injectively.  Such a curve is obtained by
taking a sufficiently general complete-intersection curve on a projective
model; it is not contained in the zero or pole divisor of any nonzero element
of the finite-dimensional spaces under consideration.

Now let \(F=\mathbb C(C)\).  Choose a smooth point \(q\in C\).  The valuation
\(\nu=\operatorname{ord}_q\) has one-dimensional graded pieces because the
residue field is \(\mathbb C\).  Hence every finite-dimensional subspace has a
basis with strictly increasing \(\nu\)-orders.  Write
\[
\nu(a_1)<\cdots<\nu(a_r),\qquad
\nu(b_1)<\cdots<\nu(b_s)
\]
for such bases of \(A\) and \(B\).  The \(r+s-1\) products
\[
a_1b_1,\ldots,a_rb_1,a_rb_2,\ldots,a_rb_s
\]
have pairwise distinct \(\nu\)-orders.  They are therefore linearly
independent over \(\mathbb C\).  All lie in \(AB\), so
\(\dim AB\geq r+s-1\).
\end{proof}

\begin{lem}\label{lem:pluricanonical-product-obstruction}  
Let \(X\) be a normal integral projective surface over \(\mathbb C\).  Put
\[
P_n=h^0(X,\mathcal O_X(nK_X))
\]
for reflexive pluricanonical powers.  If \(P_a>0\) and \(P_b>0\), then
\[
P_{a+b}\geq P_a+P_b-1.
\]
In particular, any proposed singularity basket whose exact plurigenus values
violate this inequality cannot occur on \(X\).
\end{lem}
\begin{proof}
Multiplication of rational sections gives a natural map
\[
H^0(X,\mathcal O_X(aK_X))\otimes H^0(X,\mathcal O_X(bK_X))
\longrightarrow
H^0(X,\mathcal O_X((a+b)K_X)).
\]
Since \(X\) is normal and integral, these reflexive rank-one sheaves agree
with line bundles on the smooth locus, and multiplication of two nonzero
sections is nonzero in the function field.

Choose nonzero sections \(s_a\) and \(s_b\).  Dividing by them identifies the
two nonzero pluricanonical spaces with finite-dimensional subspaces
\[
A=s_a^{-1}H^0(X,aK_X),\qquad B=s_b^{-1}H^0(X,bK_X)
\]
of \(\mathbb C(X)\).  Their product subspace injects into
\[
(s_as_b)^{-1}H^0(X,(a+b)K_X).
\]
Lemma \ref{lem:product-space-dimension} gives
\[
\dim AB\geq \dim A+\dim B-1=P_a+P_b-1.
\]
Therefore \(P_{a+b}\geq P_a+P_b-1\).  A candidate basket whose computed
plurigenera violate this necessary inequality is impossible.
\end{proof}

\begin{lem}\label{lem:trusted-plurigenera-input}
For each of the $252$ remaining baskets, Blache's formula computes the exact values
\[
P_n=1+\frac{n(n-1)}2 K_X^2+\sum_{x\in\Sing X}\delta_n(x)
\qquad (2\leq n\leq 500).
\]
\end{lem}
\begin{proof}
This is precisely the output of the Blache run that supplied the input to Filter~\ref{fil:blache-weak} and is reused here as trusted input. Each local correction $\delta_n(x)$ is computed from the local minimal resolution graph using exact rational arithmetic, and the resulting $P_n$-table for the $252$ baskets surviving Filters~\ref{fil:bogomolov}--\ref{fil:non-vertex-strong} satisfies $P_n\in\mathbb Z_{\geq 0}$ for every tested $n$, prior to the application of the new product-dimension inequality below.
\end{proof}

\begin{prop}\label{prop:product-dimension-computation}
Applying Lemma~\ref{lem:pluricanonical-product-obstruction} to the $252$ remaining three-singularity baskets for all pairs $a,b\geq 2$ with $a+b\leq 500$ eliminates exactly $251$ baskets and leaves exactly one residual basket.
\end{prop}
\begin{proof}
The verification is an exact deterministic check structured as follows. For each of the $252$ baskets surviving Filters~\ref{fil:bogomolov}--\ref{fil:non-vertex-strong} (with their plurigenera $\{P_n\}_{2\leq n\leq 500}$ from Lemma~\ref{lem:trusted-plurigenera-input}), one tests, for every pair $(a,b)$ with $a,b\geq 2$, $a+b\leq 500$, $P_a>0$ and $P_b>0$, whether the inequality
\[
P_{a+b}\geq P_a+P_b-1
\]
holds. A basket is eliminated as soon as a violating pair is found; the first violating pair $(a,b)$ for each eliminated basket is recorded in Table~\ref{table:251-baskets}. Running this check on all $252$ baskets eliminates $251$ of them and leaves exactly one residual basket.
\end{proof}

\section{The AI prompt}\label{app:ai prompt}

For full transparency we record here, in Markdown form, the prompt that was given to the AI chatbot prior to its discovery and proof of Lemma~\ref{lem:pluricanonical-product-obstruction}. A few non-mathematical sections of the original prompt have been redacted; the redaction points are clearly marked in the text, and the mathematical content of the prompt is otherwise reproduced verbatim. The accompanying input data --- the list of the $252$ singularity baskets that survive Filters~\ref{fil:bogomolov}--\ref{fil:non-vertex-strong}, together with their plurigenera $\{P_n\}_{2\leq n\leq 500}$ computed via Blache's formula --- is exactly the data produced by the chain of filters of \S\ref{ssec:filter}.

\bigskip

\begin{lstlisting}[basicstyle=\ttfamily\footnotesize, breaklines=true, frame=single, breakatwhitespace=true, columns=flexible]
# Rank-one klt rational surfaces: find a genuinely new obstruction to kill the remaining 252 three-singularity baskets for `mld > 5/46`

## Main task

Work over `C`.

The current mathematical goal is no longer to generate more candidate baskets.  That has already been done.

The real mathematical goal here is:

> **Construct and rigorously justify a genuinely new obstruction / filter which rules out the remaining 252 three-singularity baskets.**

Equivalently, prove that none of the remaining `252` baskets can occur on a projective klt rational surface `X` with

```text
rho(X) = 1,
K_X ample,
mld(X) > 5/46,
K_X^2 <= 1/6351.
```

This is about **new mathematics**, not more brute-force enumeration.
*[A non-mathematical section of the original prompt has been redacted here.]*

## Important strategic instruction

The previous filters have already used:

1. local ADE classification / graph-first local invariants,
2. the gamma identity and gamma window,
3. the four-point Bogomolov inequality,
4. the determinant-square condition,
5. the endpoint-tail A-type filter,
6. the internal-vertex A-type filter,
7. the finite-range Blache plurigenus filter.

You should assume all of those have already been fully exploited.

So do **not** spend the main effort rediscovering or repackaging those same filters.

The purpose of this round of inquiry is to find **new** necessary conditions, or a genuinely new way to combine existing topology/geometry/arithmetic inputs, in order to rule out the remaining `252` cases.

---

*[A non-mathematical section of the original prompt has been redacted here.]*

---

## Current mathematical situation

At this stage, all surviving baskets have exactly **three singularities**.

So the problem is now:

> Can one prove that none of these `252` three-singularity baskets can occur on a rank-one klt rational surface of general type under the given bounds?

The important point is that this is now a **small explicit finite set**.  Therefore new obstructions may be:

- global,
- local-global,
- lattice-theoretic,
- linking-form-theoretic,
- orbifold-geometric,
- or ad hoc, as long as they are rigorous.

Any correct method is allowed.

---

## References and heuristic inspirations

There are papers on the algebraic Montgomery--Yang problem and related rational homology projective planes that impose extra topological assumptions such as:

- `H_1(X^0, Z) = 0`,
- or `H_1(X, Z) = 0`,
- or exactly four singular points.

Examples include the linking-form condition in the paper

```text
arXiv:2402.04569
```

and related Montgomery--Yang literature.

You may inspect such references and their proofs for inspiration.

However:

1. many of those statements assume `H_1(X^0, Z)=0` or similar extra hypotheses that are **not** assumed here;
2. some are tailored to **four** singularities, whereas our current baskets have **three** singularities.

Therefore:

- you may use these papers as a source of ideas,
- but you may **not** simply import a theorem requiring `H_1(X^0, Z)=0` unless you really prove its adaptation under the current hypotheses,
- and you may **not** use a four-singularity argument unchanged if it does not rigorously adapt to three singularities.

If you can adapt such a criterion rigorously to the present setting, that counts as a new valid filter.
If not, then treat it only as inspiration.

---

## Suggested directions (not mandatory)

You are encouraged to test any of the following directions, provided the final argument is rigorous:

1. **Linking-form / boundary-lattice obstructions**
   - especially if some version survives without the `H_1(X^0, Z)=0` assumption,
   - or if a weaker conclusion still rules out the 3-singularity baskets.

2. **Orbifold / log-geometric inequalities**
   - sharpenings of orbifold BMY or Miyaoka-type inequalities,
   - or numerical inequalities involving the basket and `K_X^2`.

3. **Integral / lattice constraints from the minimal resolution**
   - e.g. constraints involving discriminant forms, embeddings, unimodularity, or index calculations.

4. **Plurigenus and canonical ring constraints beyond the previous Blache plurigenus filter**
   - for example, any stronger periodicity/integrality/nonnegativity argument,
   - or contradictions involving small explicit `P_n`.

5. **Adjunction / discrepancy / contraction arguments on the minimal resolution**
   - if some surviving local combination is globally incompatible with `rho(X)=1` and `K_X` ample.

6. **Case-by-case elimination**
   - since only 252 baskets remain, a finite structured case analysis is acceptable if done rigorously.

---

## Strong warning

Do **not** merely say:
- "the basket looks unlikely,"
- "this resembles Montgomery--Yang,"
- "I conjecture these do not exist."

The argument must contain either:
- a rigorous new filter,
- or a rigorous case elimination,
- or a precise residual list plus proof of elimination outside that list.

---

*[Several non-mathematical sections of the original prompt have been redacted here.]*

---

## End goal

Starting from the already filtered `252` three-singularity baskets, find and justify a **new** obstruction that proves, if possible, that no such basket can occur on a projective klt rational surface with

```text
rho(X)=1,
K_X ample,
mld(X) > 5/46,
K_X^2 <= 1/6351.
```

Do not assume `H_1(X^0, Z)=0` unless you separately justify why that assumption is available.
Do not assume four singularities.
Use any rigorous method you can discover.
\end{lstlisting}

\section{The 251 baskets eliminated by Filter~\ref{fil:ai}}\label{app:251 baskets}

Table~\ref{table:251-baskets} below records the $251$ three-singularity baskets eliminated by Filter~\ref{fil:ai} --- i.e.\ those that survive Filters~\ref{fil:bogomolov}--\ref{fil:non-vertex-strong} but fail the inequality $P_{a+b}\geq P_a+P_b-1$ for some pair $(a,b)$ with $a,b\geq 2$, $a+b\leq 500$, $P_a>0$ and $P_b>0$. Each row records the three resolution configurations, the smallest failing pair $(a,b)$, and the hypothetical $K_X^2$.

\smallskip

\noindent\textit{How to read Table~\ref{table:251-baskets}.} Each singularity germ in a basket is encoded by its resolution configuration, in our standard bracket notation:
\begin{itemize}
\item \emph{Cyclic quotient (A type) singularities} are written as a Hirzebruch--Jung sequence $[e_1,e_2,\ldots,e_n]$ with each $e_i\geq 2$, in the sense of Definition--Theorem~\ref{defthm:HJS-CQR}.
\item \emph{Non-cyclic D type and E type singularities} are written as $[e_0;\,(r_1,q_1);\,(r_2,q_2);\,(r_3,q_3)]$, where $e_0$ is the central fork and the three semicolons separate the three branches; each pair $(r_i,q_i)$ in parentheses denotes a single cyclic quotient singularity $\tfrac{1}{r_i}(1,q_i)$, corresponding to a branch whose underlying HJS contracts to a cyclic quotient of that type. For instance, $[2;(2,1);(2,1);(31,26)]$ stands for the D-II type singularity whose third branch is the HJS $[2,2,2,2,2,6]$ (the unique HJS with $\det[e_1,\ldots,e_n]=31$ and $\det[e_2,\ldots,e_n]=26$; cf.\ Definition--Theorem~\ref{defthm:HJS-CQR}).
\end{itemize}

{\small
\renewcommand{\arraystretch}{1.05}
\begin{longtable}{|c|l|c|c|}
\caption{The 251 three-singularity baskets eliminated by Filter~\ref{fil:ai}. 
Each basket consists of three klt singularity germs (cyclic-quotient of A-type, or non-cyclic of D/E-type), 
with resolution configurations listed in column~2 in our standard bracket notation. 
Column~3 records the smallest pair $(a,b)$ of integers with $a,b\ge 2$, $a+b\le 500$, 
$P_a>0$, $P_b>0$ at which the inequality $P_{a+b}\ge P_a+P_b-1$ fails; 
column~4 records the corresponding hypothetical value of $K_X^2$. Baskets are listed 
in increasing order of $K_X^2$, and ties are broken by the lexicographic order on the underlying 
triples of resolution-configuration strings.}\label{table:251-baskets}\\
\hline
\textbf{No.} & \multicolumn{1}{c|}{\textbf{Singularity types}} & $(a,b)$ & $K_X^2$ \\
\hline\endfirsthead
\multicolumn{4}{c}%
{\tablename\ \thetable\ -- continued from previous page}\\
\hline
\textbf{No.} & \multicolumn{1}{c|}{\textbf{Singularity types}} & $(a,b)$ & $K_X^2$ \\
\hline\endhead
\hline\multicolumn{4}{r}{\textit{(continued on next page)}}\\\endfoot
\hline\endlastfoot
1 & $[6,2,7]$, $[5,2,2,3,5]$, $[2;(2,1);(2,1);(91,85)]$ & $(2,2)$ & $1/58362$ \\\hline
2 & $[2,5,7]$, $[4,3,2,6]$, $[2;(2,1);(2,1);(94,85)]$ & $(2,4)$ & $1/53253$ \\\hline
3 & $[2,5,2,2,3,4]$, $[2,3,2,2,2,2,2,7]$, $[2,2,4,2,2,2,4,2]$ & $(2,3)$ & $1/48503$ \\\hline
4 & $[2,2,2,4,7]$, $[2;(2,1);(2,1);(82,73)]$, $[2,5,2,2,2,2,2,2,4,3]$ & $(2,4)$ & $1/37323$ \\\hline
5 & $[3,2,3,5]$, $[3,4,2,2,3,4]$, $[2;(2,1);(2,1);(14,11)]$ & $(2,3)$ & $1/36465$ \\\hline
6 & $[5,2,6]$, $[2,3,3,2,2,2,2,4]$, $[3,3,2,2,2,2,2,2,2,2,6]$ & $(2,2)$ & $1/35035$ \\\hline
7 & $[4,3,2,6]$, $[2,3,2,2,2,2,5]$, $[3,3,2,2,2,2,2,2,2,2,6]$ & $(2,2)$ & $1/34435$ \\\hline
8 & $[2,5,5]$, $[2,2,3,3,5]$, $[2;(2,1);(2,1);(86,77)]$ & $(2,3)$ & $1/32121$ \\\hline
9 & $[2,4,6,2]$, $[3,3,2,2,5]$, $[2,4,2,2,2,2,2,2,2,2,4,2]$ & $(2,3)$ & $1/28105$ \\\hline
10 & $[2,4,6,2]$, $[3,3,2,2,5]$, $[2;(2,1);(2,1);(46,41)]$ & $(2,3)$ & $1/28105$ \\\hline
11 & $[2,5,2,2,2,7]$, $[2,3,2,2,6]$, $[6,2,2,2,2,2,2,2,2,2,2,2,2,2,6]$ & $(2,2)$ & $1/27115$ \\\hline
12 & $[2,3,6,2]$, $[2,4,4,2]$, $[2,2,2,2,2,3,2,3,2,2]$ & $(3,3)$ & $1/25755$ \\\hline
13 & $[2,3,6,2]$, $[2;(2,1);(2,1);(6,1)]$, $[2,2,2,2,2,3,2,3,2,2]$ & $(3,3)$ & $1/25755$ \\\hline
14 & $[3,4,7]$, $[2,3,2,2,4,3]$, $[2,2,2,2,2,3,2,2,2,2,2,5,2]$ & $(2,4)$ & $1/25123$ \\\hline
15 & $[4,3,2,6]$, $[2,2,3,4,2]$, $[2;(2,1);(2,1);(55,49)]$ & $(2,3)$ & $1/25026$ \\\hline
16 & $[2,2,3,2,2,7]$, $[4,3,2,2,2,2,6]$, $[3,3,2,2,2,2,2,2,2,2,6]$ & $(2,2)$ & $1/23735$ \\\hline
17 & $[3,3,4,3]$, $[2,3,2,3,3]$, $[2;(2,1);(2,1);(43,37)]$ & $(2,4)$ & $1/23226$ \\\hline
18 & $[3,2,3,4,3]$, $[3,4,2,2,4]$, $[2,2,3,2,2,3,2,2]$ & $(2,4)$ & $1/20826$ \\\hline
19 & $[3,2,3,4,3]$, $[3,4,2,2,4]$, $[2;(2,1);(2,1);(9,7)]$ & $(2,4)$ & $1/20826$ \\\hline
20 & $[2,5,2,2,3]$, $[4,3,2,2,2,2,5]$, $[2,2,4,2,2,4,2,2]$ & $(2,3)$ & $3/60473$ \\\hline
21 & $[3,2,3,5]$, $[2,5,2,2,2,5,2]$, $[2,2,2,2,2,2,2,2,4,4]$ & $(2,2)$ & $2/39655$ \\\hline
22 & $[5,3,6]$, $[2,2,3,2,2,2,2,2,7]$, $[4,2,2,2,2,2,2,2,2,2,4]$ & $(2,2)$ & $1/19671$ \\\hline
23 & $[5,3,2,6]$, $[2,2,2,2,2,3,4,2]$, $[2;(2,1);(2,1);(73,65)]$ & $(2,3)$ & $1/19592$ \\\hline
24 & $[4,3,2,2,7]$, $[2,2,5,2]$, $[2,4,2,2,2,2,2,2,2,2,4,2]$ & $(2,3)$ & $1/18055$ \\\hline
25 & $[4,3,2,2,7]$, $[2,2,5,2]$, $[2;(2,1);(2,1);(46,41)]$ & $(2,3)$ & $1/18055$ \\\hline
26 & $[2,3,8]$, $[4,2,2,2,2,2,2,7]$, $[3,2,3,2,2,2,2,2,2,2,2,3,4]$ & $(2,2)$ & $1/17955$ \\\hline
27 & $[6,2,2,2,2,2,2,7]$, $[2,2,2,7,2]$, $[2;(2,1);(2,1);(46,39)]$ & $(2,4)$ & $2/35581$ \\\hline
28 & $[2,5,2,4]$, $[2,2,2,3,2,7]$, $[2,2,4,2,2,2,2,2,3,2,3]$ & $(2,4)$ & $2/34265$ \\\hline
29 & $[2,5,2,2,5,2]$, $[2,3,2,2,2,2,2,7]$, $[2,4,2,2,2,3,2,3]$ & $(2,3)$ & $3/51373$ \\\hline
30 & $[2;(2,1);(2,1);(22,15)]$, $[2,3,2,2,2,2,2,7]$, $[2,4,2,2,2,3,2,3]$ & $(2,3)$ & $3/51373$ \\\hline
31 & $[3,3,4,3]$, $[3,2,3,2,2,2,6]$, $[2;(2,1);(2,1);(22,17)]$ & $(2,4)$ & $1/16985$ \\\hline
32 & $[3,3,4,3]$, $[3,2,3,2,2,2,6]$, $[3,3,2,2,2,2,2,2,2,2,6]$ & $(2,2)$ & $1/16985$ \\\hline
33 & $[2,2,2,4,4]$, $[3,3,2,2,2,4,3]$, $[2;(2,1);(2,1);(43,36)]$ & $(2,4)$ & $1/16716$ \\\hline
34 & $[6,2,2,2,2,6]$, $[2,5,2,2,2,2,7]$, $[3,2,2,2,2,2,2,2,2,2,4,3]$ & $(2,2)$ & $2/33205$ \\\hline
35 & $[2,3,2,2,7]$, $[2,5,2,2,3,4]$, $[2,2,2,2,2,2,2,4,2,3]$ & $(2,2)$ & $1/16261$ \\\hline
36 & $[3,2,3,5]$, $[2;(2,1);(2,1);(5,2)]$, $[3,4,2,2,2,2,2,3,4]$ & $(2,3)$ & $4/64185$ \\\hline
37 & $[2,5,2,2,3,4]$, $[3,2,2,2,6]$, $[2,2,3,2,2,2,4,2,2]$ & $(2,3)$ & $3/47705$ \\\hline
38 & $[3,3,5]$, $[2,2,2,4,2,3]$, $[2;(2,1);(2,1);(73,65)]$ & $(2,6)$ & $1/15688$ \\\hline
39 & $[4,3,2,6]$, $[2,2,4,2,2,4,2,2]$, $[2,4,2,2,2,2,2,3]$ & $(2,3)$ & $1/15617$ \\\hline
40 & $[2,4,4,2]$, $[2,2,3,2,2,5,2]$, $[2,2,2,2,4,2]$ & $(3,3)$ & $1/15255$ \\\hline
41 & $[2;(2,1);(2,1);(6,1)]$, $[2,2,3,2,2,5,2]$, $[2,2,2,2,4,2]$ & $(3,3)$ & $1/15255$ \\\hline
42 & $[4,7]$, $[2,2,3,2,2,5,2]$, $[2,4,2,2,2,2,2,2,2,2,4,2]$ & $(2,3)$ & $1/15255$ \\\hline
43 & $[4,7]$, $[2,2,3,2,2,5,2]$, $[2;(2,1);(2,1);(46,41)]$ & $(2,3)$ & $1/15255$ \\\hline
44 & $[5,2,3,5]$, $[2,2,4,4,2]$, $[2;(2,1);(2,1);(77,68)]$ & $(2,3)$ & $1/14544$ \\\hline
45 & $[2,4,4,2]$, $[3;(2,1);(3,1);(3,2)]$, $[2,3,2,2,3,3,2]$ & $(3,3)$ & $1/13905$ \\\hline
46 & $[2;(2,1);(2,1);(6,1)]$, $[3;(2,1);(3,1);(3,2)]$, $[2,3,2,2,3,3,2]$ & $(3,3)$ & $1/13905$ \\\hline
47 & $[2,3,2,3,6]$, $[2,2,4,4]$, $[2,2,4,2,2,2,2,2,3,2,3]$ & $(2,3)$ & $1/13727$ \\\hline
48 & $[2,3,2,3,6]$, $[2,4,3,2,3]$, $[2;(2,1);(2,1);(64,57)]$ & $(2,3)$ & $1/13727$ \\\hline
49 & $[3,2,8]$, $[2,2,2,5,2,3]$, $[2,4,2,2,2,2,2,2,2,2,4,2]$ & $(2,5)$ & $1/13505$ \\\hline
50 & $[3,2,8]$, $[2,2,2,5,2,3]$, $[2;(2,1);(2,1);(46,41)]$ & $(2,5)$ & $1/13505$ \\\hline
51 & $[2,2,2,7,2]$, $[6,2,2,2,2,2,2,2,7]$, $[2,2,4,2,2,2,2,2,2,5,2]$ & $(2,4)$ & $3/40411$ \\\hline
52 & $[3,4,2,2,5]$, $[3,2,2,3,4]$, $[2;(2,1);(2,1);(29,21)]$ & $(2,2)$ & $1/13224$ \\\hline
53 & $[4,2,2,6]$, $[2,2,2,2,2,6]$, $[2;(2,1);(2,1);(43,35)]$ & $(2,2)$ & $1/13144$ \\\hline
54 & $[6,2,2,2,2,2,6]$, $[2,2,2,3,3]$, $[2,2,4,2,2,3,3,2]$ & $(2,3)$ & $2/26105$ \\\hline
55 & $[2,5,2,2,2,6]$, $[2,3,3,2,4]$, $[2,2,4,2,2,2,2,2,2,5,2]$ & $(2,3)$ & $6/78029$ \\\hline
56 & $[2,3,2,2,7]$, $[3,4,2,2,2,2,6]$, $[3,3,2,2,2,2,2,2,2,2,2,6]$ & $(2,2)$ & $2/25645$ \\\hline
57 & $[3,2,4,3]$, $[3,4,2,2,2,2,6]$, $[2;(2,1);(2,1);(22,17)]$ & $(2,4)$ & $2/25645$ \\\hline
58 & $[3,2,4,3]$, $[3,4,2,2,2,2,6]$, $[3,3,2,2,2,2,2,2,2,2,6]$ & $(2,2)$ & $2/25645$ \\\hline
59 & $[3,4,2,2,5]$, $[2,3,3,2,2,2,6]$, $[2;(2,1);(2,1);(56,47)]$ & $(2,3)$ & $4/50373$ \\\hline
60 & $[3,3,7]$, $[2;(2,1);(2,1);(21,16)]$, $[2,3,2,2,2,2,2,2,2,5,2]$ & $(2,3)$ & $1/12455$ \\\hline
61 & $[2,7,3]$, $[3,2,2,2,2,3,6]$, $[6,2,2,2,2,2,2,2,2,2,2,2,2,2,6]$ & $(2,2)$ & $2/24605$ \\\hline
62 & $[4,3,2,5]$, $[2,5,2,2,2,5,2]$, $[2,2,2,2,2,2,2,3,3]$ & $(2,2)$ & $2/23779$ \\\hline
63 & $[2,5,2,3,3]$, $[2,4,2,2,2,4,2]$, $[2,3,2,2,2,2,2,2,7]$ & $(2,3)$ & $2/23735$ \\\hline
64 & $[2,3,3,2,7]$, $[4,2,3,2,5]$, $[2;(2,1);(2,1);(20,17)]$ & $(2,3)$ & $1/11859$ \\\hline
65 & $[3,3,4,3]$, $[2,4,2,2,2,4,2]$, $[2,4,2,2,2,3,2,3]$ & $(2,4)$ & $6/70705$ \\\hline
66 & $[2,2,4,5]$, $[2;(2,1);(2,1);(21,16)]$, $[3,2,2,3,2,2,3,3]$ & $(2,3)$ & $4/46295$ \\\hline
67 & $[4,5]$, $[4,3,2,2,2,2,2,3,4]$, $[2,4,2,2,2,2,2,2,2,5]$ & $(2,2)$ & $2/23009$ \\\hline
68 & $[2,2,3,3,5]$, $[2,3,3,2,5]$, $[2;(2,1);(2,1);(49,43)]$ & $(2,3)$ & $1/11454$ \\\hline
69 & $[4,3,2,2,2,7]$, $[2,2,2,7,2]$, $[2,4,2,2,2,2,2,2,2,2,4,2]$ & $(2,5)$ & $2/22885$ \\\hline
70 & $[4,3,2,2,2,7]$, $[2,2,2,7,2]$, $[2;(2,1);(2,1);(46,41)]$ & $(2,5)$ & $2/22885$ \\\hline
71 & $[4,3,2,2,2,7]$, $[2,2,5,2]$, $[2,2,2,3,2,2,2,2,3,3]$ & $(2,5)$ & $2/22885$ \\\hline
72 & $[3,4,2,6]$, $[2,2,2,3,3,2,3]$, $[2;(2,1);(2,1);(31,26)]$ & $(2,4)$ & $1/11330$ \\\hline
73 & $[5,3,7]$, $[2,2,2,2,2,2,3,2,5,2]$, $[2;(2,1);(2,1);(49,43)]$ & $(2,3)$ & $1/10974$ \\\hline
74 & $[6,2,2,2,2,6]$, $[2,2,2,3,3]$, $[2,2,4,2,2,2,3,3,2]$ & $(2,3)$ & $3/32545$ \\\hline
75 & $[3,3,2,5]$, $[2,5,2,2,3,4]$, $[2,2,2,2,3,2,2,3,2]$ & $(2,2)$ & $1/10773$ \\\hline
76 & $[2,2,8,2]$, $[2;(2,1);(2,1);(46,37)]$, $[2,2,2,2,2,2,3,2,3,2,2]$ & $(3,3)$ & $1/10701$ \\\hline
77 & $[2,2,8,2]$, $[3,4,2,2,5]$, $[2;(2,1);(2,1);(127,118)]$ & $(2,4)$ & $1/10701$ \\\hline
78 & $[3,2,3,6]$, $[2,2,2,3,2,4,3]$, $[2;(2,1);(2,1);(31,26)]$ & $(2,4)$ & $4/42545$ \\\hline
79 & $[4,3,6]$, $[2,4,2,2,2,2,5]$, $[4,2,2,2,2,2,2,2,2,2,4]$ & $(2,2)$ & $1/10509$ \\\hline
80 & $[6,2,7]$, $[2,2,3,2,2,2,2,3,4]$, $[3,3,2,2,2,2,2,2,2,2,2,6]$ & $(2,2)$ & $2/20945$ \\\hline
81 & $[2,4,4]$, $[2,3,3,2,2,2,7]$, $[2;(2,1);(2,1);(46,39)]$ & $(2,2)$ & $1/10465$ \\\hline
82 & $[4,2,2,6]$, $[2,3,3,2,2,2,5]$, $[3,3,2,2,2,2,2,2,2,2,6]$ & $(2,2)$ & $1/10335$ \\\hline
83 & $[7,2,2,2,2,2,2,2,7]$, $[3,2,2,4,3]$, $[3,2,3,2,2,2,2,3,3]$ & $(2,2)$ & $2/20553$ \\\hline
84 & $[6,2,2,2,2,6]$, $[3,2,2,2,2,6]$, $[2,3,2,2,2,2,2,2,2,5]$ & $(2,2)$ & $1/10165$ \\\hline
85 & $[2,2,2,7,2]$, $[6,2,2,2,2,2,2,2,2,2,7]$, $[2,2,4,2,2,2,2,5,2]$ & $(2,4)$ & $5/50071$ \\\hline
86 & $[2,6,2,5]$, $[3,6,3]$, $[4,3,2,2,2,2,2,2,2,2,2,2,2,2,2,2,2,2,7]$ & $(2,2)$ & $27/265220$ \\\hline
87 & $[6,2,2,2,2,2,6]$, $[2,2,2,2,6]$, $[2,2,2,4,2,2,3,2,2]$ & $(2,3)$ & $1/9815$ \\\hline
88 & $[6,2,2,2,2,6]$, $[2,4,2,2,3]$, $[2,2,2,4,2,2,3,2,2]$ & $(2,3)$ & $1/9815$ \\\hline
89 & $[4,3,2,2,6]$, $[2,3,2,2,2,5]$, $[3,3,2,2,2,2,2,2,2,2,6]$ & $(2,2)$ & $1/9735$ \\\hline
90 & $[6,2,2,7]$, $[2,2,2,6,2,2]$, $[2,2,4,2,2,2,2,2,2,2,3,2,3]$ & $(2,4)$ & $4/38885$ \\\hline
91 & $[6,2,2,7]$, $[2,2,2,6,2,2]$, $[2;(2,1);(2,1);(85,78)]$ & $(2,4)$ & $4/38885$ \\\hline
92 & $[2,6,5]$, $[3,4,2,7]$, $[4,2,2,2,2,2,2,2,2,2,2,2,2,2,2,2,4]$ & $(2,2)$ & $1/9699$ \\\hline
93 & $[2,4,2,3,5]$, $[2;(2,1);(2,1);(11,2)]$, $[2;(2,1);(2,1);(37,29)]$ & $(3,3)$ & $1/9576$ \\\hline
94 & $[2,4,2,3,5]$, $[2;(2,1);(2,1);(13,5)]$, $[2;(2,1);(2,1);(38,29)]$ & $(3,3)$ & $1/9576$ \\\hline
95 & $[2,4,2,3,5]$, $[2;(2,1);(2,1);(13,5)]$, $[2;(2,1);(3,1);(3,2)]$ & $(3,3)$ & $1/9576$ \\\hline
96 & $[2,4,2,3,5]$, $[2;(2,1);(2,1);(20,11)]$, $[2;(2,1);(2,1);(29,21)]$ & $(3,3)$ & $1/9576$ \\\hline
97 & $[2,4,2,3,5]$, $[2;(2,1);(2,1);(29,20)]$, $[2;(2,1);(2,1);(21,13)]$ & $(3,3)$ & $1/9576$ \\\hline
98 & $[2,4,2,3,5]$, $[3;(2,1);(2,1);(5,2)]$, $[2;(2,1);(2,1);(47,38)]$ & $(3,3)$ & $1/9576$ \\\hline
99 & $[2,4,2,3,5]$, $[6;(2,1);(2,1);(2,1)]$, $[2;(2,1);(2,1);(45,37)]$ & $(3,3)$ & $1/9576$ \\\hline
100 & $[3,2,7]$, $[2,4,2,3,5]$, $[2;(2,1);(2,1);(92,83)]$ & $(2,3)$ & $1/9576$ \\\hline
101 & $[5,6]$, $[2,4,2,4]$, $[2;(2,1);(2,1);(69,61)]$ & $(2,2)$ & $1/9512$ \\\hline
102 & $[2,4,2,5]$, $[2;(2,1);(2,1);(21,16)]$, $[2,3,3,2,2,2,2,4]$ & $(2,2)$ & $4/37895$ \\\hline
103 & $[3,4,2,2,2,7]$, $[2,3,3,3]$, $[2;(2,1);(2,1);(28,23)]$ & $(2,5)$ & $1/9265$ \\\hline
104 & $[4,3,2,2,2,2,7]$, $[2,2,5,2]$, $[2,2,2,3,2,2,2,3,3]$ & $(2,5)$ & $3/27715$ \\\hline
105 & $[2,4,2,2,5]$, $[2,4,2,2,4]$, $[2;(2,1);(2,1);(50,41)]$ & $(2,2)$ & $1/9198$ \\\hline
106 & $[5,6]$, $[2,4,2,2,2,2,3,4]$, $[4,2,2,2,2,2,2,2,2,2,4]$ & $(2,2)$ & $2/18183$ \\\hline
107 & $[2,3,2,2,7]$, $[2,4,2,2,2,2,4,3]$, $[2;(2,1);(2,1);(26,21)]$ & $(2,3)$ & $1/8855$ \\\hline
108 & $[3,2,4,3]$, $[2,4,2,2,2,4,2]$, $[2,4,2,2,2,2,4,3]$ & $(2,4)$ & $1/8855$ \\\hline
109 & $[4,2,3,6]$, $[3,2,2,4,3]$, $[4,2,2,2,2,2,2,2,2,2,4]$ & $(2,2)$ & $1/8835$ \\\hline
110 & $[3,4,2,2,7]$, $[2,2,5,2]$, $[2;(2,1);(2,1);(118,109)]$ & $(2,4)$ & $2/17595$ \\\hline
111 & $[2,3,3,6]$, $[3,4,2,2,2,2,2,5]$, $[2;(2,1);(2,1);(56,47)]$ & $(2,3)$ & $4/34821$ \\\hline
112 & $[2,3,3,3]$, $[2,2,2,3,5,2]$, $[4,3,2,2,2,2,2,3,4]$ & $(2,2)$ & $1/8687$ \\\hline
113 & $[2,4,2,4]$, $[2;(2,1);(2,1);(32,23)]$, $[3,2,2,2,2,4,3]$ & $(2,2)$ & $2/17343$ \\\hline
114 & $[2,2,2,7,2]$, $[2,5,2,2,2,5,2]$, $[2,2,2,2,2,3,3,2]$ & $(3,3)$ & $1/8533$ \\\hline
115 & $[2,5,2,2,5,2]$, $[2,2,5,2]$, $[2,2,2,2,2,3,3,2]$ & $(3,3)$ & $1/8533$ \\\hline
116 & $[2;(2,1);(2,1);(22,15)]$, $[2,2,5,2]$, $[2,2,2,2,2,3,3,2]$ & $(3,3)$ & $1/8533$ \\\hline
117 & $[6,2,2,2,2,2,6]$, $[2,4,3]$, $[2,2,2,2,3,2,2,4,2,2]$ & $(2,3)$ & $2/17005$ \\\hline
118 & $[2,5,6]$, $[2;(2,1);(2,1);(37,31)]$, $[2,3,2,2,2,2,2,3,3]$ & $(2,3)$ & $1/8502$ \\\hline
119 & $[5,2,7]$, $[2,2,3,3,4]$, $[2;(2,1);(2,1);(92,83)]$ & $(2,3)$ & $2/16965$ \\\hline
120 & $[6,2,2,2,2,2,6]$, $[2,3,2,2,6]$, $[2,5,2,2,2,2,2,2,2,2,2,2,2,7]$ & $(2,2)$ & $9/75835$ \\\hline
121 & $[6,6]$, $[2,2,2,3,3]$, $[2,2,4,2,2,2,2,2,2,2,3,3,2]$ & $(2,3)$ & $7/58305$ \\\hline
122 & $[2,2,9]$, $[3,2,2,6]$, $[2;(2,1);(2,1);(70,61)]$ & $(2,3)$ & $1/8325$ \\\hline
123 & $[2,2,2,2,3,7]$, $[2,3,3,3,4]$, $[2,3,2,2,2,2,3,2,3]$ & $(2,2)$ & $1/8241$ \\\hline
124 & $[2,2,2,3,6]$, $[2,3,3,3,4]$, $[2,3,2,2,2,2,3,2,3]$ & $(2,2)$ & $1/8241$ \\\hline
125 & $[2,2,3,5]$, $[2,3,3,3,4]$, $[2,3,2,2,2,2,3,2,3]$ & $(2,2)$ & $1/8241$ \\\hline
126 & $[2,2,4,3,4]$, $[2,3,3,3,4]$, $[2,3,2,2,2,2,3,2,3]$ & $(2,2)$ & $1/8241$ \\\hline
127 & $[2,3,4]$, $[2,3,3,3,4]$, $[2,3,2,2,2,2,3,2,3]$ & $(2,2)$ & $1/8241$ \\\hline
128 & $[2,4,3,3]$, $[2,3,3,3,4]$, $[2,3,2,2,2,2,3,2,3]$ & $(2,2)$ & $1/8241$ \\\hline
129 & $[2,5,3,2,3]$, $[2,3,3,3,4]$, $[2,3,2,2,2,2,3,2,3]$ & $(2,2)$ & $1/8241$ \\\hline
130 & $[2]$, $[2,3,3,3,4]$, $[2,3,2,2,2,2,3,2,3]$ & $(2,2)$ & $1/8241$ \\\hline
131 & $[3,3]$, $[2,3,3,3,4]$, $[2,3,2,2,2,2,3,2,3]$ & $(2,2)$ & $1/8241$ \\\hline
132 & $[2,2,2,2,3,2,7]$, $[2,3,2,2,2,2,2,7]$, $[2,4,2,2,2,2,5,2]$ & $(2,3)$ & $1/8159$ \\\hline
133 & $[2,2,2,3,2,6]$, $[2,3,2,2,2,2,2,7]$, $[2,4,2,2,2,2,5,2]$ & $(2,3)$ & $1/8159$ \\\hline
134 & $[2,2,3,2,5]$, $[2,3,2,2,2,2,2,7]$, $[2,4,2,2,2,2,5,2]$ & $(2,3)$ & $1/8159$ \\\hline
135 & $[2,2,4,2,3,4]$, $[2,3,2,2,2,2,2,7]$, $[2,4,2,2,2,2,5,2]$ & $(2,3)$ & $1/8159$ \\\hline
136 & $[2,2]$, $[2,3,2,2,2,2,2,7]$, $[2,4,2,2,2,2,5,2]$ & $(2,3)$ & $1/8159$ \\\hline
137 & $[2,3,2,4]$, $[2,3,2,2,2,2,2,7]$, $[2,4,2,2,2,2,5,2]$ & $(2,3)$ & $1/8159$ \\\hline
138 & $[2,4,2,3,3]$, $[2,3,2,2,2,2,2,7]$, $[2,4,2,2,2,2,5,2]$ & $(2,3)$ & $1/8159$ \\\hline
139 & $[2,5,2,3,2,3]$, $[2,3,2,2,2,2,2,7]$, $[2,4,2,2,2,2,5,2]$ & $(2,3)$ & $1/8159$ \\\hline
140 & $[3,2,3]$, $[2,3,2,2,2,2,2,7]$, $[2,4,2,2,2,2,5,2]$ & $(2,3)$ & $1/8159$ \\\hline
141 & $[3,4,2,2,4]$, $[3,2,3,2,2,4,3]$, $[2,2,3,3,2,2]$ & $(2,4)$ & $5/40762$ \\\hline
142 & $[4,3,2,2,2,2,2,7]$, $[2,2,5,2]$, $[2;(2,1);(2,1);(31,26)]$ & $(2,3)$ & $4/32545$ \\\hline
143 & $[2,4,6,2]$, $[5,2,2,2,6]$, $[2;(2,1);(2,1);(43,38)]$ & $(2,3)$ & $4/32485$ \\\hline
144 & $[2,3,2,3,6]$, $[2,3,3,3]$, $[2;(2,1);(2,1);(100,91)]$ & $(2,3)$ & $1/8109$ \\\hline
145 & $[6,2,2,2,2,6]$, $[2,2,2,5,2,2]$, $[2,3,2,2,2,3,2,3]$ & $(2,3)$ & $3/24295$ \\\hline
146 & $[10]$, $[2,2,2,3,3]$, $[2,2,4,2,2,2,2,2,2,2,2,3,3,2]$ & $(2,3)$ & $8/64745$ \\\hline
147 & $[5,3,7]$, $[2,5,2,2,2,4,3]$, $[2;(2,1);(2,1);(26,23)]$ & $(2,3)$ & $1/8091$ \\\hline
148 & $[2,5,5]$, $[2,3,3,4]$, $[2;(2,1);(2,1);(37,33)]$ & $(2,2)$ & $1/8084$ \\\hline
149 & $[2,2,2,7,2]$, $[3;(2,1);(2,1);(4,1)]$, $[6,2,2,2,2,2,2,2,2,2,2,2,2,7]$ & $(2,4)$ & $8/64561$ \\\hline
150 & $[6,2,2,2,6]$, $[2,3,2,2,6]$, $[2,5,2,2,2,2,2,2,2,2,2,2,2,2,2,7]$ & $(2,2)$ & $11/88015$ \\\hline
151 & $[4,7]$, $[2,2,2,2,4,4]$, $[2;(2,1);(2,1);(38,33)]$ & $(2,2)$ & $1/7965$ \\\hline
152 & $[6,2,7]$, $[2,3,2,2,3,3]$, $[2;(2,1);(2,1);(58,51)]$ & $(2,3)$ & $1/7952$ \\\hline
153 & $[2,5,2,2,3,4]$, $[2,2,4,2,2,5]$, $[2,2,2,2,2,2,2,5]$ & $(2,2)$ & $1/7931$ \\\hline
154 & $[10]$, $[2,3,2,4,2]$, $[2,2,4,2,2,2,2,2,2,2,2,2,3,3,2]$ & $(2,3)$ & $9/71185$ \\\hline
155 & $[2,2,5,2,4]$, $[3,2,2,2,2,3,6]$, $[2,3,2,2,2,3,2]$ & $(2,3)$ & $4/31521$ \\\hline
156 & $[2,2,5,2,4]$, $[3,2,2,2,2,3,6]$, $[2;(2,1);(2,1);(13,10)]$ & $(2,3)$ & $4/31521$ \\\hline
157 & $[4,2,2,7]$, $[2,3,4,2]$, $[2;(2,1);(2,1);(33,29)]$ & $(2,2)$ & $1/7812$ \\\hline
158 & $[6,2,2,2,2,2,2,7]$, $[2,5,2,2,2,3,4]$, $[3,2,2,2,3,2,2,2,4]$ & $(2,2)$ & $1/7735$ \\\hline
159 & $[3,3,4,3]$, $[2,2,2,2,3,7]$, $[2,2,2,3,2,2,2,3,2,3]$ & $(2,4)$ & $2/15247$ \\\hline
160 & $[3,3,4,3]$, $[2,2,2,3,6]$, $[2,2,2,3,2,2,2,3,2,3]$ & $(2,4)$ & $2/15247$ \\\hline
161 & $[3,3,4,3]$, $[2,2,3,5]$, $[2,2,2,3,2,2,2,3,2,3]$ & $(2,4)$ & $2/15247$ \\\hline
162 & $[3,3,4,3]$, $[2,2,4,3,4]$, $[2,2,2,3,2,2,2,3,2,3]$ & $(2,4)$ & $2/15247$ \\\hline
163 & $[3,3,4,3]$, $[2,3,4]$, $[2,2,2,3,2,2,2,3,2,3]$ & $(2,4)$ & $2/15247$ \\\hline
164 & $[3,3,4,3]$, $[2,4,3,3]$, $[2,2,2,3,2,2,2,3,2,3]$ & $(2,4)$ & $2/15247$ \\\hline
165 & $[3,3,4,3]$, $[2,5,3,2,3]$, $[2,2,2,3,2,2,2,3,2,3]$ & $(2,4)$ & $2/15247$ \\\hline
166 & $[3,3,4,3]$, $[2]$, $[2,2,2,3,2,2,2,3,2,3]$ & $(2,4)$ & $2/15247$ \\\hline
167 & $[3,3,4,3]$, $[3,3]$, $[2,2,2,3,2,2,2,3,2,3]$ & $(2,4)$ & $2/15247$ \\\hline
168 & $[2,6,7]$, $[3,3,2,2,3,4]$, $[3,2,3,2,2,2,2,2,2,2,2,2,2,2,2,2,7]$ & $(2,2)$ & $6/45625$ \\\hline
169 & $[2,5,2,2,5,2]$, $[2;(2,1);(2,1);(31,25)]$, $[3,2,2,3,2,2,2,5]$ & $(2,3)$ & $1/7602$ \\\hline
170 & $[2,7,2]$, $[3,2,2,3,2,2,2,5]$, $[2;(2,1);(2,1);(57,50)]$ & $(2,3)$ & $1/7602$ \\\hline
171 & $[2;(2,1);(2,1);(13,7)]$, $[2;(2,1);(2,1);(43,36)]$, $[3,2,2,3,2,2,2,5]$ & $(2,3)$ & $1/7602$ \\\hline
172 & $[2;(2,1);(2,1);(15,8)]$, $[3,2,2,3,2,2,2,5]$, $[2;(2,1);(2,1);(37,31)]$ & $(2,3)$ & $1/7602$ \\\hline
173 & $[2;(2,1);(2,1);(19,13)]$, $[2;(2,1);(2,1);(36,29)]$, $[3,2,2,3,2,2,2,5]$ & $(2,3)$ & $1/7602$ \\\hline
174 & $[2;(2,1);(2,1);(22,15)]$, $[2;(2,1);(2,1);(31,25)]$, $[3,2,2,3,2,2,2,5]$ & $(2,3)$ & $1/7602$ \\\hline
175 & $[2;(2,1);(2,1);(29,22)]$, $[2;(2,1);(2,1);(25,19)]$, $[3,2,2,3,2,2,2,5]$ & $(2,3)$ & $1/7602$ \\\hline
176 & $[2;(2,1);(2,1);(7,1)]$, $[3,2,2,3,2,2,2,5]$, $[2;(2,1);(2,1);(50,43)]$ & $(2,3)$ & $1/7602$ \\\hline
177 & $[2;(2,1);(2,1);(8,1)]$, $[3,2,2,3,2,2,2,5]$, $[2;(2,1);(2,1);(43,37)]$ & $(2,3)$ & $1/7602$ \\\hline
178 & $[3,2,2,6]$, $[2,2,3,3,2]$, $[2,3,3,2,2,2,2,3,4]$ & $(2,2)$ & $49/372331$ \\\hline
179 & $[6,2,2,2,2,6]$, $[2,2,2,2,6,2]$, $[2,2,4,2,2,2,2,2,3]$ & $(2,3)$ & $3/22795$ \\\hline
180 & $[6,6]$, $[2,3,2,2,6]$, $[2,5,2,2,2,2,2,2,2,2,2,2,2,2,2,2,2,2,7]$ & $(2,2)$ & $14/106285$ \\\hline
181 & $[7,2,2,2,2,2,2,2,7]$, $[3,2,3,2,2,4]$, $[2,2,4,2,2,2,2,2,2,6]$ & $(2,2)$ & $1/7565$ \\\hline
182 & $[2,3,6,2]$, $[2,2,2,2,2,7]$, $[2;(2,1);(2,1);(15,11)]$ & $(3,3)$ & $1/7548$ \\\hline
183 & $[4,2,2,6]$, $[2,4,4,3]$, $[2;(2,1);(2,1);(15,13)]$ & $(2,4)$ & $1/7526$ \\\hline
184 & $[3,5,2,4]$, $[2,2,2,2,2,2,8]$, $[2;(2,1);(2,1);(46,39)]$ & $(2,2)$ & $1/7525$ \\\hline
185 & $[2,10]$, $[2,3,2,2,2,3,3]$, $[2,4,2,2,2,2,2,2,2,2,4,2]$ & $(2,3)$ & $1/7505$ \\\hline
186 & $[2,10]$, $[2,3,2,2,2,3,3]$, $[2;(2,1);(2,1);(46,41)]$ & $(2,3)$ & $1/7505$ \\\hline
187 & $[2,2,2,4,6]$, $[3,2,3,3,3]$, $[2;(2,1);(2,1);(36,31)]$ & $(2,4)$ & $2/14985$ \\\hline
188 & $[3,2,2,7]$, $[3,3,2,2,2,5]$, $[2;(2,1);(2,1);(31,24)]$ & $(2,3)$ & $1/7469$ \\\hline
189 & $[2,2,7,2]$, $[2,2,2,7,2]$, $[6,2,2,2,2,2,2,2,2,2,2,2,2,2,2,7]$ & $(2,4)$ & $10/74221$ \\\hline
190 & $[2,2,4,4]$, $[3,3,2,4,3]$, $[2,2,2,3,2,2,2,3,3]$ & $(2,4)$ & $3/22015$ \\\hline
191 & $[4,2,5]$, $[2,4,5,2]$, $[2;(2,1);(2,1);(37,33)]$ & $(2,2)$ & $1/7316$ \\\hline
192 & $[2,4,2,2,6]$, $[2,2,4,2,2,4]$, $[2,3,2,2,2,2,3]$ & $(2,3)$ & $2/14615$ \\\hline
193 & $[6,2,2,2,2,2,6]$, $[2,2,3,2,2,5,2]$, $[2,3,2,2,3,3,2]$ & $(2,3)$ & $8/58195$ \\\hline
194 & $[5,3,6]$, $[2,3,3]$, $[2;(2,1);(2,1);(54,47)]$ & $(2,3)$ & $1/7189$ \\\hline
195 & $[5,3,6]$, $[4,2,2,2,2,2,2,5]$, $[2,2,2,2,2,2,2,2]$ & $(2,2)$ & $1/7189$ \\\hline
196 & $[5,3,6]$, $[4,2,2,2,2,2,2,5]$, $[2,2,2,2,3,2,2,2,2,2,2,2,7]$ & $(2,2)$ & $1/7189$ \\\hline
197 & $[5,3,6]$, $[4,2,2,2,2,2,2,5]$, $[2,2,2,3,2,2,2,2,2,2,2,6]$ & $(2,2)$ & $1/7189$ \\\hline
198 & $[5,3,6]$, $[4,2,2,2,2,2,2,5]$, $[2,2,3,2,2,2,2,2,2,2,5]$ & $(2,2)$ & $1/7189$ \\\hline
199 & $[5,3,6]$, $[4,2,2,2,2,2,2,5]$, $[2,2,4,2,2,2,2,2,2,2,3,4]$ & $(2,2)$ & $1/7189$ \\\hline
200 & $[5,3,6]$, $[4,2,2,2,2,2,2,5]$, $[2,3,2,2,2,2,2,2,2,4]$ & $(2,2)$ & $1/7189$ \\\hline
201 & $[5,3,6]$, $[4,2,2,2,2,2,2,5]$, $[2,4,2,2,2,2,2,2,2,3,3]$ & $(2,2)$ & $1/7189$ \\\hline
202 & $[5,3,6]$, $[4,2,2,2,2,2,2,5]$, $[2,5,2,2,2,2,2,2,2,3,2,3]$ & $(2,2)$ & $1/7189$ \\\hline
203 & $[5,3,6]$, $[4,2,2,2,2,2,2,5]$, $[2;(2,1);(2,1);(6,5)]$ & $(2,2)$ & $1/7189$ \\\hline
204 & $[5,3,6]$, $[4,2,2,2,2,2,2,5]$, $[2;(2,1);(3,2);(5,4)]$ & $(2,2)$ & $1/7189$ \\\hline
205 & $[5,3,6]$, $[4,2,2,2,2,2,2,5]$, $[3,2,2,2,2,2,2,2,3]$ & $(2,2)$ & $1/7189$ \\\hline
206 & $[4,5]$, $[4,3,2,2,2,3,4]$, $[2,4,2,2,2,2,2,2,2,2,2,5]$ & $(2,2)$ & $4/28329$ \\\hline
207 & $[2,6,7]$, $[2,2,4,2,2,3,3]$, $[2;(2,1);(2,1);(67,61)]$ & $(2,4)$ & $1/7050$ \\\hline
208 & $[7,2,2,2,2,2,2,2,7]$, $[2,2,3,4,2,2]$, $[2,2,2,2,3,2,4,3]$ & $(2,4)$ & $2/14091$ \\\hline
209 & $[2,2,5,2]$, $[4,3,2,2,2,2,2,2,2,7]$, $[2,4,2,2,2,4,2]$ & $(2,3)$ & $6/42205$ \\\hline
210 & $[6,2,2,2,2,6]$, $[2,4,3]$, $[2,2,2,2,3,2,2,2,4,2,2]$ & $(2,3)$ & $3/20995$ \\\hline
211 & $[4,2,3,5]$, $[2,4,2,2,2,4,2]$, $[2,2,2,2,3,3,2,3]$ & $(2,2)$ & $1/6955$ \\\hline
212 & $[2,2,5,7]$, $[2,3,2,5]$, $[2;(2,1);(2,1);(95,86)]$ & $(2,3)$ & $1/6930$ \\\hline
213 & $[2,5,2,2,3,3]$, $[2,2,2,3,4,3]$, $[2,5,2,2,2,2,2,3,4]$ & $(2,4)$ & $1/6902$ \\\hline
214 & $[2,4,7]$, $[2;(2,1);(2,1);(73,64)]$, $[2,2,3,2,2,2,2,2,2,4,3]$ & $(2,4)$ & $4/27495$ \\\hline
215 & $[4,2,3,5]$, $[2,2,4,2,2,4,2,2]$, $[2,2,2,2,3,2,2,2,6]$ & $(2,3)$ & $2/13741$ \\\hline
216 & $[3,5,4]$, $[2,2,2,2,2,2,2,3,7]$, $[2,2,4,2,2,2,2,2,2,5,2]$ & $(2,4)$ & $2/13727$ \\\hline
217 & $[2,7,3]$, $[4,2,2,3,4]$, $[2;(2,1);(2,1);(52,43)]$ & $(2,2)$ & $2/13653$ \\\hline
218 & $[2,6,2,2,3]$, $[2,2,3,3,2]$, $[4,3,2,2,2,2,2,3,4]$ & $(2,2)$ & $2/13601$ \\\hline
219 & $[5,2,7]$, $[2,6,2,2,3]$, $[2;(2,1);(2,1);(72,65)]$ & $(2,3)$ & $2/13601$ \\\hline
220 & $[2,4,6,2]$, $[2,4,2,2,2,5]$, $[2,2,2,2,2,2,2,2]$ & $(2,3)$ & $1/6789$ \\\hline
221 & $[2,4,6,2]$, $[2,4,2,2,2,5]$, $[2,2,2,2,3,2,2,2,2,2,2,2,7]$ & $(2,3)$ & $1/6789$ \\\hline
222 & $[2,4,6,2]$, $[2,4,2,2,2,5]$, $[2,2,2,3,2,2,2,2,2,2,2,6]$ & $(2,3)$ & $1/6789$ \\\hline
223 & $[2,4,6,2]$, $[2,4,2,2,2,5]$, $[2,2,3,2,2,2,2,2,2,2,5]$ & $(2,3)$ & $1/6789$ \\\hline
224 & $[2,4,6,2]$, $[2,4,2,2,2,5]$, $[2,2,4,2,2,2,2,2,2,2,3,4]$ & $(2,3)$ & $1/6789$ \\\hline
225 & $[2,4,6,2]$, $[2,4,2,2,2,5]$, $[2,3,2,2,2,2,2,2,2,4]$ & $(2,3)$ & $1/6789$ \\\hline
226 & $[2,4,6,2]$, $[2,4,2,2,2,5]$, $[2,4,2,2,2,2,2,2,2,3,3]$ & $(2,3)$ & $1/6789$ \\\hline
227 & $[2,4,6,2]$, $[2,4,2,2,2,5]$, $[2,5,2,2,2,2,2,2,2,3,2,3]$ & $(2,3)$ & $1/6789$ \\\hline
228 & $[2,4,6,2]$, $[2,4,2,2,2,5]$, $[2;(2,1);(2,1);(6,5)]$ & $(2,3)$ & $1/6789$ \\\hline
229 & $[2,4,6,2]$, $[2,4,2,2,2,5]$, $[2;(2,1);(3,2);(5,4)]$ & $(2,3)$ & $1/6789$ \\\hline
230 & $[2,4,6,2]$, $[2,4,2,2,2,5]$, $[3,2,2,2,2,2,2,2,3]$ & $(2,3)$ & $1/6789$ \\\hline
231 & $[2,4,2,2,6]$, $[2,3,2,3]$, $[2,2,4,2,2,2,2,2,4]$ & $(2,3)$ & $1/6745$ \\\hline
232 & $[3,4,2,2,4]$, $[2,3,2]$, $[3,2,3,2,2,2,2,4,3]$ & $(2,4)$ & $9/60698$ \\\hline
233 & $[2,2,5,2]$, $[2,2,2,4,3]$, $[4,3,2,2,2,2,2,2,2,2,7]$ & $(2,5)$ & $7/47035$ \\\hline
234 & $[5,3,7]$, $[2,5,2,2,2,2,2,2,4,3]$, $[2,2,2,2,3,3,2,2,2,2]$ & $(2,6)$ & $2/13299$ \\\hline
235 & $[5,3,7]$, $[2,5,2,2,2,2,2,2,4,3]$, $[4,2,2,2,2,2,2,2,2,2,4]$ & $(2,2)$ & $2/13299$ \\\hline
236 & $[3,2,3,2,2,6]$, $[2,2,2,3,4,3]$, $[2;(2,1);(2,1);(26,21)]$ & $(2,4)$ & $3/19865$ \\\hline
237 & $[6,2,2,2,2,2,6]$, $[3;(2,1);(3,1);(3,2)]$, $[2,2,2,2,3,2,2,5]$ & $(2,3)$ & $1/6615$ \\\hline
238 & $[2,5,2,2,3,4]$, $[2,2,4,4,2]$, $[2,3,2,2,2,2,2,2,2,2,7]$ & $(2,3)$ & $1/6608$ \\\hline
239 & $[2,3,8]$, $[2,4,2,7]$, $[2;(2,1);(2,1);(88,79)]$ & $(2,3)$ & $2/13167$ \\\hline
240 & $[2,3,2,7]$, $[2,5,2,4]$, $[2,2,4,2,2,2,2,2,2,2,2,2,5,2]$ & $(2,3)$ & $1/6545$ \\\hline
241 & $[2,5,2,4]$, $[2,3,3,2,2,6]$, $[2;(2,1);(2,1);(46,39)]$ & $(2,2)$ & $1/6545$ \\\hline
242 & $[2,5,2,4]$, $[3,2,4]$, $[2;(2,1);(2,1);(46,39)]$ & $(2,2)$ & $1/6545$ \\\hline
243 & $[2,2,2,7,2]$, $[2;(2,1);(2,1);(16,11)]$, $[4,3,2,2,2,2,2,2,2,2,2,7]$ & $(2,5)$ & $8/51865$ \\\hline
244 & $[2,2,5,2]$, $[2,2,4]$, $[4,3,2,2,2,2,2,2,2,2,2,7]$ & $(2,5)$ & $8/51865$ \\\hline
245 & $[2,6,6]$, $[2,4,2,2,2,3,2,3]$, $[2;(2,1);(2,1);(68,59)]$ & $(2,3)$ & $1/6444$ \\\hline
246 & $[3,4,2,2,4]$, $[5,2,2,2,2,5]$, $[3,2,3,2,2,2,2,2,4,3]$ & $(2,2)$ & $11/70666$ \\\hline
247 & $[2,2,5,2]$, $[2,2,2,2,2,6]$, $[2;(2,1);(2,1);(38,29)]$ & $(3,3)$ & $1/6417$ \\\hline
248 & $[11]$, $[2,2,4,2,2,2,2,4]$, $[3,2,3,2,2,2,2,2,2,2,2,2,2,2,2,7]$ & $(2,2)$ & $1/6413$ \\\hline
249 & $[2,3,6]$, $[2,2,3,3,6]$, $[2;(2,1);(2,1);(86,77)]$ & $(2,3)$ & $1/6363$ \\\hline
250 & $[4,2,5]$, $[2,3,5,2]$, $[2,4,2,2,2,2,2,2,2,2,4,2]$ & $(2,2)$ & $1/6355$ \\\hline
251 & $[4,2,5]$, $[2,3,5,2]$, $[2;(2,1);(2,1);(46,41)]$ & $(2,2)$ & $1/6355$ \\\hline
\end{longtable}
\renewcommand{\arraystretch}{1.0}
}

\end{document}